\documentclass[11pt]{amsart}
\usepackage{}
\usepackage{amssymb}

\usepackage{amsmath}
\usepackage{txfonts}

\input xy
\xyoption{all}

\usepackage{amsfonts}
\usepackage{mathrsfs}

\usepackage{latexsym}
\usepackage{graphicx}
\usepackage{amscd,amssymb,amsmath,amsbsy,amsthm}
\usepackage[all]{xy}
\usepackage[colorlinks,plainpages,backref,urlcolor=blue]{hyperref}
\usepackage{verbatim}

\usepackage{tikz}
\usetikzlibrary{arrows,calc}
\usetikzlibrary{shapes.geometric}
\tikzset{
>=stealth',
help lines/.style={dashed, thick},
axis/.style={<->},
important line/.style={thick},
connection/.style={thick, dotted},
}

\newcommand{\nc}{\newcommand}
\nc{\rnc}{\renewcommand}
\nc{\bb}[1]{{\mathbb #1}}
\nc{\bbA}{\bb{A}}\nc{\bbB}{\bb{B}}\nc{\bbC}{\bb{C}}\nc{\bbD}{\bb{D}}
\nc{\bbE}{\bb{E}}\nc{\bbF}{\bb{F}}\nc{\bbG}{\bb{G}}\nc{\bbH}{\bb{H}}
\nc{\bbI}{\bb{I}}\nc{\bbJ}{\bb{J}}\nc{\bbK}{\bb{K}}\nc{\bbL}{\bb{L}}
\nc{\bbM}{\bb{M}}\nc{\bbN}{\bb{N}}\nc{\bbO}{\bb{O}}\nc{\bbP}{\bb{P}}
\nc{\bbQ}{\bb{Q}}\nc{\bbR}{\bb{R}}\nc{\bbS}{\bb{S}}\nc{\bbT}{\bb{T}}
\nc{\bbU}{\bb{U}}\nc{\bbV}{\bb{V}}\nc{\bbW}{\bb{W}}\nc{\bbX}{\bb{X}}
\nc{\bbY}{\bb{Y}}\nc{\bbZ}{\bb{Z}}
\nc{\mbf}[1]{{\mathbf #1}}
\nc{\bfA}{\mbf{A}}\nc{\bfB}{\mbf{B}}\nc{\bfC}{\mbf{C}}\nc{\bfD}{\mbf{D}}
\nc{\bfE}{\mbf{E}}\nc{\bfF}{\mbf{F}}\nc{\bfG}{\mbf{G}}\nc{\bfH}{\mbf{H}}
\nc{\bfI}{\mbf{I}}\nc{\bfJ}{\mbf{J}}\nc{\bfK}{\mbf{K}}\nc{\bfL}{\mbf{L}}
\nc{\bfM}{\mbf{M}}\nc{\bfN}{\mbf{N}}\nc{\bfO}{\mbf{O}}\nc{\bfP}{\mbf{P}}
\nc{\bfQ}{\mbf{Q}}\nc{\bfR}{\mbf{R}}\nc{\bfS}{\mbf{S}}\nc{\bfT}{\mbf{T}}
\nc{\bfU}{\mbf{U}}\nc{\bfV}{\mbf{V}}\nc{\bfW}{\mbf{W}}\nc{\bfX}{\mbf{X}}
\nc{\bfY}{\mbf{Y}}\nc{\bfZ}{\mbf{Z}}
\nc{\bfa}{\mbf{a}}\nc{\bfb}{\mbf{b}}\nc{\bfc}{\mbf{c}}\nc{\bfd}{\mbf{d}}
\nc{\bfe}{\mbf{e}}\nc{\bff}{\mbf{f}}\nc{\bfg}{\mbf{g}}\nc{\bfh}{\mbf{h}}
\nc{\bfi}{\mbf{i}}\nc{\bfj}{\mbf{j}}\nc{\bfk}{\mbf{k}}\nc{\bfl}{\mbf{l}}
\nc{\bfm}{\mbf{m}}\nc{\bfn}{\mbf{n}}\nc{\bfo}{\mbf{o}}\nc{\bfp}{\mbf{p}}
\nc{\bfq}{\mbf{q}}\nc{\bfr}{\mbf{r}}\nc{\bfs}{\mbf{s}}\nc{\bft}{\mbf{t}}
\nc{\bfu}{\mbf{u}}\nc{\bfv}{\mbf{v}}\nc{\bfw}{\mbf{w}}\nc{\bfx}{\mbf{x}}
\nc{\bfy}{\mbf{y}}\nc{\bfz}{\mbf{z}}

\newcommand{\op}{\text{op}}

\nc{\mcal}[1]{{\mathcal #1}}
\nc{\calA}{\mcal{A}}\nc{\calB}{\mcal{B}}\nc{\calC}{\mcal{C}}\nc{\calD}{\mcal{D}}
\nc{\calE}{\mcal{E}} \nc{\calF}{\mcal{F}}\nc{\calG}{\mcal{G}}\nc{\calH}{\mcal{H}}
\nc{\calI}{\mcal{I}}\nc{\calJ}{\mcal{J}}\nc{\calK}{\mcal{K}}\nc{\calL}{\mcal{L}}
\nc{\calM}{\mcal{M}}\nc{\calN}{\mcal{N}}\nc{\calO}{\mcal{O}}\nc{\calP}{\mcal{P}}
\nc{\calQ}{\mcal{Q}}\nc{\calR}{\mcal{R}}\nc{\calS}{\mcal{S}}\nc{\calT}{\mcal{T}}
\nc{\calU}{\mcal{U}}\nc{\calV}{\mcal{V}}\nc{\calW}{\mcal{W}}\nc{\calX}{\mcal{X}}
\nc{\calY}{\mcal{Y}}\nc{\calZ}{\mcal{Z}}
\nc{\fA}{\frak{A}}\nc{\fB}{\frak{B}}\nc{\fC}{\frak{C}} \nc{\fD}{\frak{D}}
\nc{\fE}{\frak{E}}\nc{\fF}{\frak{F}}\nc{\fG}{\frak{G}}\nc{\fH}{\frak{H}}
\nc{\fI}{\frak{I}}\nc{\fJ}{\frak{J}}\nc{\fK}{\frak{K}}\nc{\fL}{\frak{L}}
\nc{\fM}{\frak{M}}\nc{\fN}{\frak{N}}\nc{\fO}{\frak{O}}\nc{\fP}{\frak{P}}
\nc{\fQ}{\frak{Q}}\nc{\fR}{\frak{R}}\nc{\fS}{\frak{S}}\nc{\fT}{\frak{T}}
\nc{\fU}{\frak{U}}\nc{\fV}{\frak{V}}\nc{\fW}{\frak{W}}\nc{\fX}{\frak{X}}
\nc{\fY}{\frak{Y}}\nc{\fZ}{\frak{Z}}
\nc{\fa}{\frak{a}}\nc{\fb}{\frak{b}}\nc{\fc}{\frak{c}} \nc{\fd}{\frak{d}}
\nc{\fe}{\frak{e}}\nc{\fFf}{\frak{f}}\nc{\fg}{\frak{g}}\nc{\fh}{\frak{h}}
\nc{\fri}{\frak{i}}\nc{\fj}{\frak{j}}\nc{\fk}{\frak{k}}\nc{\fl}{\frak{l}}
\nc{\fm}{\frak{m}}\nc{\fn}{\frak{n}}\nc{\fo}{\frak{o}}\nc{\fp}{\frak{p}}
\nc{\fq}{\frak{q}}\nc{\fr}{\frak{r}}\nc{\fs}{\frak{s}}\nc{\ft}{\frak{t}}
\nc{\fu}{\frak{u}}\nc{\fv}{\frak{v}}\nc{\fw}{\frak{w}}\nc{\fx}{\frak{x}}
\nc{\fy}{\frak{y}}\nc{\fz}{\frak{z}}

\newtheorem{theorem}{Theorem}[section]
\newtheorem{lemma}[theorem]{Lemma}
\newtheorem{corollary}[theorem]{Corollary}
\newtheorem{prop}[theorem]{Proposition}

\newtheorem{assumption}{Assumption}[section]

\theoremstyle{definition}
\newtheorem{definition}[theorem]{Definition}
\newtheorem{example}[theorem]{Example}
\newtheorem{remark}[theorem]{Remark}

\newtheorem{conj}[theorem]{Conjecture}

\newtheorem{thm}{Theorem}

\DeclareMathOperator{\im}{im} \DeclareMathOperator{\coker}{coker}
\DeclareMathOperator{\codim}{codim} \DeclareMathOperator{\id}{id}
 
\DeclareMathOperator{\ch}{ch}

 \DeclareMathOperator{\GL}{GL}
\DeclareMathOperator{\Hom}{{Hom}} 

\DeclareMathOperator{\sHom}{{\mathscr{H}om}}

\DeclareMathOperator{\loc}{{loc}}
\DeclareMathOperator{\proj}{pr} 
\DeclareMathOperator{\Lie}{Lie}
\DeclareMathOperator{\Pic}{Pic}
\DeclareMathOperator{\Spec}{{Spec}} 
\DeclareMathOperator{\triv}{triv}

\DeclareMathOperator{\cyc}{cyc}
 \DeclareMathOperator{\End}{End}

\DeclareMathOperator{\Gm}{\bbG_m}

\DeclareMathOperator{\Att}{Att}
\DeclareMathOperator{\Dr}{Dr}

\DeclareMathOperator{\Fr}{\mathrm{Fr}}

\DeclareMathOperator{\ev}{ev}
\DeclareMathOperator{\fin}{fin}

\DeclareMathOperator{\ad}{ad}
\DeclareMathOperator{\res}{res}
\DeclareMathOperator{\Gr}{Gr}

\DeclareMathOperator{\inc}{in}
\DeclareMathOperator{\out}{out}

\DeclareMathOperator{\red}{red}
\DeclareMathOperator{\even}{even}

\newcommand{\Omit}[1]{}

\DeclareMathOperator{\Stab}{Stab}

\DeclareMathOperator{\Rep}{Rep}
\DeclareMathOperator{\Res}{Res}

\newcommand{\surj}{\twoheadrightarrow}
\newcommand{\inj}{\hookrightarrow}

\newcommand{\pt}{\text{pt}}

\newcommand{\Z}{\bbZ}
\newcommand{\C}{\bbC}

\newcommand{\N}{\bbN}

\DeclareMathOperator{\MU}{MU}

\newcount\cols
{\catcode`,=\active\catcode`|=\active
 \gdef\Young(#1){\hbox{$\vcenter
 {\mathcode`,="8000\mathcode`|="8000
  \def,{\global\advance\cols by 1 &}%
  \def|{\cr
        \multispan{\the\cols}\hrulefill\cr
        &\global\cols=2 }%
  \offinterlineskip\everycr{}\tabskip=0pt
  \dimen0=\ht\strutbox \advance\dimen0 by \dp\strutbox
  \halign
   {\vrule height \ht\strutbox depth \dp\strutbox##
    &&\hbox to \dimen0{\hss$##$\hss}\vrule\cr
    \noalign{\hrule}&\global\cols=2 #1\crcr
    \multispan{\the\cols}\hrulefill\cr%
   }
 }$}}
}

\title[
Frobenii on Morava $E$-theoretical quantum groups
]{Frobenii on Morava $E$-theoretical quantum groups}

\author[Y.~Yang]{Yaping~Yang}
\address{The University of Melbourne,
	School of Mathematics and Statistics,
	813 Swanston Street, Parkville VIC 3010,
	Australia}
\email{yaping.yang1@unimelb.edu.au}

\author[G.~Zhao]{Gufang~Zhao}
\address{The University of Melbourne,
	School of Mathematics and Statistics,
	813 Swanston Street, Parkville VIC 3010,
	Australia}
\email{gufangz@unimelb.edu.au}

\dedicatory{To Leo, with love.}

\date{\today}

\begin{document}
\maketitle
\begin{abstract}
In this paper, we study a family of new quantum groups labelled by a prime number $p$ and a natural number $n$ constructed using the Morava $E$-theories. 
We define the quantum Frobenius homomorphisms among these quantum groups. This is a geometric generalization of Lusztig's quantum Frobenius from the quantum groups at a root of unity to the enveloping algebras. The main ingredient in constructing these Frobenii is the transchromatic character map of Hopkins, Kuhn, Ravenal, and Stapleton. As an application, we prove a Steinberg-type formula for irreducible representations of these quantum groups.  Consequently, we prove that, in type $A$ the characters of certain irreducible representations of these quantum groups satisfy the formulas introduced by Lusztig in 2015.
\end{abstract}

\tableofcontents
\section{Introduction}
It is well-known that the representation theory of Lusztig quantum groups at a root of unity \cite{Lu2} has many similarities to that of algebraic groups over a field of positive characteristic. Their characters of irreducible modules coincide in certain cases \cite{AJS}. In both settings, the Frobenius morphism is responsible to many interesting features of their representation theories. There is, however, a major difference between these two. The Frobenius morphism of the algebraic group, having the same domain and target, can be iterated. On the other hand, the Frobenius morphism of the quantum group, known as the quantum Frobenius, goes from Lusztig quantum group at a root of unity to the enveloping algebra of the Lie algebra, and hence can not be iterated. 
The present paper can be understood as finding an analogue of iteration of the quantum Frobenius. Roughly speaking, we consider a family of (relatively) new quantum groups labelled by a prime number $p$ and an integer $n\in\N$ (the {\em level}) constructed using the Morava $E$-theories. We construct an algebra homomorphism from the level $n$ quantum group at a $p$th root of unity to the level $n-t$ quantum group, for any $0<t\leq n$. (Here a root of unity means a $p$-torsion point in the formal group of the Morava $E$-theory.) On the level of representations, the homomorphism scales the weights by $p^t$, and hence has the feature expected of a Frobenius morphism.

These Morava $E$-theoretical quantum groups (or rather a variant of them) are first considered in \cite[Section 4]{YZ1}, and are speculated to be related to Lusztig's new character formulas in 2015 with some pieces of evidence given in \cite[Section 6]{YZ1}. The construction in \cite{YZ1} uses the approach of the cohomological Hall algebras, which is helpful in doing explicit calculations in the quantum groups. The present paper is made to be self-contained. We focus on the realization of these quantum groups in terms of the Morava $E$-theoretical convolution algebras of the Nakajima quiver varieties. 

The aforementioned Lusztig's character formulas \cite{Lusz} are obtained based on iteration of Frobenius of the algebraic group in positive characteristic, and a character formula from \cite{Lu2} for irreducible representations with ``reduced highest weights". The latter are known to 
describe irreducible representations of quantum group at a root of unity, but not those of algebraic groups. 
Therefore, to find the meaning of the new character formulas from \cite{Lusz}, one is naturally led to ``iteration of quantum Frobenius", which is the main topic of the present paper. Using the Frobenii on Morava $E$-theoretically quantum group,  we  address this meaning of Lusztig's new character formulas. 

\subsection{Overview}
Let $p$ be a prime number. Let $E_n^*$ be the integral Morava $E$-theory studied by Hopkins and Miller. 
Let $Q$ be a finite Dynkin quiver, with the set of vertices $I$. 
Labelling the simple roots $\alpha_i$ of the corresponding Lie algebra $\fg$ by $i\in I$, 
each dominant weight of $\fg$ is given by $w \in \mathbb{N}^I$.
Let $\fM(w)$ be the Nakajima quiver variety with framing $w$ (\S~\ref{subsec:quiver variety}), 
let $Z(w)$ be the Steinberg variety, which is a Lagrangian subvariety in $\fM(w)\times\fM(w)$. All these varieties are endowed with $\GL_w\times \C^*$-actions. Applying the equivariant $E_n^*$-theory, $E^*_{n, \GL_w\times \C^*}(Z(w))$ has an algebra structure by convolution (\S~\ref{subsec:conv}). 

Let $\gamma: \Lambda_n:=(\Z/p)^n\to \GL_w\times \C^*$ be a group homomorphism. Let $\mathbb{F}$ be a field of  characteristic zero satisfying  certain properties spelled out explicitly in \S\ref{sec:Frob phen}. In particular, $\mathbb{F}$ is an $E^*_{n, \Lambda_n}(\pt)$-field, and should be thought as a field containing cyclotomic integers of $E^0_n(\pt)$. We define the Morava-$E$-theoretical quantum group to be the subalgebra of the convolution algebra (Definition \ref{def:U})
\[
U_{w}^n(\gamma, \mathbb{F})\subset \mathbb{F}\otimes_{E^*_{n, \Lambda_n}(\pt)}E^*_{n, \GL_w\times \C^*}(Z(w))
\]
generated by certain tautological classes. 
Specializations of the equivariant parameters are to be understood as restricting to a finite subgroup of $\GL_w\times \C^*$.
For example,  the quantum parameter $q$ (coming from the extra $\C^*$-equivariance) lies in $E_n^*(BS^1)$, which is the formal group of the $E_n^*$-theory. We then specialize $q$ to be $\epsilon$, \textit{a root of unity}, which in this setting means a $p$-torsion point on the formal group  $E_n^*(B\Z/p)$. 
Note that the set of all the $p$-torsion points, roughly speaking, has cardinality $p^n$ (see Section 5 for details). When $n=1$, this is the usual notion of the $p$-th roots of $1$ in $\C^*$. 
 The main focus of the paper is the study of the quantum groups $U_{w}^n(\gamma, \mathbb{F})$ and their representations. These quantum groups should be thought of  as $n$-loop analogues of Lusztig's quantum groups, where each loop is specialized at a root of unity. This heuristic is to be elaborated on in \S~\ref{subsec:heuristic}.

For any positive integer $t$ such that $0< t\leq n$, assume $w$ is divisible by $p^t$, and assume further $\gamma$ is of the form $\gamma=\alpha \oplus \gamma'$, where 
$\alpha: \Lambda_t\to \GL_w\times \C^*$ is the group homomorphism defined in \S\ref{subsec:alpha} and 
$\gamma': \Lambda_{n-t}\to \GL_{w/p^t}\times \C^* \xrightarrow{\Delta}  \GL_{w}\times \C^*$ is a group homomorphism. 
We construct the quantum Frobenius map
\[
\Fr_{n, n-t}: U_{w}^n(\gamma, \mathbb{F})\to U_{w/p^t}^{n-t}(\gamma', \mathbb{F}), 
\]
as an algebra homomorphism (Theorem \ref{thm:algebra hom}). 
Geometrically, it is obtained by modifying the transchromatic character map of Hopkins, Kuhn, Ravenal \cite{HKR}, and Stapleton \cite{St}, which starts from the Morava $E_n^*$-theory and lands in a cohomology theory of chromatic level $n-t$.

% The algebra $U_{w}^n(\gamma, \mathbb{F})$ depends on the weight $w$. 
% Let $\gamma=\gamma_1\times\gamma_2$ with $\gamma_i:\Lambda_n\to \GL_{w_i}\times\bbC^*$ and $w=w_1+w_2$, together with a decomposition $W\cong W_1\oplus W_2$ as $I$-graded vector spaces. 
% We construct a coproduct of the form (Proposition \ref{prop:coproduct}). 
% \[ 
% \Delta_{w_1, w_2}: U_w(\gamma,\mathbb{F}) \to U_{w_1}(\gamma_1,\mathbb{F})\otimes U_{w_2}(\gamma_2,\mathbb{F}).\] 
% The map $\Delta$ is co-associative in the obvious sense, and is an algebra homomorphism.
% Moreover, assume $p^t|w_1$ and $p^t|w_2$, then $\Delta_{w_1,w_2}$ commutes with the Frobenii. 

One constructs the irreducible representations $L^n_w(\gamma, \mathbb{F})$ of $U_{w}^n(\gamma, \mathbb{F})$ 
by taking the irreducible quotients of the Weyl modules obtained from the quiver varieties (see  \S\ref{subsec:irreducible}) similar as \cite{Nak}. 
If $p^t|w$, denote by $\Fr_{n, n-t}^*L^{n-t}_{w/p^t}(\gamma', \mathbb{F})$ the representation of $U_{w}^n(\gamma, \mathbb{F})$ obtained by pulling back the irreducible module $L^{n-t}_{w/p^t}(\gamma', \mathbb{F})$ of $U_{w/p^t}^{n-t}(\gamma', \mathbb{F})$ along $\Fr_{n, n-t}$. 
% \[
% \mathbb{F}\otimes _{E^*_{n,\Lambda_n}}E^*_{n, \Lambda_n}( \mathfrak{M}(w)) \surj L^n_w(\gamma, \mathbb{F})
% \]
We prove the following Steinberg tensor product formula of the irreducible representations. Let $X^+=\N^I$ be the set of dominant weights and let \[
X^+_{\red}:=\{ w\in X^+\mid (\alpha_i^\vee, w) <p, \forall i\in I \}. 
\]
\begin{thm}
(Theorem \ref{thm:Steinberg tensor}) 
Let $w=w'+pw''$, where $w'\in X^+_{\red}$ and $w''\in X^+$. Let $\gamma:\Lambda_n\to \GL_{w}\times\bbC^*$ be the group homomorphism $\gamma=\gamma'\times (\alpha\oplus \gamma'')$ where $\gamma':\Lambda_n\to\GL_{w'}\times\bbC^*$, $\gamma'':\Lambda_{n-1}\to\GL_{w''}\times\bbC^*$ are group homomorphisms and $\alpha: \Lambda_1\to \GL_{pw''}\times\bbC^*$ is the homomorphism defined in  \S~\ref{subsec:alpha}. 
Under the conditions on the field $\mathbb{F}$ specified in \S\ref{subsec:Steinberg tensor}, there is an isomorphism of representations of  $U_w^n(\gamma,\mathbb{F})$
\[
L_w^{(n)}(\gamma, \mathbb{F})
\cong L_{w'}^{(n)}(\gamma', \mathbb{F})\otimes
\Fr_{n, n-1}^* L_{w''}^{(n-1)}(\gamma'', \mathbb{F}),
\]
where $U_w^n(\gamma,\mathbb{F})$ acts on the right hand side via the coproduct \eqref{DeltaE} in the setting of Section \ref{sec:tensor product with dynamical}. 
\end{thm}
When the weight $w$ lies in the reduced region $X^+_{\red}$, we choose $\alpha=(\alpha_1, \cdots, \alpha_n): \Lambda_n\to \GL_w\times \C^*$ be a group homomorphism satisfying the property in \S\ref{subsec:group_hom}. In Theorem \ref{thm:reduced region}, we show the irreducible representation $L_{w}^{(n)}(\alpha, \mathbb{F})$ at level $n$ is isomorphic to the irreducible representation $L_{w}^{(1)}(\alpha_1, \mathbb{F})$ at level $1$. 
In particular, $L_{w}^{(n)}(\alpha, \mathbb{F})$ for $w\in X^+_{\red}$ can be understood using the representation theory of the quantum loop algebra $U_{\epsilon}(L\fg)$ at roots of unity.

\subsection{Lusztig's character formulas}
Let $G$ be an almost simple, simply connected algebraic group over an algebraically closed field $k$ of characteristic $p > 0$. Fix $T\subset G$  a maximal torus of $G$ and let  $X:=\Hom(T, k^*)$.
Based on the ideas from \cite{Lu2}, in 2015 Lusztig  introduced a family of characters, denoted by $E^n_{\lambda}$ \cite{Lusz}, as elements in  $\Z[X]$ for each dominant weight $\lambda \in X^+$ of $G$. We refer the readers to Lusztig's original paper for the details of the character formulas. Here we only recall some context and properties relevant for the present paper. 
The character formula $E_{\lambda}^{n} \in \Z[X]$, labelled by a prime number $p$ and a positive integer $n$, is defined inductively with $E_{\lambda}^{0}$ the Weyl character formula. 
For an element $\xi=\sum_{\mu} c_{\mu}e^{\mu}\in \Z[X]$, we write $
\xi^{[t]}:=\sum_{\mu} c_{\mu}e^{p^t\mu}\in \Z[X]$. 
For any $\lambda\in X^+$ we can write uniquely $\lambda=\sum_{i\geq 0} p^i \lambda_i$, where $\lambda_i\in X^+_{\red}$. 
Then 
\[E^k_\lambda=E^1_{\lambda_0}(E^1_{\lambda_1})^{[1]}\dots(E^1_{\lambda_{k-1}})^{[k-1]}(E^0_{\Sigma_{j\geq k}p^{j-k}\lambda_j})^{[k]}.
\]
In particular, $E^1_\lambda$ is known to be the character of irreducible representation of Lusztig's quantum group at a root of unity with highest weight $\lambda$ \cite{Lu2}. 
For $\lambda$ in the reduced region $X^+_{\red}$,
we have 
\[
E_{\lambda}^1=E_{\lambda}^2=\cdots= E_{\lambda}^{\infty}.
\]
Lusztig's conjecture \cite{Lu2} can then be formulated as, for $p$ large enough, the limit character $E^{\infty}_{\lambda}$ is equal to that of the irreducible modular representation of $G$ with highest weight $\lambda$.

This statement is false for small $p$. We refer interested readers to \cite{W} for a discussion about the precise meaning of being small. In particular, $E^1_\lambda$ describes only the characters of irreducible representations of the quantum group, not the algebraic group. To motivate the present paper, it suffices to say that, for $k>1$, the meaning of the formula $E_{\lambda}^k$ had been unknown. 
This opens the possibility  of the existence of new quantum groups whose irreducible representations have characters given by $E^n_{\lambda}$, for the intermediate $n$ and any $\lambda\in X^+$. The motivation of the present paper is to construct such family of quantum groups with the desired properties.

Associated to an arbitrary oriented cohomology theory $A^*$,
in earlier work a quantum group $U(A^*, \fg)$ is constructed 
\cite{YZ2} as the Drinfeld double of preprojective $A^*$-cohomological Hall algebra of the Dynkin quiver of $Q$.
When $A^*$ is taken to be $K(p, n)$, the Morava-$K$-theory for a prime number $p$ and chromatic level $n$, some properties of $U(K(p, n), \fg)$ are studied in \cite[\S~6]{YZ1}, including a stabilization property that for $n$ large enough 
 compared to $\lambda$ the character of irreducible modules of $U(K(p, n), \fg)$ stabilizes. 
 
 In the present paper, by using the realization of the quantum groups in terms of the Morava $E$-theoretical convolution algebras $U^n_w(\gamma, \mathbb{F})$, we study the characters of their irreducible representations. Using the Steinberg tensor product theorem (Theorem \ref{thm:Steinberg tensor}) and the result for the reduced region (Theorem \ref{thm:reduced region}), together with  a theorem of Chari and Pressley \cite{CP}, we show, in type $A$, there exists group homomorphism $\gamma$, such that  (Corollary \ref{cor:for type A proof})
 \begin{equation}\label{eq:intro}
 \ch(L_{w}^n (\gamma, \mathbb{F}))=E_w^n. 
 \end{equation}

 Following Nakajima \cite{Nak04}, we define the $\epsilon,t$-character  $\ch_{\epsilon,t}(L_{w}^n (\gamma, \mathbb{F}))$ of $L_{w}^n (\gamma, \mathbb{F})$. 
The Steinberg tensor product theorem (Theorem \ref{thm:Steinberg tensor}) and the result of the reduced region (Theorem \ref{thm:reduced region}) then give a formula for $\ch_{\epsilon,t}(L_{w}^n (\gamma, \mathbb{F}))$  for all types, see \eqref{eq:ch_et}.

\subsection{Why Morava $E_n$-theory?}\label{subsec:heuristic}
We close the introduction with an informal cartoon illustration on the intuition behind the present paper, as why the Morava $E_n$-theory is the correct object to consider for iteration of quantum Frobenius and hence Lusztig's formula.
Lusztig's formula satisfies the property $E^{(n)}_{p^nw}=[E^{(0)}_{w}]^{[n]}$. 
This suggests that were these characters of  representations  of some quantum group $U_{q}(\fg, n)$, the quantum group  should have a Frobenius map that can be iterated $n$ times with the final target the enveloping algebra of $\fg$.  

Roughly speaking, Hopkins, Kuhn, Ravenel show that Morava theories have the following feature. 
The $p$-torsion points in the formal group of $E_n^*$ 
is isomorphic to $E_n^*(B\Z/p)$, which roughly speaking is isomorphic to $(\Z/p)^n$. 
We can view $(\Z/p)^n$ as the roots of unity on $ \overbrace
{S^1\times S^1\times \cdots \times S^1}^{\text{n-copies}}$. 
Using them to construct the quantum group $U_{q}(\fg, n)$, and specialize at roots of unity, we would obtain a deformation of the enveloping algebra of $\calO_{F[p]}\otimes \fg$, where $F[p]$ is the set of p-torsion points in the formal group law of $E_n^*$. By the discussion above, $\calO_{F[p]}\otimes \fg$ should be viewed as $p$-localized  $n$-fold loop algebra of $\fg$. Thus, $U_{q}(\fg, n)$ at roots of unity should be thought as $p$-localized  $n$-fold loop quantum group. 

Applying the quantum Frobenius morphism $\Fr_{n, n-1}$, 
we reduce one of the copies of $\bbZ/p\subset S^1$, so that we end up with $p$-torsion points in $ \overbrace
{S^1\times S^1\times \cdots \times S^1}^{\text{(n-1)-copies}}$. The chromatic level of the cohomology theory drops by $1$. 
By iteration of this is the procedure, we can get rid of those $S^1$'s in $n$ steps and end up with the enveloping algebra of $\fg$, with irreducible representations described by Weyl's formula $E^{(0)}_{w}$.

\subsection*{Acknowledgments} We thank David Ben-Zvi, Marc Levine, Ivan Mirkovi\'c, who back in 2016 encouraged us to study quantum groups built from Morava $E$-theories. We thank Nora Ganter for drawing our attention to \cite{St} and patiently explaining her related work. We thank David Gepner and Arun Rum for helpful discussions.
During the preparation of the present paper, Y.Y. was partially supported by the Australian Research Council (ARC) via grants DE190101231 and DP210103081; G.Z. was partially supported by  ARC via DE190101222 and DP210103081.

\section{Topology preliminaries}
We recall some basics in oriented cohomology theory, convolution algebra, Morava $E$-theory, transchromatic character map, as well as the Hopkins-Kuhn-Ravenel theorem. We then use it to give a localization formula for convolution algebra in Morava $E$-theory.

\subsection{Convolution algebra in oriented cohomology theory}\label{subsec:conv}
We work with a topological oriented cohomology theory $A^*$, which for any compact Lie group $G$ we extend to an equivariant version by Borel construction $A^*_G$, although most of the results hold with other equivariant cohomology theories. 
Such a theory has enough structure to define a convolution algebra. A version of such in the algebraic setting is spelled out in detail in \cite[\S~5.1]{ZZ1}. A topological version works in the same fashion. For the convenience of the readers, we briefly recall some basics.

In this paper we mainly consider topological spaces which are either quasi-projective complex manifolds or complex varieties with an action of finite groups or more generally compact Lie groups. 
In the present paper for simplicity we make the convention that all complex algebraic varieties are assumed to be separated quasi-projective of finite type over $\bbC$, but not necessarily irreducible. We call such a variety a complex manifold if it is furthermore smooth. 
For a complex manifold $M$, we consider the pointed topological space $(M\sqcup\pt,\pt)$. If $M$ has an action of a compact Lie group $G$, we let $G$ act on $M\sqcup\pt$ by the trivial action on $\pt$.

 For a complex manifold $M$ with a $G$-action,  by abusing of notation, we denote by $A^*_G(M)$ the equivariant $A^*$-cohomology of the pair $(M\sqcup\pt,\pt)$. For a not-necessarily-smooth subvariety invariant under the $G$-action $i:S\inj M$, we denote $A^*_G(M,M\setminus S)$ by $A^*_G(S)$, and the natural map $A^*_G(S)\to A^*_G(M)$ by $i_*$. 
 Recall that a complex orientation on $A^*$ gives rise,  for a rank-$r$ vector bundle $V\to M$ on a smooth manifold, to an isomorphism $A^*(M)\cong A^{*+r}(V,V\setminus M)$ where $M\inj V$ is the zero-section. This construction extends equivariantly. Hence, the above definition of  $A^*_G(S)$ has no ambiguity in the smooth case. 
 Note that this definition however does not respect the cohomological grading, hence extra care would be needed to keep track the degree. We do not use the cohomological degree in the present paper, hence for simplicity  we ignore this issue.

 Let $f:M\to N$ be an equivariant map between complex manifolds, and $i:S\inj M$, $j:T\inj N$ embeddings of $G$-invariant subvarieties so that the diagram is Cartesian
 \[\xymatrix{
 S\ar[r]\ar[d]_i&T\ar[d]^j\\
 M\ar[r]_f&N.
 }\]
 Then we have the map induced by the map of pairs 
 $f^*:
 A^*_G(T):=A^*_G(N,N\setminus T)
 \to A^*_G(M,M\setminus S)=:A^*_G(S)$, which makes the following diagram commutative
 \[\xymatrix{
 A^*_G(M)&A^*_G(N)\ar[l]_{f^*}\\
 A^*_G(S)\ar[u]_{i_*}&A^*_G(T)\ar[u]_{j_*}\ar[l]_{f^*}.
 }\] 
 This map is called the refined Gysin pullback in \cite[\S~5.2]{ZZ1} and is denoted by $f^!$ therein, and the commutative diagram above is the base-change property.

For an equivariant vector bundle $V$ on $X$, let $s:X\to V$ be the zero section. The equivariant Euler class operator $e^G(V)\cup: A^*_G(X)\to A^*_G(X)$ is defined to be $s^*s_*$. This definition is functorial with respect to pullback of vector bundle along an equivariant map $f:Y\to X$. The usual 
Thom-Pontrjagin construction defines pushforward for any equivariant projective morphism $f:M\to N$, which again extends to the singular case and gives $f_*:A^*_G(S)\to A^*_G(T)$. 
This map satisfies the projection formula with respect to the pullback.

Recall that a complex orientation on $A^*$ also gives rise to a genus, i.e., a ring homomorphism from the complex cobordism $MU^*\to A^*$.
The genus is furthermore functorial with respect to pullbacks and pushforwards. In particular, the pullbacks and pushforwards in $MU^*$ are easily described in terms of cobordant classes of complex manifolds \cite[\S~1]{Qui}.
The orientation also gives $A^*$ a formal group law, so that the formal group law on $MU^*$ is the universal one of Lazard and that the genus is compatible with formal group laws. We refer the readers to \cite[\S~II.1.]{Ad} for the details.
This discussion extends equivariantly using the Borel construction.

Let $M_i$ be smooth quasi-projective $H$-varieties for $i=1,2,3$, and let $Z_{12}\subseteq M_1\times M_2$ and $Z_{23}\subseteq M_2\times M_3$ be $H$-stable closed subvarieties, not necessarily smooth. Let $\proj_{i,j}:M_1\times M_2\times M_3\to M_i\times M_j$ be the projection. Denote $Z_{12}\times_{M_2}Z_{23}\to M_1\times M_2\times M_3$ to be the fiber product  $\proj_{12}^{-1}(Z_{12})\times_{(M_1\times M_2\times M_3)}\proj^{-1}_{23}(Z_{23})$; and define $Z_{12}\circ Z_{23}\subseteq M_1\times M_3$ to be the image of $Z_{12}\times_{M_2}Z_{23}$ under the projection $\proj_{13}:M_1\times M_2\times M_3\to M_1\times M_3$. Note that in general $Z_{12}\circ Z_{23}$ is not smooth even when both $Z_{12}$ and $Z_{23}$ are.

The following Proposition allows us to consider convolution algebras and their representations. The equivariant $K$-theory version is \cite[2.7.5]{CG}.

\begin{prop}\label{prop:conv_action}
\begin{enumerate}
\item Assume that the natural map $Z_{12}\times_{M_2}Z_{23}\to Z_{12}\circ Z_{23}$ is proper. Then, there is a well-defined map \[*:A_H(Z_{12})\otimes_{A_H(\pt)} A_H(Z_{23})\to A_H(Z_{12}\circ Z_{23}).\] 
\item The map $*$ is associative in the usual sense. 
\item Suppose $M\to M_0$ is an $H$-equivariant projective morphism. 
Let $M_x$ be the fiber at $x\in M_0$.
Let $H_x\subset H$ be the stabilizer of $x$ so that $H_x$ acts on $M_x$. Let $Z$ be an $H$-stable closed subvariety of $M\times_{M_0}M$. Then, the induced morphism $A_H(Z)\otimes A_{H_x}(M_x)\to A_{H_x}(M_x)$ is well-defined. 
\item In particular, when $Z=M\times_{M_0}M$, the $R$-module homomorphism above extends to a homomorphism of $R$-algebras $A_H(Z)\to\End_R(A_{H_x}(M_x))$.
\end{enumerate}
\end{prop}

\begin{example}\label{ex:conv_sm}
We look at special cases that will be used. Let $M_i\to M_0$ be projective equivariant morphisms of quasi-projective $H$-varieties with $M_i$ smooth for $i=1,2$. Let $i_Z:Z\to M_1\times_{M_0}M_2$ be an $H$-stable smooth closed subvariety, proper over $M$.
Take $M_3=\pt$ and $Z_{12}=Z$, $Z_{23}=M_2$, and $Z_{13}=M_1$.
Then for any $\eta\in A_H(Z)$ we have an operator $\eta*_Z\in \Hom_R(A_H(M_2),A_H(M_1))$ such that  $\eta*_Z \mu=\proj_{1*}\left((\proj_2^*\mu)\cdot \eta\right)\in A_H(M_1)$ for any $\mu\in A_H(M_2)$, where $\proj_i:Z\to M_i$ is the composition of $i_Z$ with the $i$-th projection $M_1\times_{M_0} M_2\to M_i$.
\end{example}

\subsection{Recollection of Morava $E$-theory}
\label{sub:Recall Morava}
For any oriented cohomology theory $A^*$, we denote $A^*(\pt)$ by  $A^*$, as a graded ring. The degree zero piece is denoted by $A^0$.
In the present paper, we will mostly take $A^*$ to be the Morava $E$-theory. 
There are many versions of the Morava $E$-theory. We focus on the integral version of Hopkins and Miller, denoted by $E_n$. 
Interested readers are referred to \cite{HKR,Rez98} for details. Let $W(\bbF_{p^n})$ be the ring of Witt vectors of the finite field $\bbF_{p^n}$. The theory $E_n$ has coefficient ring given by
\[
E^*_n\cong W(\bbF_{p^n})[\![w_1,\cdots,w_{n-1}]\!][u^{\pm}],\] where $w_i$ has degree zero, and $u$ has degree $p^n-1$.
Therefore, 
\[
E^0_n\cong W(\bbF_{p^n})[\![w_1,\cdots,w_{n-1}]\!]. 
\]
For each $0< t\leq n$, let $K(n-t)$ be the ``integral lifts" of the Morava $K$-theory with coefficient ring $W(\bbF_{p^{n-t}})[u, u^{-1}]$. Let $(E_n^0)_{K(n-t)}$ be the localization of $E_n^0$ along $K(n-t)$. This is a height $(n-t)$-theory. The coefficient ring is given by\[
(E_n^0)_{K(n-t)}\cong W(\bbF_{p^n})[\![w_1,\cdots,w_{n-1}]\!][w_{n-t}^{\pm}]^\wedge_{(p,w_1,\cdots,w_{n-t-1})}. 
\]
In other words, the ring $(E_n^0)_{K(n-t)}$ is obtained from $E_n^*$ by inverting the element $w_{n-t}$ and then completing with respect to the ideal $(p, w_1, \cdots, w_{n-t-1})$. 

In the present paper, for simplicity, we follow Stapleton \cite{St} and pass to the separable closer $\overline{\bbF_p}$ and still refer the resulting theory as $E_n$, so that for the purpose of the present paper, $E_n^0\cong W(\overline{\bbF_p})[\![w_1,\dots,w_{n-1}]\!]$. We remark that with a more careful albeit tedious treatment all the results can be done over explicit finite extensions of  $W(\bbF_{p^n})[\![w_1,\dots,w_{n-1}]\!]$.

Let $[p]x:=\overbrace
{x+_Fx+_Fx +\cdots +_F x}^{\text{p-copies}}$ be the $p$-series of $F$. 
By \cite[Lemma 5.7]{HKR}, there is an isomorphism 
\[
E_n^*(B\Z/p)=E_n^*[\![x]\!]/([p]x)
\]
By \cite[Theorem 5.6]{HKR}, let $L$ be an $E^*_n$-algebra that is also an algebraically closed graded field of characteristic $0$. Then,
\[
\Hom_{E^*_n}(E_n^*(B\Z/p), L)\cong (\Z/p)^n. 
\]

As shown in \cite{St}, there is a universal $(E_n^0)_{K(n-t)}$-algebra $C_{n-t}$ such that the formal group on $(E_n^0)_{K(n-t)}$ splits into the direct sum of a height $n-t$ formal group and a constant $\bbQ_p/\bbZ_p^{t}$. Furthermore, $C_{n-t}$ is non-zero and flat over $E_n^0$. We now briefly recall the construction of $C_{n-t}$ from \cite{St}.  
For any $t$, set $\Lambda_t:=(\Z/p)^{\oplus t}$. 
Let 
\[C_{n-t}':= (E_n^0)_{K(n-t)}\otimes_{E_n^0} E_n^0(B\Lambda_t)=(E_n^0)_{K(n-t)}(B\Lambda_t).
\]
The algebra $C_{n-t}$ is the localization
\[C_{n-t}:=S_{\Lambda}^{-1}C_{n-t}',\] where 
$S_{\Lambda}\subset (E_n^0)_{K(n-t)}\otimes_{E_n^0} E_n^0(B\Lambda_t)$ is the multiplicative set described as follows. 
Occasionally, we need to keep track of the dependence on $n$, and hence in that case write $C^*_{n-t}$ as $C_{n,n-t}^*$.

In general, let $A^*$ be an arbitrary cohomology theory.
Any irreducible representation $\chi$ of $\Lambda_t$ defines a line bundle $L_\chi$ on $B\Lambda_t$. 
The multiplicative set $S_{\Lambda}$ in  $A^*(B\Lambda_t)$ is generated by
\begin{equation}\label{S_Lam}
\{c_1^A(L_{\chi})\mid \chi\,\  \text{is a non-trivial irreducible representation of $\Lambda_t$}\}. 
\end{equation}

\subsection{Cyclotomic rings associated to $A^*$}

\begin{definition}
\label{def:cyclotomic}
The $t$-th cyclotomic ring of $A^*$ is defined to be the localization
\[
\Phi_t(A^*):=A^*(B\Lambda_t)[\frac{1}{S_{\Lambda}}], \text{where $S_{\Lambda}$ is defined in \eqref{S_Lam}}. 
\]
\end{definition}
We also write $\Phi_t(A^*)$ simply as $\Phi_t^A$, or $\Phi_t$ when $A^*$ is understood from the context. If $A^*\to B^*$ is a graded ring homomorphism, then we have the base-change $\Phi_t^B:=\Phi_t^A\otimes_{A^*}B^*$. If $B^*$ defines a cohomology theory coming from a flat ring homomorphism $A^0\to B^0$, 
then $A^*(B\Lambda_t)\to B^*(B\Lambda_t)$ sends $c_1^A$ in $A^*$-theory to $c_1^B$ in $B^*$-theory, and hence
$\Phi_t^B$ defined above is isomorphic to $B^*(B\Lambda_t)[\frac{1}{S_{\Lambda}}]$.
\begin{remark}
By definition of $C_{n-t}^*$, under the map $E^*_{n}(B(\Lambda_t)) \to C_{n-t}^*$ the image of the set $S_{\Lambda}$ is already invertible. However, in the application in \S~\ref{sec:quantum Frob} we need to consider $C^*_{n-t}(B\Lambda_t)$ and the invert $S_{\Lambda}$ in this sense, i.e., the Chern classes are taken in $C^*_{n-t}$-theory. Hence $\Phi_t^{C_{n-t}}$ is still useful in our setting. See also Remark~\ref{rmk:localization}.
\end{remark}

\begin{example}
\label{ex:K-theory}
In this example, we take $A^*$ to be the K-theory. 
We then have the isomorphism 
\[
K^0_{\Z/p}(\pt)_{\bbC}\cong \Rep(\Z/p)_{\bbC}\cong \C[z]/(z^{p}-1), 
\]
where $\Rep(\Z/p)$  is the representation ring of $\Z/p$ and  $z$ corresponds to the one dimensional representation $L_{\zeta}$ sending a generator of $\Z/p$ to a primitive $p$th root of unity. 

Clearly $S$ is the multiplicative set generated by 
\[
\{c_1(L_{\zeta}^{\otimes k}) \mid  k=1, 2, 3, \cdots, {p-1}\}
=\{1-z^k\mid  k=1, 2, 3, \cdots, {p-1}\}. 
\]
Then, the cyclotomic ring in Definition \ref{def:cyclotomic} is 
\[
\Phi_1(K^0_\bbC)=\C[z]/(z^{p}-1)[\prod_{k=1}^{p-1}\frac{1}{1-z^k}]
=\C[z]/(\Phi(z)), 
\]
where $\Phi(z)=1+z+z^2+\cdots +z^{p-1}$ is the cyclotomic polynomial. 
Taking spectrum, $\Spec(K^0_{\Z/p}(\pt)_\bbC)$ is the closed subvariety of $\C$ consisting of all $p$-th roots of unity in $\C^*$, and 
 $\Spec(\C[z]/\Phi(z))$ is the one consisting of all $p$-th roots of unity but $1$. 
 
 Note that the equality also holds integrally \[
\Phi_1(K^0)=\bbZ[z]/(\Phi(z)).
\]
That is, the first cyclotomic ring of $K$-theory is the ring of cyclotomic integers. 
\end{example}

The embedding $e\inj G$ of the trivial group $e$ into $G$ induces a ring homomorphism $E^*(BG) \to E^*(\pt)$. 
Taking spectrum, we get an $E^*(\pt)$-point which we denote by $1 \in \Spec E^*(BG)=\Spec E^*_G(\pt)$. 

\begin{lemma}\label{lem:free}
Let $G$ be a finite group, and let $X$ be a finite $G$ $CW$-complex. Assume the $G$ action on $X$ is free. 
Then, 
$E^*_G(X)$ as a module over $E^*_G(\pt)$ is supported on the point $1$ arising from $e\in G$. 
\end{lemma}
\begin{proof}
As the action of $G$ on $X$ is free, we have the isomorphism
\[
E^*_{G}(X)=E^*(EG\times_G X)\cong E^*(EG\times X/G)
 \cong E^*_{e}(X/G). 
\]
Therefore, the support of $E^*_{G}(X)$ is on $1 \in \Spec E^*_G(\pt)$ as the support of $E^*_{e}(X/G)$ is on $1$. 
\end{proof}

For any order $p$ element $b\in G$, we denote by $\langle b\rangle \cong \Z/p$ the cyclic subgroup generated by $b$. 
In Lemma~\ref{lem:free}, let us consider the special case when $G$ is generated by $b\in G$. In \eqref{S_Lam}, putting $\langle b\rangle $ in place of $\Lambda_t$, we have a multiplicative set $S_b\subseteq E_{\langle b\rangle}^*(\pt)$.
We then have the natural ring homomorphism of localization
\[
E_{\langle b\rangle}^*(\pt)\to E_{\langle b\rangle}^*(\pt)[\frac{1}{S_b}]
\]

\begin{lemma}\label{lem:cohomology_free_action}
Assume $\langle b\rangle$ acts freely on $X$. Then, we have
\[
E^*_{\langle b\rangle} (X)\otimes_{E^*_{\langle b\rangle}(\pt)} (E_{\langle b\rangle}^*(\pt)[\frac{1}{S_b}])=0. 
\]

\end{lemma}
\begin{proof}
By Lemma \ref{lem:free}, the module $E^*_{\langle b\rangle} (X)$ is supported on the point $1$. For any element $f$ in $S_b$, which is the $c_1$ of a non-trivial representation, we have $f$ vanishes at $1$ and hence annihilates $E^*_{\langle b\rangle} (X)$. 
\end{proof}

\subsection{The transchromatic character maps}
Generalizing the character map of Hopkins, Kuhn, and Ravenel  \cite{HKR}, a character map from the Morava $E_n^*$-theory to the   $C_{n-t}^*$-theory for $0< t\leq n$ has been constructed \cite{St}, which we now recall.

Let $G$ be a finite group, and let $X$ be a finite $G$ $CW$-complex. Fix a group homomorphism $\alpha: \Lambda_t\to G$. Let $\iota: X^{\im \alpha}\to X$ be the embedding of the fixed points of $\im \alpha$. 
We have the following diagram
\[
\xymatrix{
\Lambda_t \times X^{\im \alpha}  \ar[r]^{\alpha\times \iota} \ar[d]& G\times X \ar[d]\\
X^{\im \alpha} \ar@{^{(}->}[r]^{\iota} & X}
\]
where the vertical maps are the group actions and the action of $\Lambda_t$ on $X^{\im \alpha}$ is trivial. 
Therefore, it induces a map on spaces  
\[
E\Lambda_t \times_{\Lambda_t} X^{\im \alpha} \cong B\Lambda_t \times X^{\im \alpha} \to EG\times_G X. 
\]
Applying $E_n^*$, we obtain 
\[
E_n^*(EG\times_G X) \to 
E_n^*( B\Lambda_t \times X^{\im \alpha})\cong 
E_n^*( B\Lambda_t)\otimes E_n^*(X^{\im \alpha}). 
\]
Composing the ring homomorphim $E_n^*( B\Lambda_t)\to C_{n-t}^*$, we obtain the transchromatic character map
\[
\chi_{n, n-t}^G: E_n^*(EG\times_G X)\to C_{n-t}^*(X^{\im \alpha}). 
\]
%\yaping{Do we need to recall the action of the centralizer and put the equivariance on the target?}
Let $E^*$ be a complex oriented theory with associated formal group law $F$. We assume the graded ring $E^*$ and $F$ with height-$n$
satisfy the following condition. 
\begin{assumption}\label{ass:HKR}
\begin{itemize}
\item
$E^*$ is local with maximal ideal $\mathfrak{m}$, and complete in the $\mathfrak{m}$-adic topology.
\item
The graded residue field $E^*/\mathfrak{m}$ has characteristic 
$p > 0$.
\item $p^{-1}E^*$ is not zero.
\item
The mod $\mathfrak{m}$ reduction of $F$ has height $n < \infty$ over $E^*/\mathfrak{m}$.
\end{itemize}
\end{assumption}
Assume $E^*$ is a cohomology theory satisfying Assumption~\ref{ass:HKR} with height $n$. Define
\[
L(E^*):=S_{\Lambda}^{-1} E^*(B\Lambda_n)=C_{0}^*. 
\]
\begin{theorem}\label{thm:HKR}\cite[Theorem C]{HKR}
The character map $\chi_{n, 0}^G$ induces an isomorphism
\[
L(E^*)\otimes_{E^*} E^*(EG\times_G X) 
\cong
L(E^*)\otimes_{E^*} E^*( 
\sqcup_{\alpha\in \hom(\Lambda_n, G)} 
X^{\im \alpha})^G
\]
\end{theorem}
\subsection{The equivariant Euler class}
Now we give one application of the theorem of Hopkins, Kuhn, Ravenel, which is useful for later. 

Let $X$ be a complex variety with the trivial action of a compact Lie group $A$. 
Let $\pi: \mathcal{E} \to X$ be an $A$-equivariant vector bundle \cite[Chapter 5.1]{CG} on $X$. 
The vector bundle decomposes according to the $A$-action as \begin{equation}\label{eqn:bundle_decompose}
    \oplus_{i\in \hat A} \calE_i\otimes_\bbC V_i
\end{equation} where $V_i$ is an irreducible representation of $A$ made into a trivial vector bundle on $X$, and $\calE_i$ has fiberwise the trivial $A$-action but not necessarily trivial as a vector bundle.

The following proposition should be compared with \cite[Proposition 5.10.3]{CG}. 
\begin{prop}\label{lem:lambda_invertible}
Let $A=A_1\times A_2$ such that $A_1$ and $A_2$ are finite groups and $A_1$ is abelian. Let $\calE\to X$ be an $A$-equivariant vector bundle. 
Assume $A_1$ acts trivially on $X$ and in the decomposition \eqref{eqn:bundle_decompose} under $A_1$-action each $L_i$ with $\calE_i\neq 0$ is non-trivial as $A_1$-representation. Assume $E^*$ is a cohomology theory satisfying Assumption~\ref{ass:HKR} with height $h$. Assume further that $E^*$ has K\"unneth formula with respect to $A_1$. Let $E_{\loc}=S^{-1} E^*_{A_1}(\pt)\otimes_{E^0} L$, where $S$ is the multiplicative set generated by \[\{ c_1^{A_1}(V)\in E^*_{A_1}(\pt)\mid  V\hbox{ a non-trivial irreducible representation of }A_1\}.\] Then, multiplication by the $A$-equivariant Euler class $e(\calE)$ induces an isomorphism on the localized $E$-theories
\[
E^*_A(X)\otimes_{E^*_{A_1}} E_{\loc}\xrightarrow{e(\calE)} E^*_A(X)\otimes_{E^*_{A_1}} E_{\loc}. 
\]
\end{prop} 
Here recall that $L$ is the same as given in Theorem~\ref{thm:HKR}. By construction, we have two algebra morphisms $L\to E_{\loc}$  as well as $E^*_{A_1}\to E_{\loc}$, which agree when both are restricted to $E^0$. That is, we have the following commutative diagram
\[
\xymatrix@R=1em @C=1em{
&E_{\loc}&\\
L \ar[ur]& &E_{A_1}^* \ar[ul]\\
&E^0 \ar[ul] \ar[ur]&
}
\]
We need a few lemmas.
\begin{lemma}\label{lem1}
Let $R$ be a commutative ring with an idempotent decomposition $R=\oplus_{i\in I}R_i$ such that the index set $I$ is finite. 
Let $H$ be a finite group acting compatibly on both $I$ and $R$ (the $H$ action on $R$ preserves the ring structure). That is, for an idempotent decomposition $a=\sum_i a_i$ of $a\in R$ and for any $h\in H$, we have the idempotent decomposition $ha=\sum_{i} h(a_{i})$, such that $h(a_{i})\in R_{h(i)}$. For $a\in R^H$, if each $a_i$ in $R_i$ is invertible, then $a$ is invertible in $R^H$.
\end{lemma}
\begin{proof}
Let $a=\sum_i a_i$ be the idempotent decomposition of $a\in R^H$. The assumption that $a$ is invariant under $H$ implies
\[
ha_i=a_{h(i)}, \text{for any $h\in H$.}
\]
Each $a_i$ is invertible in $R_i$. Hence, $h \frac{1}{a_i}=\frac{1}{a_{h(i)}}$, for any $h\in H$. 
Clearly, $\sum_i \frac{1}{a_i}$ is inverse to $a$ in $R$ and is $H$ invariant. This completes the proof. 
\end{proof}
\begin{lemma}\label{lem2}
Let $A_1, A_2$ be two abelian groups and $L_1, L_2$ be the irreducible representations of $A_1, A_2$ respectively. 
If $c_1^{A_1}(L_1)$ is invertible, then, 
so is $c_1^{A_1 \times A_2}(L_1\otimes L_2)$. 
\end{lemma}
\begin{proof}
Denote the image of $c_1^{A_1}(L_1)\in E_{A_1}(\pt)$ under the natural map
\[
E_{A_1}(\pt) \to E_{A_1\times A_2}(\pt) 
\]
by the same notation $c_1^{A_1}(L_1)$. Then, we have
\[
c_1^{A_1\times A_2}(L_1\otimes L_2)
=c_1^{A_1}(L_1)+c_1^{A_2}(L_2)+h.o.t., 
\]
where $h.o.t.$ consists of higher order terms in  $c_1^{A_1}(L_1)$ and $c_1^{A_2}(L_2)$.
By the assumption that $c_1^{A_1}(L_1)$ is invertible, 
multiplying $c_1^{A_1\times A_2}(L_1\otimes L_2)$ by $(c_1^{A_1}(L_1))^{-1}$, we obtain
\begin{equation}\label{lem:c_1}
1+(c_1^{A_1}(L_1))^{-1}c_1^{A_2}(L_2)+(c_1^{A_1}(L_1))^{-1}h.o.t.. 
\end{equation}
Recall that $E_{A_1\times A_2}(\pt)$ is defined via the Borel construction, hence it is complete with respect to Chern classes. Therefore, \eqref{lem:c_1} is invertible. 
This completes the proof. 
\end{proof}

\begin{proof}[Proof of Proposition~\ref{lem:lambda_invertible}]
By K\"unneth formula, and the assumption of the trivial $A_1$-action we have $E_A^*(X)\cong E_{A_1}^*\otimes_{E^0}E_{A_2}^*(X)$. Hence $E^*_A(X)\otimes_{E^*_{A_1}}E_{\loc}=E^*_{A_2}(X)\otimes_{E^0}E_{\loc}$. Using Theorem~\ref{thm:HKR}, we have 
\[
E_{\loc}\otimes_{E^*_{A_1}}(E_A^*(X)) \cong
S^{-1}E^*_{A_1}\otimes_{E^0}(L\otimes_{E^0}E_{A_2}^*(X))\cong S^{-1}E^*_{A_1}\otimes_{E^0}\Big(\bigoplus_{\beta:(\bbZ/p)^h\to A_2}L\otimes_{E^0_{(\bbZ/p)^h}}E^*_{(\bbZ/p)^h}(X^{\im \beta})\Big)^{A_2}.
\]
Therefore, to prove the proposition, without loss of generality, we may assume that $A_2$ acts trivially on $X$ as well.
Now Lemma~\ref{lem2} above reduces the proof to a standard argument (c.f., proof of \cite[Proposition 5.10.3]{CG}). 

Step 1: Assume $\calE=X\times F$ is a trivial vector bundle. 
We have projections in the following diagram
\[
\xymatrix{
\calE=X\times F\ar[r]^(0.6){p\times \id_F} \ar[d]^{\pi} & F \ar[d]_{\pi_F}\\
X\ar[r]^{p} \ar@/^1.0pc/[u]^{s}
&\pt\ar@/_1.0pc/[u]_{s_F}
}
\]
with $s_F: \pt \to F$ the zero section of $F$ on $\pt$. 
Therefore, $e(\calE)$, which comes from the pullback of a class in $E^*_A(\pt)_{\loc}$, induces an isomorphism thanks to Lemma~\ref{lem2}. 

Step 2: 
We use induction on $\dim(X)$. 
If $dim(X)=0$ then each irreducible component of $X$ is a point and by Step 1 the Proposition follows. 

Let $U\subset X$ be the Zariski open dense subset such that $U$ is stable under $A$-action and $E|_{U}$ is trivial. Let $Y:=X\setminus U$ be the complement. Then, $\dim(Y)<\dim(X)$. 
Write the long exact sequence in $E^*$-theory
\[
\xymatrix{
\cdots \ar[r] & E^{*-1}_{A}(Y)_{\loc} \ar[r] \ar[d]^{e(\calE)}& E^*_{A}(U)_{\loc} \ar[r] \ar[d]^{e(\calE)}& E^*_{A}(X)_{\loc} \ar[r] \ar[d]^{e(\calE)}& E^*_{A}(Y)_{\loc}\ar[r] \ar[d]^{e(\calE)}& \cdots \\
\cdots \ar[r] & E^{*-1}_{A}(Y)_{\loc} \ar[r] & E^*_{A}(U)_{\loc} \ar[r] & E^*_{A}(X)_{\loc} \ar[r] & E^*_{A}(Y)_{\loc}\ar[r] & \cdots \\
}\]
The restriction of $\calE$ to $U$ being trivial, the multiplication by $e(\calE)$ induces an isomorphism on $E^*_{A}(U)_{\loc}$ by Step 1. Also, $\dim(Y)<\dim(X)$, hence $e(\calE)$ induces an isomorphism on $E^*_{A}(Y)_{\loc}$ by the induction hypothesis. The Proposition follows from the $5$ lemma. 
\end{proof}

Let $G$ be a finite group and $A_1\times A_2$ a subgroup with $A_1\cong (\Z/p)^t$ with $t\leq h$.
Let $S$ be the multiplicative set generated by $\{ c_1^{A_1}(V)\in E^*_{A_1}(\pt)\mid  V\hbox{ a non-trivial irreducible representation of }A_1\}$. Thanks to Lemma~\ref{lem2}, we have the following composition of maps
$S^{-1}E^*_{A_1}\to S^{-1}E^*_{A_1\times (\Z/p)^{h-t}} \to  L$.
\begin{corollary} Notations as above. 
Let $X$ be a smooth $G$-variety, and let $N$ be the normal bundle of the embedding $i:X^{A_1}\inj X$, and $e_N$ its Euler class. Then, multiplication by $e_N$ induces an isomorphism on 
\[L\otimes_{E_{A_1}}E_{A_1\times A_2}(X^{A_1})\to L\otimes_{ E_{A_1}}E_{A_1\times A_2}(X^{A_1}).\]
Moreover, the pullback \[i^*:L\otimes_{E_{A_1}}E_{A_1\times A_2}(X)\to L\otimes_{E_{A_1}}E_{A_1\times A_2}(X^{A_1})\]
is an isomorphism. 
\end{corollary}
\begin{proof}
Consider the isomorphism of \cite{HKR} $L\otimes_{E^0}E_G(X)\cong L\otimes_{E_{(\Z/p)^{h}}}(\oplus_{\alpha:(\Z/p)^{h}\to G}E_{ (\Z/p)^{h}}^*(X^{\im\alpha}))^{G}$, and base-change via the ring map $L\otimes_{E^0}E_G\to L\otimes_{E_{A_1}} E_G$, we obtain 
 \[L\otimes_{E_{A_1}}E_G(X)\cong L\otimes_{E_{A_1\times (\Z/p)^{h-t}}}(\oplus_{\beta:(\Z/p)^{h-t}\to A_2}E_{A_1\times (\Z/p)^{h-t}}^*(X^{A_1\times\im\beta}))^{A_2}\cong  L\otimes_{E_{A_1}}(\oplus_{\beta:(\Z/p)^{h-t}\to A_2}E_{A_1}^*(X^{A_1\times\im\beta}))^{A_2}.\]
Observe that $X^{A_1\times\im\beta}=(X^{A_1})^{\im\beta}$ and that the vector bundle $N$ on $X^{A_1}$ satisfies the assumption of Proposition~\ref{prop:localization_convolution}, hence, the same argument as that of Proposition~\ref{lem:lambda_invertible} implies the assertion. 
\end{proof}

\subsection{Localization of convolution algebras}
The following proposition is
deduced from a standard argument, which can be found, say, in  \cite[Theorem 5.11.10]{CG}. However, it is useful in our setting. For the convenience of the readers, we sketch the proof below.

\begin{prop}\label{prop:localization_convolution}
Assume $E$ is a cohomology theory satisfying Assumption~\ref{ass:HKR} with height $h$. 
Assume $A=A_1\times A_2$ with $A_1$ and $A_2$ finite  groups with $A_1=(\Z/p)^t$ with $t\leq h$. Assume $M_i$ is smooth with $A$-action $i=1,2,3$, and the $A$-equivariant Euler class of the normal bundle $M_i^{A_1}\subseteq M_i$ is denoted by $e_i$. 

Then, $\chi=(1\boxtimes \frac{1}{e_2})\circ i^*:E^*_A(Z_{1,2})\to E^*_A(Z_{1,2}^{A_1})$ is well-defined when applying $-\otimes_{E_{A_1}}L$ and commutes with convolutions. 
\end{prop}

\begin{proof}
Let $Z_{12}\subset M_1\times M_2$, $Z_{23}\subset M_2\times M_3$. Let $p_{ij}: M_1\times M_2\times M_3\to M_i\times M_j$ be the projection.  
Let $\calF_{12}\in E_G^*(Z_{12})$, $\calF_{23}\in E_G^*(Z_{23})$. 
Let $e_i\in E_{A }^*(M_i^{A_1})$, $i=1, 2, 3$. 
Then,
\begin{align}
(\chi \calF_{12})* (\chi \calF_{23})
=&(p_{13})_*\Big(p_{12}^* (\chi \calF_{12}) \otimes_{ } p_{23}^* (\chi \calF_{23})\Big)  &&\text{by definition of $*$} \notag\\
=&(p_{13})_* \Big((1\boxtimes \frac{1}{e_2} \boxtimes \frac{1}{e_3} )\otimes (i^*\calF_{12} \otimes i^*\calF_{23}) \Big)
&&\text{by definition of $\chi$}\notag\\
=&(p_{13})_* \Big(p_{13}^*(e_1\boxtimes 1)\otimes (\frac{1}{e_1}\boxtimes \frac{1}{e_2} \boxtimes \frac{1}{e_3}) \otimes (i^*\calF_{12} \otimes i^*\calF_{23}) \Big)&&\notag\\
&\phantom{12}\text{by the equality 
$1\boxtimes \frac{1}{e_2} \boxtimes \frac{1}{e_3}
=p_{13}^*\Big(e_1\boxtimes 1\otimes (\frac{1}{e_1}\boxtimes \frac{1}{e_2} \boxtimes \frac{1}{e_3})\Big)$}\notag\\
=&(e_1\boxtimes 1)\otimes (p_{13})_* \Big( (\frac{1}{e_1}\boxtimes \frac{1}{e_2} \boxtimes \frac{1}{e_3}) \otimes (i^*\calF_{12} \otimes i^*\calF_{23})\Big) && \text{by the projection formula} \label{pf:chif}
\end{align}
The commutative diagram
\[
\xymatrix{
M_1\times M_2\times M_3 \ar[r]^{p_{13}}& M_1\times M_3\\
M_1^{A_1}\times M_2^{A_1}\times M_3^{A_1} \ar[r]^{p_{13}} \ar[u]^{i}
& M_1^{A_1}\times M_3^{A_1} \ar[u]^{i}
}
\]
induces the following commutative diagram on cohomology
\[
\xymatrix{
E^*_{A}(M_1\times M_2\times M_3) \ar[r]^{(p_{13})_*}  \ar[d]^{\frac{1}{e_1}\boxtimes \frac{1}{e_2} \boxtimes \frac{1}{e_3}}&  E^*_{A}(M_1\times M_3) \ar[d]^{\frac{1}{e_1} \boxtimes \frac{1}{e_3}}\\
L\otimes_{E_{A_1}}E^*_{A }(M_1^{A_1}\times M_2^{A_1}\times M_3^{A_1}) \ar[r]^{(p_{13})_*} &  L\otimes_{E_{A_1}}E^*_{A }(M_1^{A_1}\times M_3^{A_1})
}\]
The vertical maps are the inverse of $i_*$. 
As a consequence, we have 
\begin{align*}
\eqref{pf:chif}=&(e_1\boxtimes 1)\otimes  ((\frac{1}{e_1}\boxtimes \frac{1}{e_3}) )(p_{13})_* \Big((i^*\calF_{12} \otimes i^*\calF_{23})\Big) \\
=&(1\boxtimes \frac{1}{e_3}) i^*( \calF_{12} * \calF_{23}) &&\text{by $e_1\frac{1}{e_1}=1$ and definition of $*$} \\
=&\chi ( \calF_{12} * \calF_{23}) &&\text{by definition of $\chi$} 
\end{align*}
\end{proof}

\section{Quiver representations}
We recall some notations and recollect some basics on quiver representations, Nakajima quiver varieties, and cohomological Hall algebras. Experts in this area are recommended to skip ahead to \S~\ref{subsec:def_quant_groups}, where we introduce the quantum groups. 
\subsection{Generality of quiver representations}
Recall that a quiver $\Gamma=(\Gamma_0,\Gamma_1)$ has a finite set of vertices $\Gamma_0$ and a finite set of arrows $\Gamma_1$, with each arrow $h\in \Gamma_1$ having a outgoing vertex $\out(h)\in \Gamma_0$ and an incoming vertex $\inc(h)$ making $\Gamma$ an oriented graph. A representation is a functor $V$ from the graph $\Gamma$ to the category of finite dimensional vector spaces, with $V^i$ the vector space associated to $i\in \Gamma_0$. Abusing notations, we also write $V=\oplus_{i\in \Gamma_0}V^i$ as the underlying vector space. 
We call an element $\gamma=(\gamma^i)_{i\in \Gamma_0}\in \bbN^{\Gamma_0}$ a dimension vector. A representation of the quiver $V$ has the dimension vector with $\gamma^i=\dim V^i$ for  $i\in \Gamma_0$. For each  dimension vector $\gamma\in \bbN^{\Gamma_0}$, let 
\[
\Rep(\Gamma,\gamma):=\prod_{h\in \Gamma_1}\Hom(\bbC^{\gamma^{\out(h)}},\bbC^{\gamma^{\inc(h)}}). 
\] 
The group $\GL_\gamma:=\prod_{i\in I}\GL_{\gamma^i}$ acts on $\Rep(\Gamma,\gamma)$ by conjugation.

Assume there is a subset $\Gamma_0''\subset \Gamma_0$, called the {\em frozen} vertices, with the complement $\Gamma_0'\subseteq \Gamma_0$. Any $\theta\in \bbQ^{\Gamma_0'}$ defines a rational character of $\GL_{\gamma'}$ for any $\gamma'\in \bbN^{\Gamma_0'}$, which we use as GIT stability condition, and hence 
has an open subset $$\Rep(\Gamma,\gamma)^{\theta}\subseteq \Rep(\Gamma,\gamma)$$ consisting of $\theta$-semi-stable points. 
We have the GIT quotient $\coprod_{\gamma\in \bbN^{\Gamma_0}}\Rep(\Gamma,\gamma)^{\theta}/\GL_{\gamma'}$. 
Note that on each $\Rep(\Gamma,\gamma)^{\theta}/\GL_{\gamma'}$ there is an action on $\GL_{\gamma''}$ together with an equivariant projective morphism to the categorical quotient 
$$\Rep(\Gamma,\gamma)^{\theta}/\GL_{\gamma'}\to \Rep(\Gamma,\gamma)/\GL_{\gamma'}.$$ 

We introduce some notations. For a quiver $Q=(I,H)$, the opposite quiver $Q^{\op}=(I,H^{\op})$ has the same set of vertices $I$ as $Q$, and the set of arrows $H^{\op}$ has a fixed one-to-one correspondence with $H$ so that for each $h\in H$, the corresponding arrow $h^*\in H^{\op}$ has  $out(h^*)=in(h)$ and $in(h^*)=out(h)$. 
The double quiver $\overline{Q}=Q\sqcup Q^{\op}$ has the same set of vertices $I$ as $Q$, and the set of arrows is $H\sqcup H^{\op}$.
Note that for any $v\in\bbN^I$ we have an $\GL_v$-equivariant isomorphism $\Rep(\overline{Q},v)\cong T^*\Rep(Q,v)$.

For a quiver $Q$, let $Q^\heartsuit$ be the \textit{ framed quiver}. 
The set of vertices of $Q^\heartsuit$ is $I \sqcup I'$, where $I'$ is another copy of the set $I$, equipped with the bijection $I\to I'$, $i\mapsto i'$. 
The set of arrows of $Q^\heartsuit$ is, by definition, the disjoint union of $H$ and a set of additional edges $j_i : i \to i'$, one for each vertex $i\in I$.
We follow the tradition that $v\in \bbN^{I'}$ is the notation for the dimension vector at $I$, and $w\in \bbN^I$ is the dimension vector at $I'$. Also, we consider the vertices $I'$ as frozen. 

\subsection{Nakajima quiver varieties}
\label{subsec:quiver variety}
In this section, we review the Nakajima quiver varieties and their convolution algebra. Interested readers can find most of the missing details in \cite{Nak}.

Let $\overline{Q^{\heartsuit}}=Q^{\heartsuit}\sqcup Q^{\heartsuit, \op}$ be the double of $Q^{\heartsuit}$. We have the isomorphism
\[
\Rep(\overline{Q^\heartsuit},(v,w))\cong T^*\Rep(Q^\heartsuit,v,w).
\]
Let $\mu_{v, w}:T^*\Rep(Q^\heartsuit,v,w)\to \fg\fl_v^*\cong \fg\fl_v$ be the moment map
 \[
\mu_{v, w}: (x, x^*, i, j)\mapsto 
\sum[x, x^*]+i\circ j \in \fg\fl_v.
\]
For any $\theta=(\theta_i)_{i\in I} \in \Z^{I}$, 
let $\chi_\theta: \GL_v\to \Gm$ be the character $g=(g_i)_{i\in I} \mapsto \prod_{i\in I} \det (g_i)^{-\theta_i}.$ 
The set of $\chi_\theta$-semistable points in $T^*\Rep(Q^\heartsuit, v, w)$ is denoted by $\Rep(\overline{Q^\heartsuit},v,w)^{ss}$.
The Nakajima quiver variety is defined to be the Hamiltonian reduction 
\[\fM_{\theta}(v, w):=\mu_{v, w}^{-1}(0)/\!/_\theta \GL_v.\] 
In this paper, we use the stability condition $\theta^+=(1,\cdots,1)$ and denote $\fM_{\theta}(v, w)$ by $\fM(v, w)$ for simplicity. The categorical quotient is denoted by $\fM_{0}(v, w)$ to avoid possible confusion. We also write $\mu_{v, w}^{-1}(0)^s$ for the $\theta^+$-stable locus in $\mu_{v, w}^{-1}(0)$.
We  write $\fM(w)=\sqcup_{v\in\bbN^I}\fM(v, w)$ and $\fM_0(w)=\cup_{v\in\bbN^I}\fM_0(v, w)$. The meaning of $\cup$ here is explained in \cite[\S~2.5]{Nak}.

\subsubsection{Hecke correspondence and tautological bundles}

As $\fM(w)$ has many connected components, so is $Z(w):=\fM(w)\times_{\fM_0(w)}\fM(w)$. We describe one component relevant to us together with a tautological line bundle on it.

For two dimension vectors $v_1$ and $v_2$,
the composition $\fM(v_i,w)\to \fM_{0}(v_i, w) \subset \fM_{0}(v_1+v_2, w)$ are denoted by $\pi_i$.
Let 
\[
Z(v_1, v_2, w):=\{ (x_1, x_2)\in \fM(v_1,w)\times \fM(v_2,w) \mid \pi_1(x_1)=\pi_2(x_2)\}
\] be the Steinberg variety. 
By the construction of $\fM(v, w)$, we have the tautological vector bundle 
\[
\mu_v^{-1}(0)^{ss}\times _{\GL_v} V \to \fM(v, w)
\]
associated to the principal $\GL_v$-bundle $\mu_v^{-1}(0)^{ss} \to \fM(v, w)$. Here 
$V$ is the $\GL_v$ representation with dimension vector $v$. We denote the vector bundle by $\calV(v, w)$.
We now consider the case when $v_1=v_2-e_k$, where $e_k$ is the dimension vector whose entry $k$ is $1$, and other entries are $0$. The Hecke correspondence $\fP_k(v_2-e_k, v_2, w)$  \cite[\S~5.1]{Nak} is an irreducible, smooth component of  $Z(v_1,v_2,w)$, and  a Lagrangian subvariety of $\fM(v_2-e_k, w)\times \fM(v_2, w)$ as shown by Nakajima. In {\it loc. cit.} $\sqcup_{v_2}\fP_k(v_2-e_k, v_2, w)$ is denoted by $\fP_k(w)$.
The tautological line bundle $\calL_k$ of $\fP_k(v_2-e_k, v_2, w)$ is defined to be the quotient 
\[\calL_k:=\calV(v_2,w)/\calV(v_2-e_k,w).\]
Let $\fM(v_1, w)\times \fM(v_2, w)\xrightarrow{(1, 2)} \fM(v_2, w)\times \fM(v_1, w)$ be the exchange of factors. 
Denote $(1, 2)(\fP_k(v_2, v_2+e_k, w))$ by $\fP_k(v_2+e_k, v_2, w)$. As on $\fP_k(v_2-e_k, v_2, w)$, we have a natural line bundle over $\fP_k(v_2+e_k, v_2, w)$. Let us denote it by
\[
\calL_k^{\op}:=\calV(v_2+e_k,w)/\calV(v_2,w).
\]
 
Similarly, there are generalized Hecke correspondences \cite[\S~5.3]{Nak} $\fP^{(\pm n)}_k=\sqcup_{v}\fP(v,v\pm ne_k,w)$ with each $\fP(v,v\pm ne_k,w)$ a smooth component of $Z(v,v\pm ne_k,w)$. There is a rank-$n$ tautological vector bundle $\calV_k^{(n)}$ on  $\fP(v,v+ne_k,w)$ (resp. a rank-$n$ tautological vector bundle $\calV_k^{(n), \op}$ on  $\fP(v,v-ne_k,w)$ ).

\subsubsection{The torus action}\label{sec:torus_action}
Let $\sqcup_{v\in \N^I}\mathfrak{M}(v, w)$ be the Nakajima quiver variety and $\varmathbb{G}:=\C^* \times \GL_{w}$. 
We now describe the action of $\varmathbb{G}$ on $\mathfrak{M}(w)$ and $\mathfrak{M}_0$ \cite[Section 2.7]{Nak}.
The action of $\GL_{w}$ is defined by its natural action on $\Rep(\overline{Q^\heartsuit},(v,w))$. It preserves the moment map equation $[x, x^*]+ij=0$ and commutes with the conjugation action of $\GL_v$. Hence it induces an action on $\mathfrak{M}(w)$ and $\mathfrak{M}_0$. 

For each pair $k, l\in I$ such that the number of arrows $b'$ from $k$ to $l$ is at least $1$, we fix a numbering $1, 2, \dots, b'$ on the arrows from $k$ to $l$. It induces a numbering $h_1, \dots, h_{b'}, h_1^*, \dots, h_{b'}^*$ on oriented edges between $k$ and $l$. 
Define a weight function $m: H\to \Z$ by 
\[
m(h_s)=b'+1-2s, \,\ m(h_s^*)=-b'-1+2s. 
\]
In the special case when there is only one arrow $h$ from $k$ to $l$. We have $m(h)=m(h^*)=0$. 
The $\C^*$-action on $\Rep(\overline{Q^\heartsuit},(v,w))$ is by 
\[
B_h\mapsto t^{m(h)+1} B_h, i\mapsto ti, j\mapsto tj, \text{for $t\in \C^*$}. 
\]
 This $\varmathbb{G}$-action makes the projective morphism $\mathfrak{M}\to \mathfrak{M}_0$ equivariant.

\subsection{Cohomological Hall algebras and affine quantum groups}\label{subsec:def_quant_groups}

This is not used in the present paper in an essential way and is included in the present paper for the context and future reference.

For any $v\in\bbN^I$, let $\mu_v:\Rep(\overline{Q},v)\cong T^*\Rep(Q,v)\to\fg\fl_v$, $(x, x^*)\mapsto \sum[x, x^*]$ be the moment map. On $\mu_v^{-1}(0)$ there is an action of $\GL_v\times\bbC^*$ where $\bbC^*$-acts the same way as in \S~\ref{sec:torus_action}. On $\oplus_v A_{\GL_v\times\bbC^*}(\mu_v^{-1}(0))$ there is a structure of an associative algebra, called the preprojective cohomological Hall algebra (COHA) associated to the cohomology theory $E^*$ and the quiver $Q$ \cite{YZ1}. The construction follows closely that in \cite{Gr2,KS,SV}.

Let $\calH(E^*, Q)$ ( resp. $\calH'(E^*, Q)$) be the subalgebra of $\oplus_v A_{\GL_v\times\bbC^*}(\mu^{-1}(0))$ generated by elements in $ A_{\GL_{n e_k}\times\bbC^*}(\mu_{n e_k }^{-1}(0))$, where $n$ varies in $\N$ and $k$ varies in $I$ (resp. generated by elements in $ A_{\GL_{ e_k}\times\bbC^*}(\mu_{e_k }^{-1}(0))$, where $k$ varies in $I$). We call $\calH(E^*, Q)$ (resp. $\calH'(E^*, Q)$) the Lusztig spherical COHA (resp. De Concini-Kac spherical COHA).%, as $\calH(E^*, Q)$ contains the divided powers (see \cite[Section 6]{YZ1}). The De Concini-Kac spherical COHA $\calH'(E^*, Q)$ is defined in \cite[Definition 4.2]{YZ1}. 

Similarly, for any $w$, let $A^*_{\varmathbb{G}}(Z(w))$ be the convolution algebra. Consider the subalgebra $U_w(\varmathbb{G}, E^*)$ (resp. $U'_w(\varmathbb{G}, E^*)$) generated by the Chern classes of the tautological vector bundles $\calV_{k}^{(n)}$ and $\calV_{k}^{(n), \op}$ as $n$ varies in $\N$ and $k$ varies in $I$ (resp. generated by the Chern classes of the tautological line bundles $\calL_{k}$ and $\calL_{k}^{\op}$ as $k$ varies in $I$). 
Here when $n=0$, the Hecke correspondence is the diagonal $\fM(w)$, and the tautological bundles are $\calV(v,w)$. Although this is not used in the present paper, the Kirwan surjectivity \cite{MN} shows $A_{\bbG}(\fM(w))$ is tautologically generated.

When the quiver $Q$ has no edge loops, \cite[Theorem 5.6]{YZ1} is true for the generalized Hecke correspondence $\fP(v,v+ne_k,w)$ with the following replacements. We replace $e_k$ by $ne_k$, replace $c_1(\calL_k)$ by the Chern classes of $\calV_k^{(n)}$, and replace 
$ A_{\GL_{e_k}\times\bbC^*}(\mu_{e_k }^{-1}(0))=A_{\bbC^*}(\pt)[z^{(k)}]$ by $ A_{\GL_{ne_k}\times\bbC^*}(\mu_{ne_k }^{-1}(0))=A_{\bbC^*}(\pt)[z^{(k)}_1, \cdots, z^{(k)}_n]^{\Sigma_n}$. 
As a consequence of the generalized statement of \cite[Theorem 5.6]{YZ1}, there is an algebra homomorphisms $\calH(A^*, Q)\to  U_w(\varmathbb{G}, A^*)$ which restricts to a homomorphism $\calH'(A^*, Q)\to  U'_w(\varmathbb{G}, A^*)$. It is compatible with their respective actions on the $A^*$-cohomology of the Nakajima quiver variety
\begin{equation}
\label{eq:algebra hom}
\xymatrix@C=1em
{
\calH(E^*, Q)\ar[rr]^{\Phi_w} \ar[dr]_{a_{\mathcal{H}}} && 
U_w(\varmathbb{G}, E^*) \ar[ld]^{a_{Z}}\\
&\End(E^*_{\varmathbb{G}}(\mathfrak{M}(w)))&
}
\end{equation}
One can take the (reduced) Drinfeld double of $\calH(E^*, Q)$ as in \cite{YZ2}, denoted by $D(\calH(E^*, Q))$. 
The morphism $\Phi_w$ extends to an epimorphism $D(\Phi_w)$ from $D(\calH(E^*, Q))$ to $ U_w(\varmathbb{G}, E^*)$. 
Therefore we  have the isomorphism $D(\calH(E^*, Q))/\ker(D\Phi_w) \cong U_w(\varmathbb{G}, E^*)$. The highest weight modules of $D(\calH(E^*, Q))$ of weight $w$ factors through $U_w(\varmathbb{G}, E^*)$. The algebra $U_w(\varmathbb{G}, E^*)$ is the main focus of the present paper. This algebra should rather be thought of as the $E^*$-theory analogue of affine $q$-Schur algebra \cite{D}. 

 %=========
 
\section{Quiver with automorphism and folding of affine quantum groups}
We use the Hopkins-Kuhn-Ravenel theorem to define convolution algebra and cohomological Hall algebra of a quiver with automorphism. We then apply it to a special example, i.e., Nakajima cyclic quiver variety. 
\label{sec:QuiverAuto}
\subsection{The definition}
\label{subsec:rep_auto}
Let $\Gamma=(\Gamma_0,\Gamma_1)$ be a quiver. 
An \text{admissible automorphism} of the quiver $\Gamma=(\Gamma_0,\Gamma_1)$ \cite[Chapter 12.1.1]{Lus93} consists of a permutation $a: \Gamma_0 \to \Gamma_0$ and a permutation $a: \Gamma_1\to \Gamma_1$ such that for any $h\in \Gamma_1$, we have $a(\out(h))=\out(a(h))$, $a(\inc(h))=\inc(a(h))$ and such that there is no edge joining two vertices in the same $a$-orbit.
In the present paper for simplicity we assume that $a$ is of order $p$ which is a prime number. 

Let $V$ be a representation of $\Gamma$, then, the automorphism $a$ induces another representation $a^*V$ with $(a^*V)_i=V_{a(i)}$ for any $i\in \Gamma_0$ and similar for the linear maps assigned to $h\in \Gamma_1$. 
For any dimension vector $\gamma\in\bbN^{\Gamma_0}$, and any representation $V$ of dimension $\gamma$, let $a^*\gamma \in \bbN^{\Gamma_0}$ be the dimension vector of $a^*V$. Hence, we get an action of $a$ on $\bbN^{\Gamma_0}$, and decomposes the latter into $a$-orbits. Our assumption on the order of $a$ implies that each orbit is either free or a singleton. Furthermore, for each orbit we fix a system of isomorphisms $b:a^*\bbC^{\gamma}\cong \bbC^{a^*\gamma}$ of $\Gamma_0$-graded vector spaces for each $\gamma\in \bbN^{\Gamma_0}$  so that $b^p=\id$. In particular, if $\gamma$ is fixed by $a$, this gives an order-$p$  linear automorphism $b$ on $\bbC^{\gamma}$ such that $a(V_i)=V_{ai}$ for each $i\in \Gamma_0$. Such a structure induces an action of $b$ on $\sqcup_{\gamma\in\bbN^{\Gamma_0}}\Rep(\Gamma,\gamma)$
which is compatible with the action of $a$ on $\bbN^{\Gamma_0}$. Furthermore, $b$ induces an isomorphism of groups  by conjugation
$\GL_\gamma\cong \GL_{a^*\gamma}$, which is compatible with their actions
\[
\begin{xymatrix}@R=0.3em{
\GL_\gamma \ar[r]^{b} &\GL_{a^*\gamma}\\
&\\
\ar@(ul,ur)[]&\ar@(ul,ur)[]\\
\Rep(\Gamma,\gamma)
 \ar[r]^{b} & \Rep(\Gamma,a^*\gamma)
}
\end{xymatrix}
\]

Assume there is a subset $\Gamma_0''\subset \Gamma_0$ closed under $a$, called the {\em frozen} vertices, with the complement $\Gamma_0'\subseteq \Gamma_0$. 
When drawing a quiver, we draw the frozen vertices with $\square$ and the unfrozen ones with $\circ$.
We consider $\bbQ^{\Gamma_0'}$ as the space of stability conditions for $\GL_{\gamma'}:=\prod_{i\in \Gamma_0'}\GL(V_i)$. Let $\theta\in \bbQ^{\Gamma_0'}$ be a stability condition fixed by $a$. Then, we have an action of $b$ on $(\coprod_{\gamma\in \bbN^{\Gamma_0}}\Rep(\Gamma,\gamma)^{\theta})/\GL_{\gamma'}$. Note that this action is compatible with the action of $a$ on $\bbN^{\Gamma_0}$ and the action of $\GL_{\gamma''}:=\prod_{i\in \Gamma_0''}\GL(V_i)$. 
As in the case of Nakajima quiver varieties, we will also consider the $\GL_{\gamma'}$-quotients of closed subsets  $N_\gamma\subseteq \Rep(\Gamma,\gamma)^{\theta}$ that are invariant under $b$.
For simplicity, we fix $\gamma''\in\bbN^{\Gamma_0''}$ invariant under $a$, and let $N(\gamma'')=\sqcup_{\gamma'\in\bbN^{\Gamma_0'}}N_{(\gamma', \gamma'')}$ which is endowed with the action of $\GL_{\gamma''}$ and an equivariant projective morphism \[
N(\gamma'')\to \cup_{\gamma'\in\bbN^{\Gamma_0'}}\Rep(\Gamma,(\gamma', \gamma''))/\GL_{\gamma'}.
\]
Here $\cup$ is taken in the same sense as \cite[\S~2.5]{Nak}.

We say $\Rep(\Gamma,\gamma)$ (resp. $\Rep(\Gamma,\gamma)^{\theta}/\GL_{\gamma'}$ or $N_\gamma$) is an {\em even component} if $\gamma$ is fixed by $a$ and hence $b$ induces an action of $\langle b \rangle=\bbZ/p$ on this component. We say $\Rep(\Gamma,\gamma)$ or resp. $\Rep(\Gamma,\gamma)^{\theta}/\GL_{\gamma'}$ or $N_{\gamma}$ is an {\em uneven component} if $\gamma$ is not fixed by $a$. 
In each of the cases, we write $|_{\even}$ for the union of the even components. 
In this case we have easily the following.

\begin{lemma}\label{lem:uneven_free}
Notations and terminologies as above, assume $\Rep(\Gamma,\gamma)$ (resp. $\Rep(\Gamma,\gamma)^{\theta}/\GL_{\gamma'}$ or $N_{\gamma}$) is an uneven component, then $b$ induces an action of $\langle b \rangle=\bbZ/p$ on $\sqcup_{\nu\in \langle a^*\rangle \gamma}\Rep(\Gamma,\nu)$ (resp. $\sqcup_{\nu\in \langle a^*\rangle \gamma}\Rep(\Gamma,\nu)^{\theta}/\GL_{\nu'}$ or $\sqcup_{\nu\in \langle a^*\rangle \gamma}N_{\nu}$) which is a free action, where $\langle a^*\rangle \gamma=\{(a^*)^i\gamma\mid i=0, 1, \cdots, p-1\}$. 
\end{lemma}
\begin{proof}
Recall the $b$ is compatible with the $a$-action on $\bbN^{\Gamma_0}$, which in turn is free in the case of uneven components. 
\end{proof}

We would like to consider the cohomological Hall algebras and the convolution algebras in the setting of a quiver with an automorphism, although again the former is not needed for the purpose of the present paper. For this, it suffices to notice that by the discussion above
$$\calH_{\langle b\rangle}(\Gamma,a):=
\oplus_{\gamma\in \bbN^{\Gamma_0}/a}
A^*_{\langle b\rangle\ltimes \prod_{\nu\in \langle a^*\rangle \gamma }\GL_{\nu}}
\big(\sqcup_{\nu\in \langle a^*\rangle \gamma}\Rep(\Gamma,\nu)\big)$$
and the Hall multiplication is well-defined equivariantly. 
Here $\bbN^{\Gamma_0}/a$ consists of representatives of the $a$-orbits. So is the convolution product on $$\calA_{\langle b\rangle}:=A^*_{\langle b \rangle\ltimes \GL_{\gamma''}}\left(N(\gamma'')\times_{\big( \sqcup_{\gamma'\in\bbN^{\Gamma_0'}}\Rep(\Gamma,(\gamma', \gamma''))^{\theta}/\GL_{\gamma'})\big)}N(\gamma'')\right).$$
\begin{definition}
Notations as above, we call $\calH_{\langle b \rangle}(\Gamma,a)\otimes_{A^*_{\langle b \rangle}(\pt)}A^*_{\langle b \rangle}(\pt)[\frac{1}{S}]$ (resp. $\calA_{\langle b \rangle}\otimes_{A^*_{\langle b \rangle}(\pt)}A^*_{\langle b \rangle}(\pt)[\frac{1}{S}]$) endowed with the multiplication induced from that on $\calH_{\langle b \rangle}(\Gamma,a)$ (resp. $\calA_{\langle b \rangle}$) the cohomological Hall algebra (resp. the convolution algebra) of quiver with automorphisms. 
\end{definition}

The above definition is analogue of and motivated by Luszitg's construction of constructible Hall algebra of a quiver with automorphism.
\begin{prop}\label{prop:folding}
Notations as above, the multiplications on $\calH_{\langle b \rangle}(\Gamma,a)$ (resp. $\calA_{\langle b \rangle}$)
and 
 $\calH_{\langle b \rangle}(\Gamma,a)\otimes_{A^*_{\langle b \rangle}(\pt)}A^*_{\langle b \rangle}(\pt)[\frac{1}{S}]$ (resp. $\calA_{\langle b \rangle}\otimes_{A^*_{\langle b \rangle}(\pt)}A^*_{\langle b \rangle}(\pt)[\frac{1}{S}]$) are associative. Furthermore, the natural map  $\calH_{\langle b \rangle}(\Gamma,a)\to \calH_{\langle b \rangle}(\Gamma,a)\otimes_{A^*_{\langle b \rangle}(\pt)}A^*_{\langle b \rangle}(\pt)[\frac{1}{S}]$ (resp. $\calA_{\langle b \rangle}\to \calA_{\langle b \rangle}\otimes_{A^*_{\langle b \rangle}(\pt)}A^*_{\langle b \rangle}(\pt)[\frac{1}{S}]$) is a map of  associative algebras which sends the uneven components to zero. 
\end{prop}
\begin{proof}
One observes that in the proof of associativity in either case \cite[\S~2.7.18]{CG} (see also Proposition~\ref{prop:conv_action} for the statement in the present generality) and \cite[Theorem~4.1]{YZ1}, the proof carries through $\langle b \rangle$-equivariantly and hence defines $A^*_{\langle b \rangle}$-algebras. Change of scalars with respect to the central subalgebra $A^*_{\langle b \rangle}$ preserves associativity. 

The statement about uneven components follows directly from Lemmas~\ref{lem:cohomology_free_action} and \ref{lem:uneven_free}.
\end{proof}

Examples are to be given  in \S~\ref{sec:mod}, where the algebra is identified explicitly.

\subsection{Two examples}
\label{sub:two exas}
Let $Q=(I, H)$ be a quiver. 
Let $\mu_l$ be the set of $l$-th roots of $1$ and $\zeta\in \mu_l$ be the  primitive $l$-th root of unity. 

We introduce two quivers with automorphism associated to $Q=(I, H)$. 
The first one is 
\begin{equation}\label{quiver_l}
Q_l=(I\times \mu_l,H_l,a),
\end{equation}
where 
\begin{enumerate}
\item the set of vertices of $Q_l$ is $I_l=I\times \mu_l=\{(i, r) \mid i\in I, r\in \mu_l\}$,
\item the set of edges is 
$H_l=\{ h_r: (i,r) \to (j,s) \mid  \text{if there is an edge $i \xrightarrow{h} j$ in $H$ and $r=\zeta^{1+m(h)} s$} \}$, where $m(h)+1$ is the weight of $h$.  
\item 
the automorphism is given as
$a: I_l\to I_l, (i, r) \mapsto (i, \zeta r)$, $a: H_l\to H_l, h_r\mapsto h_{\zeta r}$ 
\end{enumerate}
The second quiver with automorphism is 
\[
Q\times \mu_l=(I \sqcup \cdots \sqcup I, H \sqcup \cdots \sqcup H, a),  
\]
where 
\begin{enumerate}
\item the set of vertices of $Q\times \mu_l$ is $I \sqcup \cdots \sqcup I=I\times \mu_l=\{(i, r) \mid i\in I, r\in \mu_l\}$,
\item the set of edges is 
$H \sqcup \cdots \sqcup H=\{ h_r: (i,r) \to (j,s) \mid  \text{if there is an edge $i \xrightarrow{h} j$ in $H$ 
and $r=s$} \}$. 
\item 
the automorphism is given as
$a: (i, r) \mapsto (i, \zeta r)$, $a: h_r\mapsto h_{\zeta r}$ 
\end{enumerate}
The following pictures illustrate the two quivers with automorphism.
Let $Q$  be the following quiver. 
\[
\begin{tikzpicture}
  \draw ($(0,0)$) circle (.1);
   \draw ($(1,0)$) circle (.1);
  \draw [->] (0.1, 0) -- (0.9, 0);
   \filldraw (0,0) node[anchor=north, yshift=0.3cm, xshift=-1cm] {$Q:$};
   \filldraw (0,0) node[anchor=north, yshift=0.7cm, xshift=0cm] {$i$};
   \filldraw (1,0) node[anchor=north, yshift=0.7cm, xshift=0cm] {$j$};
    \filldraw (0.5,0) node[anchor=north, yshift=0.5cm, xshift=0cm] {$h$};
\end{tikzpicture}
\]
We then have
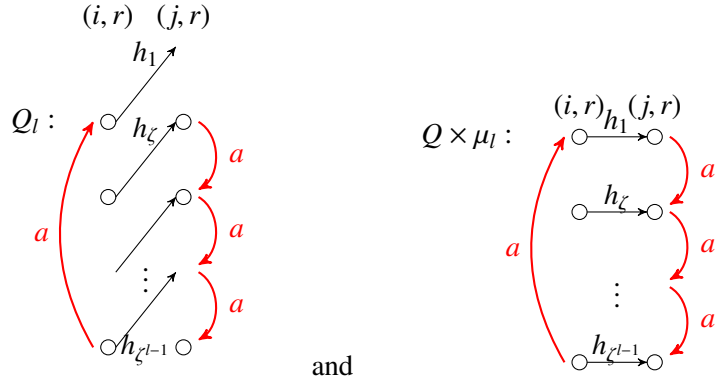
\[
\begin{tikzpicture}

  \draw ($(0,0)$) circle (.1);
   \draw ($(1,0)$) circle (.1);
  \foreach \x in {0,-1, -2, -3}{
	  \draw [->] (0.1, \x) -- (0.9, \x+1);
	}
  
    \draw ($(0,-1)$) circle (.1);
   \draw ($(1,-1)$) circle (.1);
      	\node at (0.5, -2) {$\vdots$};  
    \draw ($(0,-3)$) circle (.1);
   \draw ($(1,-3)$) circle (.1);	
   \filldraw (0,0) node[anchor=north, yshift=0.3cm, xshift=-1cm] {$Q_l:$};
   \filldraw (0,0) node[anchor=north, yshift=1.7cm, xshift=0cm] {$(i, r)$};
   \filldraw (1,0) node[anchor=north, yshift=1.7cm, xshift=0cm] {$(j, r)$};
    \filldraw (0.5,0) node[anchor=north, yshift=1.2cm, xshift=0cm] {$h_1$};
        \filldraw (0.5,-1) node[anchor=north, yshift=1.2cm, xshift=0cm] {$h_{\zeta}$};
                \filldraw (0.5,-3) node[anchor=north, yshift=0.3cm, xshift=0cm] {$h_{\zeta^{l-1}}$};
                  \foreach \x in {0,-1, -2}{                
                  \draw [->, red, out=-30, in=30, thick] (1.2, \x) to node[xshift=8]{$a$} (1.2,\x-0.9);}
                  \draw [->, red, out=180-60, in=180+60, thick] (-0.2, -3) to node[xshift=-7]{$a$} (-0.2, 0);
\end{tikzpicture}
\qquad\hbox{and}\qquad
\begin{tikzpicture}
%\foreach \x in {0,...,4}{
%	  \draw ($(\x,-0.1)$) circle (.1);
%	}
  \draw ($(0,0)$) circle (.1);
   \draw ($(1,0)$) circle (.1);
  \draw [->] (0.1, 0) -- (0.9, 0);
    \draw ($(0,-1)$) circle (.1);
   \draw ($(1,-1)$) circle (.1);
     \draw [->] (0.1, -1) -- (0.9, -1);
      	\node at (0.5, -2) {$\vdots$};  
    \draw ($(0,-3)$) circle (.1);
   \draw ($(1,-3)$) circle (.1);
     \draw [->] (0.1, -3) -- (0.9, -3);	
   \filldraw (0,0) node[anchor=north, yshift=0.3cm, xshift=-1.5cm] {$Q\times \mu_l:$};
   \filldraw (0,0) node[anchor=north, yshift=0.7cm, xshift=0cm] {$(i, r)$};
   \filldraw (1,0) node[anchor=north, yshift=0.7cm, xshift=0cm] {$(j, r)$};
    \filldraw (0.5,0) node[anchor=north, yshift=0.5cm, xshift=0cm] {$h_1$};
        \filldraw (0.5,-1) node[anchor=north, yshift=0.5cm, xshift=0cm] {$h_{\zeta}$};
                \filldraw (0.5,-3) node[anchor=north, yshift=0.5cm, xshift=0cm] {$h_{\zeta^{l-1}}$};
                \foreach \x in {0,-1, -2}{                
                  \draw [->, red, out=-30, in=30, thick] (1.2, \x) to node[xshift=8]{$a$} (1.2,\x-0.9);}
                  \draw [->, red, out=180-60, in=180+60, thick] (-0.2, -3) to node[xshift=-7]{$a$} (-0.2, 0);
\end{tikzpicture}\]

Let $v=(v_{i, r})_{i\in I, r\in \mu_l} \in \N^{I\times \mu_l}$ be a dimension vector of $Q_l$. The dimension vector $a^*v$ on the vertex $a(i, r)=(i, \zeta r)$ is given as $(a^*v)_{a(i, r)}:=v_{i, r}$. 
For an $I\times \mu_l$-graded vector space $V$ with dimension vector $v$, 
let $a^*V$ be the vector space with dimension vector $a^*v$. That is, $a^*V$ is isomorphic to $V$ as vector spaces, but it has a different grading than that on $V$. For simplicity, we write $V_\zeta$ as the $I$-graded vector space obtained as  the direct summand of $V$ with the $\mu_l$-component $\zeta$.

The automorphism $a$ gives the pullback 
\[
a^*: \Rep(Q_l, v) \to \Rep(Q_l, a^*v), 
(V_{\out(h), r} \xrightarrow{x} V_{\inc(h), s})
\mapsto 
(a^*V_{\out(h), r} \xrightarrow{x} a^*V_{\inc(h), s}). 
\]

%==========================
%==========================
\subsection{Fixed point subvariety of the quiver variety}
Let $\mathfrak{M}(w)$ be the Nakajima quiver variety defined in \S~\ref{subsec:quiver variety}. Assume the framing vector $w=(w^i)_{i\in I}$ is divisible by $p$, that is, $p|w^i$ for each $i\in I$. 
Fix a group homomorphism
\begin{equation}\label{map:alpha}
\alpha: \Z/p\to \varmathbb{G}=\C^* \times \GL_{w}, 
\alpha(1)=(\zeta, g_0), 
\end{equation}
such that
\begin{itemize}
\item  $\zeta$ is a primitive $p$-th root of $1$, 
\item the eigenvalues of $g_0$ are all the $p$-th roots of $1$ and each has multiplicity $w/p$.
\end{itemize} 
One can write $g_0$ as the following diagonal matrix with $\id$ the identity matrix of size $w/p$
\[g_0=\begin{bmatrix}
\id &0& 0&\cdots &0\\
0& \zeta \id & 0&\cdots &0\\
0& 0& \zeta^2 \id&\cdots &0\\
&\vdots && \ddots&\vdots\\
0& 0&  0 & \cdots& \zeta^{p-1} \id \end{bmatrix}\]

Let $x\in \mathfrak{M}(w)^{\im \alpha}$. 
Take a representative $(B, i, j) \in \mu^{-1}(0)^s$ of $x$. 
We then have \cite[Section 4]{Nak}
\begin{align}
 &a \star (B, i, j)=\rho(a)^{-1} \cdot (B, i, j),\,\  \text{for any $a\in \im \alpha$} \label{eq:fix}
% &
% \text{That is} \,\ 
% (\zeta \cdot B, \zeta i \circ g_0^{-1}, \zeta g_0\circ j)
% =(\rho(a)^{-1} B \rho(a),  \rho(a)^{-1}\circ i, j \circ \rho(a)). 
\end{align}
 The map $a\mapsto  \rho(a)$  is a homomorphism as the action of $\GL_V$ on $\mu^{-1}(0)^s$ is free. 
 Following \cite{Nak}, we have the decomposition
 \[
 \mathfrak{M}(w)^{\im \alpha}=\coprod_{\rho}  \mathfrak{M}[\rho],
 \]
 where $\rho$ runs the set of homomorphisms $\im \alpha\to \GL(v)$ (with various $v$) and $\mathfrak{M}(\rho)$ is a union of connected components of $ \mathfrak{M}(w)^{\im \alpha}$. 
Note that $\mathfrak{M}(\rho)$ depends only on the $\GL(v)$-conjugacy class of $\rho$.

For the fixed homomorphism $\alpha$ \eqref{map:alpha}, let 
\[b=\begin{bmatrix}
0&\id& 0&\cdots & 0 \\
0&0& \id&\cdots & 0 \\
&\vdots&& &\vdots\\
0&0& 0&\cdots & \id \\
\id&0& 0&\cdots & 0 \\
\end{bmatrix}
\]
be the block matrix where $\id$ is the identity matrix with size $w/p$. In particular, $b$ gives an identification 
\[
b: W_{\zeta^s}\xrightarrow{\cong} W_{\zeta^{s+1}}
\]
of the $g_0$-eigenspaces. 
The conjugation $bg_0 b^{-1}$ is given by $\zeta g_0$, whose eigenvalues
% \[bg_0 b^{-1}=\begin{bmatrix}
% \zeta\id &0& 0&\cdots &0\\
% 0&  \zeta^2\id & 0&\cdots &0\\
% 0& 0&  \zeta^3\id&\cdots &0\\
% &\vdots && \ddots&\vdots\\
% 0& 0&  0 & \cdots&  \id \end{bmatrix}=\zeta g_0\]
% The eigenvalues of 
%$bg_0 b^{-1}=\zeta g_0$ 
are obtained by a cyclic permutation of that of $g_0$. 
% The $g_0$-eigenspace $W_{\zeta^s}$ with eigenvalue $\zeta^s$ becomes the $bg_0 b^{-1}$-eigenspace with eigenvalue $\zeta^{s+1}$. 

Let $\langle b\rangle\cong\Z/p$ be the cyclic group generated by $b$. 
Then, $\langle b\rangle$ is a subgroup in the normalizer of $C(\im \alpha)\subset \GL_w\times \C^*$. 
Indeed, $C(\im \alpha)=\prod_{s=0}^{p-1}(\GL_{w/p})\times \C^*$. 
The action of $b$ on an element $(A, t) \in C(\im \alpha)$ is given by
\[
\Big(bAb^{-1}=\begin{bmatrix}
A_1 &0& 0&\cdots &0\\
0& A_2 & 0&\cdots &0\\
0& 0& A_3&\cdots &0\\
&\vdots && \ddots&\vdots\\
0& 0&  0 & \cdots& A_{0} \end{bmatrix}, t\Big), \text{where $A=\begin{bmatrix}
A_0 & 0& 0&\cdots &0\\
0& A_1 & 0&\cdots &0\\
0& 0& A_2&\cdots &0\\
&\vdots && \ddots&\vdots\\
0& 0&  0 & \cdots& A_{p-1} \end{bmatrix}$}
\]
We now show there is an action of $\langle b\rangle \ltimes C(\im \alpha)$ on the fixed points $\mathfrak{M}(w)^{\im \alpha}$. 

\begin{lemma}\label{lem:C^* trivial action}
Let $\C^*\subset \GL_{w}, t\mapsto (t\id_{w})$ be the diagonal $\C^*$ which lies in the center of $\GL_{w}$. The action of this $\C^*$ on $\mathfrak{M}(w)$ is trivial. 
\end{lemma}
\begin{proof}
For any $(B_h, i, j)\in \mathfrak{M}(w)$, $t\in \C^*$, we have
\[
t \star (B_h, i, j)=(B_h, t^{-1}i, tj)=(t^{-1}\id_{V})\cdot (B_h, i, j)
\]
Thus, $(B_h, i, j)$ is fixed under the action of $t$. 
\end{proof}

\begin{lemma}\label{lem:inv of M}
We have the equality 
$\mathfrak{M}(w)^{\im \alpha}=\mathfrak{M}(w)^{\im b\alpha b^{-1}}$. 
\end{lemma}
\begin{proof}
By Lemma \ref{lem:C^* trivial action}, the action of $\xi\id_{w}$ on $\mathfrak{M}(w)$ is trivial. 
The claim follows from the following identity
\[
\mathfrak{M}(w)^{\im b\alpha b^{-1}}=\mathfrak{M}(w)^{(\zeta, bg_0 b^{-1})}
=\mathfrak{M}(w)^{(\zeta, \zeta g_0)}=\mathfrak{M}(w)^{(\zeta, g_0)}=\mathfrak{M}(w)^{\im \alpha}. 
\]
\end{proof}
As a consequence, the action of $\langle b\rangle$ on $\mathfrak{M}(w)^{\im \alpha}$ is given by 
\[
 b: \mathfrak{M}(w)^{\im \alpha}\to \mathfrak{M}(w)^{\im b\alpha b^{-1}}=\mathfrak{M}(w)^{\im \alpha}.
\]

By $C(\im \alpha)=C(\im b\alpha b^{-1})$ and Lemma \ref{lem:inv of M}, 
we have the following commutative diagram
\begin{equation}
\label{eq:action of b}
\xymatrix{
C(\im \alpha) \times \mathfrak{M}(w)^{\im\alpha} \ar[r] \ar[d]^{(b(-)b^{-1}, b)}& \mathfrak{M}(w)^{\im\alpha} \ar[d]^{b}\\
C(\im \alpha ) \times \mathfrak{M}(w)^{\im \alpha } \ar[r]& \mathfrak{M}(w)^{\im  \alpha }\\
}
 \,\ \text{given by}\,\
 \xymatrix{
 (a, x)\ar@{|->}[r] \ar@{|->}[d] & ax \ar@{|->}[d]\\
 (bab^{-1}, bx) \ar@{|->}[r]& bax\\
 }
\end{equation}
Thus, we have an action
\[
\big(\langle b\rangle\ltimes C(\im \alpha) \big)\times \mathfrak{M}(w)^{\im\alpha}  \to \mathfrak{M}(w)^{\im\alpha}. 
\]

\subsection{Cyclic quiver variety and the fixed point subvariety}
In this section, we recall the \textit{cyclic} quiver variety defined in \cite[Section 4]{Nak04}. 
Let $Q$ be a quiver, and  $\overline{Q^{\heartsuit}}$ the doubled framed quiver, we have the quiver with automorphism 
$(\overline{Q^{\heartsuit}})_l$ as in \eqref{quiver_l}.  
Let $V, W$ be the $I\times \mu_l$-graded vector spaces with dimension vectors 
$(v, w)=(v_{i, r}, w_{i, r})_{i\in I, r\in \mu_l}$. Denoted by $V_{i, r}$ its $i\times r$-component. 
Define vector spaces
\begin{align*}
&L^{\bullet}(V, W)^{[n]}:=\oplus_{i\in I, r\in \mu_l} \Hom(V_{i, r}, W_{i, r\zeta^n})\\
&E^{\bullet}(V, W)^{[n]}:=\oplus_{h\in H, r\in \mu_l} \Hom(V_{\out(h), r}, W_{\inc(h), r\zeta^n})\\
& M^{\bullet}(V, W):=
E^{\bullet}(V, V)^{[-1]}\oplus L^{\bullet}(W, V)^{[-1]}\oplus 
L^{\bullet}(V, W)^{[-1]}. 
\end{align*}
Define the moment map $\mu_l:  M^{\bullet}(V, W)\to L^{\bullet}(V, V)^{[-2]}$ by
\[
(\mu_l)_{k, r}(B, i, j)\mapsto \sum_{\inc(h)=k}\epsilon(h) B_{h, r\zeta^{-1}} B_{h^*, \zeta}+i_{k, r\zeta^{-1}} j_{k, r}, 
\]
where $(\mu_l)_{k, r}$ is the $(k, r)$-component of $\mu_l$. The group $\prod_{k, r}\GL(V_{k, r})$ acts on $M^{\bullet}$ by conjugation as follows
\[
(B, i, j)\mapsto 
g\cdot (B, i, j)=
(g_{\inc(h), r\zeta^{-1}}B_{h, r}g_{\out(h), r}^{-1}, g_{k, a\zeta^{-1}i_{k, r}}, j_{k, r}g_{\out(h), r}^{-1}). 
\]
Define the cyclic quiver variety to be the quotient
\[
\mathfrak{M}^{\cyc}(v, w):=(\mu_l)_{v, w}^{-1}(0)^{s}/\prod_{k, r}\GL(V_{k, r}). 
\]

Recall $\mathfrak{M}[\rho]\subset \mathfrak{M}^{\im \alpha}$. We regard $V$ as an $\im \alpha$-module via $\rho: \im \alpha\to \GL(V)$ and consider the weight space decomposition of $V$ as
\[
V=\oplus_{s} V_{\zeta^s}, 
\text{where $V_{\zeta^s}:=\{v\in V\mid \rho(1) v=\zeta^s v\}$}. 
\]
We say a component of $\mathfrak{M}[\rho]$ is an even component if
\[
\dim(V_{\zeta^0})=\dim(V_{\zeta^s})=\cdots=\dim(V_{\zeta^{p-1}}).
\]

The following is basically from \cite[\S~4.1]{Nak}, which we briefly review so that we can keep track of the quiver automorphism.
\begin{prop}\label{prop:Nakajima}
(1). 
There is an isomorphism of varieties
 \[
 \mathfrak{M}^{\im \alpha} \cong  \mathfrak{M}^{\cyc}(v, w).
\] Furthermore, the above isomorphism commutes with the actions of $C(\im \alpha)\cong \GL(w)\times\bbC^*$, where $C(\im \alpha)\subset \prod_{i} \GL(\oplus_{r\in \mu_l} W_{i, r})\times \C^*$ is the centralizer of $\im \alpha$ and $\GL(w)=\prod_{i, r} \GL(W_{i, r})$, and sends an even component on one side to an even component on the other side and hence induces $\mathfrak{M}^{\im \alpha}|_{\even} \cong  \mathfrak{M}^{\cyc}(v, w)|_{\even} $. 

(2). The isomorphism $ \mathfrak{M}^{\im \alpha} \cong  \mathfrak{M}^{\cyc}(v, w)$ is compatible with the $\langle b\rangle$-actions on both sides. 
\end{prop}
Here recall that $\mathfrak{M}^{\im \alpha}|_{\even}$ is the union of all the even components of $\mathfrak{M}^{\im \alpha}$. 
\begin{proof}
We prove (1). Recall we have the decomposition  $\mathfrak{M}(w)^{\im \alpha}=\coprod_{\rho}  \mathfrak{M}[\rho]$. Take a representative $(B, i, j) \in \mu^{-1}(0)^s$ of $x\in \mathfrak{M}[\rho]$.
Let  
\[
W=\oplus_{s} W_{\zeta^s}, 
\text{where $W_{\zeta^s}:=\{w\in W\mid g_0 w=\zeta^s w\}$}.  
\] be the weight space decomposition of $W$ under the action of $\im \alpha$. 
 
We have the decomposition of $V$ as
\[
V=\oplus_{s} V_{\zeta^s}, 
\text{where $V_{\zeta^s}:=\{v\in V\mid \rho(1) v=\zeta^s v\}$}. 
\]
The triple $(B, i, j)$ satisfies \eqref{eq:fix} if and only if 
$i$ maps $W_{\zeta^s}$ to $V_{\zeta^{s-1}}$, 
$j$ maps $V_{\zeta^s}$ to $W_{\zeta^{s-1}}$, and 
$B_h$ maps $V_{\zeta^s}^{(\out(h))}$ to $V_{\zeta^{s-(m(h)+1)}}^{(\inc(h))}$. 
% Indeed, for $w\in W_{\zeta^s}$, we have 
% $\zeta i \circ g_0^{-1} (w)=\zeta\zeta^{-s} i(w)=\rho(a)^{-1} i(w)$. 
% Thus, $i(w)\in V_{\zeta^{s-1}}$. 
We illustrate this condition (of $i, j$) in the following picture. 
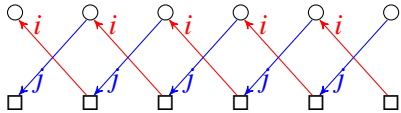
\[
\begin{tikzpicture}
%\foreach \x in {-1,...,4}{
%	  \draw ($(\x,-0.1)$) circle (.1);
%	}
\foreach \x in {-1,...,4}{
	  \draw ($(\x,1.1)$) circle (.1);
	}
	\foreach \x in {-1,...,4}{	
	    \filldraw (\x,-0.1) node {$\square$};
	    }
	
\foreach \x in {0,...,4}{
	   \draw [red, ->] (\x-0.05, 0) -- (\x-0.95, 1);
	}
\foreach \x in {-1,...,3}{	
	    \filldraw (\x+0.5,0.5) node[anchor=north, yshift=0.7cm, xshift=-0.2cm, red] {$i$};
	    }
\foreach \x in {-1,...,3}{
	   \draw [blue, <-] (\x+0.05, 0) -- (\x+0.95, 1);
	}
	\foreach \x in {-1,...,3}{	
	    \filldraw (\x+0.5,0.5) node[anchor=north, yshift=0cm, xshift=-0.2cm, blue] {$j$};
	    }
\end{tikzpicture}
\]
As a consequence, the restriction of $(B, i, j)$ on the weight spaces gives an element in $\mathfrak{M}^{\cyc}(v, w)$. 
% On the other hand, for $((B_h)_{r}, i_r, j_r)\in \mathfrak{M}^{\cyc}(v, w)|_{\even}$, 
% the direct sums 
% \[
% \oplus_r(B_h)_r: \oplus_r V_{\out(h), r} \to \oplus_r V_{\inc(h), r},\,\
% \oplus_r i_r: W_{i, r} \to \oplus_r V_{i, r}, \,\ 
% \oplus_r j_r: V_{i, r} \to \oplus_r W_{i, r}
% \]
% give an element in $\mathfrak{M}[\rho]$, 
% where $\rho(1)$ is the operator in $\prod_i\GL(\oplus_r V_{i, r})$, such that $V_{i, r}$ is the eigenspace with eigenvalue $r$. 

Now we prove (2). 
We analyze the quiver automorphism $a$ and the action of $b$ on $\mathfrak{M}(w)^{\im \alpha}$. 
For $b\in \GL(W)$, the action of $b$ on $(B_h, i, j)\in \mathfrak{M}[\rho]$ is given as 
\[
b\star (B_h, i, j)=(B_h, i \circ b^{-1}, b\circ j)=b_V\cdot (B_h, i \circ b^{-1}, b\circ j)=(b_VB_hb_V^{-1}, b_V i  b^{-1}, b j  b_V^{-1}), 
\]
where $b_V: \oplus_{s=0}^{p-1} V_{\zeta^s}\to  \oplus_{s=0}^{p-1} V_{\zeta^{s+1}}$ is the cyclic permutation 
\[
b_V=\begin{bmatrix}
0&\id_{V_{\zeta}}& 0&\cdots & 0 \\
0&0& \id_{V_{\zeta^2}}&\cdots & 0 \\
&\vdots&& &\vdots\\
0&0& 0&\cdots & \id_{V_{\zeta^{p-1}}} \\
\id_{V_1}&0& 0&\cdots & 0 \\
\end{bmatrix},
\] which permutes the direct summands as follows. 
\[
\begin{tikzpicture}
   \filldraw (0,0) node {$V_1$};
    \filldraw (1,0) node {$V_{\zeta}$};
    \filldraw (2,0) node {$\cdots$};
      \filldraw (3,0) node {$V_{\zeta^{p-1}}$};
       \filldraw (0,-1) node {$V_{\zeta}$};
    \filldraw (1,-1) node {$V_{\zeta^2}$};
    \filldraw (2,-1) node {$\cdots$};
      \filldraw (3,-1) node {$V_{1}$};
  \foreach \x in {0,1, 2}{
	 \filldraw (\x+0.5,0) node {$\oplus$};
	  \filldraw (\x+0.5,-1) node {$\oplus$};
	}
   \foreach \x in {1, 2, 3}{
     \draw [->] (\x, -0.2) -- (\x-1, -0.8);
	} 
	     \draw [dashed, ->] (0, -0.2) -- (3, -0.7);

\end{tikzpicture}
\]
For the morphism $i: \oplus_{s=0}^{p-1} W_{\zeta^s} \to \oplus_{s=0}^{p-1} V_{\zeta^s}$ such that $i(W_{\zeta^s})\subset V_{\zeta^{s-1}}$. The conjugation $b_V i  b^{-1}$ is given as 
\[
\begin{tikzpicture}
   \filldraw (0,2) node {$W_{\zeta}$};
    \filldraw (1,2) node {$W_{\zeta^2}$};
    \filldraw (2,2) node {$\cdots$};
      \filldraw (3,2) node {$W_{\zeta^{1}}$};

   \filldraw (0,1) node {$W_1$};
    \filldraw (1,1) node {$W_{\zeta}$};
    \filldraw (2,1) node {$\cdots$};
      \filldraw (3,1) node {$W_{\zeta^{p-1}}$};

   \filldraw (0,0) node {$V_1$};
    \filldraw (1,0) node {$V_{\zeta}$};
    \filldraw (2,0) node {$\cdots$};
      \filldraw (3,0) node {$V_{\zeta^{p-1}}$};
      
     \filldraw (0,-1) node {$V_{\zeta}$};
    \filldraw (1,-1) node {$V_{\zeta^2}$};
    \filldraw (2,-1) node {$\cdots$};
      \filldraw (3,-1) node {$V_{1}$};

  \foreach \x in {0,1, 2}{
	 \filldraw (\x+0.5,0) node {$\oplus$};
	  \filldraw (\x+0.5,-1) node {$\oplus$};
	   \filldraw (\x+0.5,1) node {$\oplus$};
	    \filldraw (\x+0.5,2) node {$\oplus$};
	}
   \foreach \x in {0, 1, 2}{
     \draw [->] (\x, 1.8) -- (\x+1, 1.2);}
       \draw [dashed, ->] (3, -0.2+2) -- (0, -0.7+2);
       
        \foreach \x in {1, 2, 3}{
     \draw [->] (\x, -0.2+1) -- (\x-1, -0.8+1);
          \draw [->] (\x, -0.2+1) -- (\x-1, -0.8+1);
	} 
   \foreach \x in {1, 2, 3}{
     \draw [->] (\x, -0.2) -- (\x-1, -0.8);
	} 
	     \draw [dashed, ->] (0, -0.2) -- (3, -0.7);
  \draw [dashed, ->] (0, -0.2+1) -- (3, -0.7+1);
 \filldraw (-1,1.5) node {$b^{-1}$};
 \filldraw (-1,0.5) node {$i$};
 \filldraw (-1,-0.5) node {$b_V$};
\end{tikzpicture}
\]
Therefore, we have $b_V i  b^{-1}=a^*(i)$. Similar equality holds for $j$ and $B_h$. 
Therefore, we have, $a^*(B_h, i, j)=b\star (B_h, i, j)$.
\end{proof}
% Let $\alpha: \Z/p\to G$ be a group homomorphism. 
% For a $G$-variety $X$, the centralizer 
% $C(\im \alpha)$ acts on the fixed points $X^{\im \alpha}$. Furthermore, for any $g\in G$, there is a commutative diagram
% \begin{equation}\label{com:g}
% \xymatrix{
% C(\im \alpha) \times X^{\im\alpha} \ar[r] \ar[d]^{(g(-)g^{-1}, g)}& X^{\im\alpha} \ar[d]^{g}\\
% C(\im g\alpha g^{-1}) \times X^{\im g\alpha g^{-1}} \ar[r]& X^{\im g \alpha g^{-1}}\\
% }
% \,\ \text{given by}\,\
% \xymatrix{
% (a, x)\ar@{|->}[r] \ar@{|->}[d] & ax \ar@{|->}[d]\\
% (gag^{-1}, gx) \ar@{|->}[r]& gax\\
% }
% \end{equation}

\subsection{Quiver varieties associated to $(\overline{Q^{\heartsuit}})_l$ and $\overline{Q^{\heartsuit}}\times \mu_l$}
In this subsection, we compare the cyclic quiver varieties associated to the two quivers with automorphisms from Example \ref{sub:two exas}. 

In $\mathfrak{M}^{\cyc}_{(\overline{Q^{\heartsuit}})_l}(w)$, we consider the following locus consisting of those
$
(B_h, i, j)=((B_h)_s, i_s, j_s)_{s=0, \cdots,p-1}
$
with
\[
(B_h)_s: V_{\zeta^s}^{\out(h)} \to V_{\zeta^{s-(m(h)+1)}}^{\inc(h)}, 
i_s: W_{\zeta^s}^{(k)} \to V_{\zeta^{s-1}}^{(k)}, 
j_s: V_{\zeta^s}^{(k)} \to W_{\zeta^{s-1}}^{(k)}, 
\]
satisfy the commutative diagrams
\[
\xymatrix{
V_{\zeta^s}^{\out(h)} \ar[r]^{(B_h)_s}\ar[d]^{b} & V_{\zeta^{s-(m(h)+1)}}^{\inc(h)} \ar[d]^{b}
\\
V_{\zeta^{s+1}}^{\out(h)} \ar[r]^{(B_h)_{s+1}} & V_{\zeta^{s-m(h)}}^{\inc(h)}
}
\,\ \,\ 
\xymatrix{
W_{\zeta^s}^{(k)} \ar[r]^{i_s}\ar[d]^{b} & V_{\zeta^{s-1}}^{(k)} \ar[d]^{b}
\\
W_{\zeta^{s+1}}^{(k)} \ar[r]^{i{s+1}} & V_{\zeta^{s}}^{(k)}
}
\,\ \,\ 
\xymatrix{
V_{\zeta^s}^{(k)} \ar[r]^{j_s}\ar[d]^{b} & W_{\zeta^{s-1}}^{(k)} \ar[d]^{b}
\\
V_{\zeta^{s+1}}^{(k)} \ar[r]^{j_{s+1}} & W_{\zeta^{s}}^{(k)}
}
\]
We denote this locus by $\mathfrak{M}^{\cyc}_{(\overline{Q^{\heartsuit}})_l}(w)|_{\Delta}$. In other words, it is the $\langle b \rangle$-fixed points. 
It is straightforward to verify that the diagonal $\GL(w/p)\subset \prod_p \GL(w/p)$ and $\C^*$ act on $\mathfrak{M}^{\cyc}_{(\overline{Q^{\heartsuit}})_l}(w)|_{\Delta}$. 

It is clear that $
\mathfrak{M}^{\cyc}_{\overline{Q^{\heartsuit}}\times \mu_l}(w)\cong \prod_{p} \mathfrak{M}(w/p)$. 
There is an diagonal embedding 
$\mathfrak{M}(w/p) \xrightarrow{\Delta} \prod_{p} \mathfrak{M}(w/p)$. 

\begin{prop}\label{lem:iso to products}
\begin{enumerate}
\item
We have
\begin{align*}
\mathfrak{M}^{\cyc}_{(\overline{Q^{\heartsuit}})_l}(w)|_{\Delta}
\cong
\mathfrak{M}^{\cyc}_{\overline{Q^{\heartsuit}}\times \mu_l}(w)|_{\Delta}
\cong \mathfrak{M}(w/p)
. 
\end{align*}
\item
The above isomorphisms are compatible with the action of 
$\GL(w/p)\times \C^*$. 
\end{enumerate}
\end{prop}
\begin{proof}
We prove (1). Define a map 
\[
\mathfrak{M}^{\cyc}_{(\overline{Q^{\heartsuit}})_l}(w)|_{\Delta}
\to
\prod_{p} \mathfrak{M}(w/p)|_{\Delta}. 
\] by assigning to $((B_h)_s, i_s, j_s)$ the following compositions
\begin{align}
& 
V_{\zeta^{s+1}}^{(k)} \xrightarrow{j_{s+1}} W_{\zeta^{s}}^{(k)}, \,\ \,\
W_{\zeta^{s}}^{(k)}
\xrightarrow{i_s}
V_{\zeta^{s-1}}^{(k)}
\xrightarrow{b^{2}}
V_{\zeta^{s+1}}^{(k)}, \,\ \,\
V_{\zeta^{s+1}}^{(\out(h))} \xrightarrow{(B_h)_{s+1}} V_{\zeta^{s-m(h)}}^{(\inc(h))}
\xrightarrow{b^{m(h)+1}}
V_{\zeta^{s+1}}^{(\inc(h))}. \label{eq:comp}
\end{align}
It is straightforward to see that this gives an isomorphism. 

Now we prove  (2). Recall the action of $\GL(w/p)\times \C^*$ on $\mathfrak{M}^{\cyc}_{(\overline{Q^{\heartsuit}})_l}(w)$ is given as 
\[
(B_h, i, j)\mapsto  (t^{m(h)+1} B_h, t   i g, t g^{-1}j), \text{where $(g, t)\in \GL(w/p)\times \C^*$}.
\] 
We have the following diagram of the actions 
\[
\begin{xymatrix}
{
((g, t),  (B_h, i, j)) \ar@{|->}[r]  \ar@{|->}[d] & ((g, t), (b^{m(h)+1}B_h , b^2 i , j))
\ar@{|->}[d] \\
(t^{m(h)+1} B_h, t   i g, t g^{-1}j)
\ar@{|->}[r] & 
(t^{m(h)+1}b^{m(h)+1}B_h , t b^2 i g, t g^{-1}j)
}
\end{xymatrix}
\]
As a consequence, the isomorphism in (1) is compatible with respect to the $\GL(w/p)\times \C^*$ action. 
\end{proof}

\subsection{Restriction to the diagonal}
Let $i_\Delta:\mathfrak{M}(w/p)\to \mathfrak{M}(w)^{\im \alpha}$ be the embedding, obtained via composing the isomorphism from Proposition~\ref{lem:iso to products} with the embedding $\mathfrak{M}(w)^{\im \alpha}|_\Delta\subseteq \mathfrak{M}(w)^{\im \alpha}$, which without causing confusions we also refer to as the diagonal embedding.
Let $H$ be a finite subgroup of $\GL(w/p)\times\bbC^*$. Note that the action of $H$ on $\fM(w)^{\im \alpha}$ comes from the composition of the isomorphism in  Proposition~\ref{lem:iso to products} and that in Proposition~\ref{prop:Nakajima}. Explicitly this comes from the diagonal embedding of $\GL(w/p)\times\bbC^*\inj \GL(w)\times\bbC^*$ with image landing in $C(\im \alpha)\cong \prod_p \GL(w/p)\times\bbC^*$. Recall there is an action of $\langle b\rangle\cong\Z/p$ on $\fM(w)^{\im \alpha}$.
\begin{lemma}\label{lem:diag_power}
The action of $\langle b\rangle$ is free on the complement of $\mathfrak{M}(w)^{\im \alpha}|_\Delta$ in $\mathfrak{M}(w)^{\im \alpha}$, denoted by $\mathfrak{M}(w)^{\im \alpha}\setminus\Delta$. In particular, 
\begin{enumerate}
    \item pulling back  $i_\Delta^*:E^*_{\langle b\rangle\times H}(\mathfrak{M}(w)^{\im \alpha})\to E^*_{\langle b\rangle\times H}(\mathfrak{M}(w)^{\im \alpha}|_\Delta)$ becomes an isomorphism after applying $-\otimes_{E^*_{\langle b\rangle}}\Phi_1$;
    \item similarly, $i_\Delta^*:E^*_{\langle b\rangle\times H}(Z(w)^{\im \alpha})\to E^*_{\langle b\rangle\times H}(Z(w)^{\im \alpha}|_\Delta)$ becomes an isomorphism after applying $-\otimes_{E^*_{\langle b\rangle}}\Phi_1$;
    \item the Euler class of the normal bundle of the embedding $\fM(w)^{\im \alpha}|_\Delta\inj\mathfrak{M}(w)^{\im \alpha}$ becomes invertible  after applying $-\otimes_{E^*_{\langle b\rangle}}\Phi_1^L$.
\end{enumerate}

\end{lemma}
\begin{proof}
It is clear that $\langle b\rangle$-acts freely on $\mathfrak{M}(w)^{\im \alpha}\setminus\Delta$. We then use the localizing exact sequence on cohomology
\[\cdots \to E^{*-1}_{\langle b\rangle\times H}(\mathfrak{M}(w)^{\im \alpha}|_\Delta)\to E^*_{\langle b\rangle\times H}(\mathfrak{M}(w)^{\im \alpha}\setminus \Delta)\to E^*_{\langle b\rangle\times H}(\mathfrak{M}(w)^{\im \alpha})\to E^*_{\langle b\rangle\times H}(\mathfrak{M}(w)^{\im \alpha}|_\Delta)\to \cdots .\]
By Lemma~\ref{lem:cohomology_free_action}, the term $E^*_{\langle b\rangle\times H}(\mathfrak{M}(w)^{\im \alpha}\setminus\Delta)$ vanishes when applying $-\otimes_{E^*_{\langle b\rangle}}\Phi_1$. Hence we are done.

The statement (2) about $Z(w)$ is similar. The statement (3) about Euler class of the normal bundle is a direct consequence of Lemma~\ref{lem:lambda_invertible}.
\end{proof}

\section{The quantum Frobenius maps}
\label{sec:quantum Frob}
In this section, we show that the transchromatic character map of \cite{HKR, St}, modified by the Euler class of a normal bundle,  defines an algebra homomorphism between the convolution algebras. This is the quantum Frobenius homomorphism in the title of the present paper. 

\subsection{Group homomorphisms}
\label{subsec:alpha}
Recall $\Lambda_t=(\Z/p)^t$. 
Let us consider a group homomorphism \[
\alpha:\Lambda_t \to \GL(W)\times\bbC^*
\] satisfying the following conditions. 
Write $\alpha=(\alpha_1,\alpha_2,\dots,\alpha_t)$ where $\alpha_i: \Z/p\to \GL(W)\times\bbC^*, 1\mapsto (g_i, \zeta)$ for $i=1,\cdots,t$. Assuming $p^t|w$, the $g_1$-eigenspace decomposition of $W=W_1\oplus\cdots\oplus W_p$ is so that the eigenspaces are isomorphic as $I$-graded vector spaces, and hence we fix such an isomorphism $b_1:W_i\cong W_{i+1}$. Similarly, we assume that $g_2$-eigenspace on each of $W_i$ have the same property, with a fixed operator $b_2$. We continue this pattern. In other words, $b_k$ is the cyclic permutation of the eigenspaces of $g_k$ and the order of $b_k$ is $p$, for $1\leq k\leq t$. For $1\leq i\neq j \leq t$, $b_i, b_j$ commute. 
Let $W=\oplus_{1\leq i_1, \cdots i_k, \cdots, i_t \leq p}W_{i_1, \cdots, i_k, \cdots, i_t}$ be the common eigenspace decomposition of $(g_1, \cdots, g_t)$. 
Under the above assumption, we have the following operators $b_1, \cdots, b_t$ in $\GL(W)$, such that
$
b_k(W_{i_1, \cdots, i_k, \cdots, i_t})
\subseteq W_{i_1, \cdots, i_k+1, \cdots, i_t}, \,\ k=1, \cdots t. 
$
Let $\langle b\rangle$ be the abelian subgroup generated by $b_1, \cdots, b_t$. We then have
\[
\langle b\rangle=
\langle b_1 \rangle
\oplus \langle b_2 \rangle 
\oplus \cdots \oplus 
\langle b_t \rangle \cong (\Z/p)^{\oplus t}. 
\]
For each $k$, $\langle b_k\rangle$ acts on $(\Z/p)^p$  by cyclic permutation and it induces the following morphism
\[ 
\langle b_k\rangle \wr \alpha_k:=\langle b_k\rangle\ltimes(
\alpha_k\oplus b_k\alpha_kb_k^{-1} \oplus\cdots \oplus b_k^{p-1}\alpha_kb_k^{-p+1}):\langle b_k\rangle\ltimes(\Z/p)^{\oplus p} \to \GL_w\times\bbC^*. 
\] 
When $k\neq l$, we have $b_k\alpha_lb_k^{-1}=\alpha_l$. 
Therefore, we have the following morphism
\[
\langle b\rangle \wr \alpha:=(\langle b_1\rangle \wr \alpha_1) \times\cdots \times (\langle b_t\rangle \wr \alpha_t):
\langle b\rangle\ltimes((\Z/p)^{ p})^t\to \GL_w\times\bbC^*.
\]

The centralizer $C(\im \alpha)$ is isomorphic to $\prod_{p^t} \GL(w/p^t)\times \C^*$. 
Let $H\subseteq \GL(w/p^t)\times\bbC^*$ be any finite subgroup. We similarly have
\[
\langle b\rangle\wr (\alpha\times H):=
\langle b\rangle\ltimes\Big(
(\alpha\times H)\times b(\alpha\times H)b^{-1} \times \cdots  \times \cdots
b^{p-1}(\alpha\times H)b^{-p+1}
\Big).
\]
It is clear that there is a group homomorphism 
\begin{equation}\label{eq:group}
\langle b\rangle\wr (\alpha\times H)
\to \GL(w)\times\bbC^*. 
\end{equation}

In $((\Z/p)^p)^t\cong(\Lambda_t)^p$, we take a diagonal subgroup $\Lambda_t$, such that $\langle b \rangle$ acts trivially on $\Lambda_t$. Let $\Delta(\alpha):\Lambda_t \subset ((\Z/p)^p)^t \to \GL_w\times\bbC^*$ be the group homomorphism. We have the equality 
\[
Z(w)^{\im b^i\alpha b^{-i}}= Z(w)^{ \im \alpha}=Z(w)^{ \im \Delta(\alpha)}, i=0, \cdots p-1.
\]By the diagram \eqref{eq:action of b}, 
we have an action of $\langle b\rangle\wr (\alpha\times H)$ on $Z(w)^{ \im \alpha}$. It restricts to an action of the subgroup $\langle b\rangle\times \Delta(\alpha)\times \Delta(H) \cong \langle b\rangle\times \Lambda_t\times H$ of $\langle b\rangle\wr (\alpha\times H)$. 
We make the following conventions.
For $\alpha:A\to \GL_{w_1}$ and $\beta:A\to \GL_{w_2}$, we write $\alpha\times\beta:A\to \GL_{w_1+w_2}$ the composition of $\alpha\times\beta:A\to\GL_{w_1}\times\GL_{w_2}$ with the inclusion $\GL_{w_1}\times\GL_{w_2}\inj \GL_{w_1+w_2}$. Similarly when $A$ and $B$ are both abelian groups and $\alpha:A\to G$ and $\beta:B\to G$, we write $\alpha\oplus\beta:A\oplus B\to G$.

\subsection{Restriction to diagonal}\label{subsec:restr}
Assume $p^t| w$. Let $\alpha:\Lambda_t\to \GL_w\times\bbC^*$ be a group homomorphism as above. 
We have as in Lemma~\ref{lem:diag_power} the diagonal embedding $\fM(w/p)\inj \fM^{\im \alpha_1}$. By assumption of $\alpha$, we have $\alpha_2$ being obtained from a map $\Z/p\to \GL_{w/p}\times\bbC^*$, which without causing confusions is still denoted by $\alpha_2$. In particular, we have $\fM(w/p^2)\inj \fM(w/p)^{\im \alpha_2}$. Composition of the two diagonal embeddings gives us $i_\Delta:\fM(w/p^2)\inj \fM(w)^{\im \alpha_1\oplus\alpha_2}$.

In general repeating this process we get $\mathfrak{M}(w)^{\im \alpha}|_\Delta\cong \fM(w/p^t)$ and the embedding $i_\Delta:\fM(w/p^t)\inj \fM(w)^{\im \alpha}$.
Let $H$ be a finite subgroup of $\GL(w/p^t)\times\bbC^*$, embedded into $C(\im \alpha)\cong \prod_{p^t}\GL(w/p^t)\times\bbC^*$ diagonally as in \S~\ref{subsec:alpha}. Note that by construction the conjugation action of $\langle b\rangle\cong(\Z/p)^t$ is trivial on $H$ and hence the actions of $\langle b\rangle$ and $H$ commute on $\fM(w)^{\im \alpha}$. 

\begin{lemma}\label{lem:t_diag}
Let $i_\Delta:\fM(w/p^t)\inj \fM(w)^{\im \alpha}$ be the diagonal embedding, which is   $\langle b\rangle\times \Delta(\alpha)\times H$-equivariant. 
Then, 
\begin{enumerate}
    \item pulling back  $i_\Delta^*:E^*_{\langle b\rangle\times \Delta(\alpha)\times H}(\mathfrak{M}(w)^{\im \alpha})\to E^*_{\langle b\rangle\times \Delta(\alpha)\times H}(\mathfrak{M}(w)^{\im \alpha}|_\Delta)$ becomes an isomorphism after applying $-\otimes_{E^*_{\langle b\rangle}}\Phi_t$;
    \item similarly, $i_\Delta^*:E^*_{\langle b\rangle\times \Delta(\alpha)\times H}(Z(w)^{\im \alpha})\to E^*_{\langle b\rangle\times \Delta(\alpha)\times H}(Z(w)^{\im \alpha}|_\Delta)$ becomes an isomorphism after applying $-\otimes_{E^*_{\langle b\rangle}}\Phi_t$;
    \item the Euler class of the normal bundle of the embedding $\fM(w)^{\im \alpha}|_\Delta\inj\mathfrak{M}(w)^{\im \alpha}$ becomes invertible  after applying $-\otimes_{E^*_{\langle b\rangle}}\Phi_t^L$.
\end{enumerate}
\end{lemma}
\begin{proof}
We use Lemma~\ref{lem:diag_power} inductively.
Without loss of generality, we show this for $t=2$.
Note that $-\otimes_{E^*_{(\Z/p)^2}}\Phi_2$ is a localization, which  can be done in two steps, inverting $c_1$ of non-trivial representations of the second $\Z/p$-factor, and then those of the first $\Z/p$-factor. We factor $i_\Delta$ into $i_{\Delta_2}:\fM(w/p)\inj \fM^{\im \alpha_1}$ and $i_{\Delta_1}:\fM(w/p^2)\inj \fM(w/p)^{\im \alpha_1\oplus \alpha_2}$. Pulling back via $\Delta_2$ is an isomorphism when inverting $c_1$ of non-trivial representations of the second $\Z/p$-factor, and that of $\Delta_1$ is an isomorphism when inverting $c_1$ of non-trivial representations of the first $\Z/p$-factor.

To show the statement of the normal bundle, note that the normal bundle of $i_\Delta$ is an extension of the normal bundle of $i_{\Delta_1}$ by that of $i_{\Delta_2}$ and hence its Euler class is the product that of $i_{\Delta_1}$ and $i_{\Delta_2}$. Each individual factor is invertible by Lemma~\ref{lem:diag_power}.
\end{proof}
To summarize, we have the following diagram of equivariance
\[\xymatrix@R=0.5em{
\langle b\rangle\times \Delta(\alpha)\times H \ar@{}[r]|{\subset} & 
\langle b\rangle\wr (\alpha\times H)
\ar@{}[r]|{\phantom{1234}\subset} &\varmathbb{G}\\
&&\\
\ar@(ul,ur)[]& \ar@(ul,ur)[]&\ar@(ul,ur)[]\\
\fM(w/p^t) \ar@{^{(}->}[r]
& \fM(w)^{\im \alpha} \ar@{^{(}->}[r] & \fM\\
}
\]

\subsection{Coefficients}\label{subsec:coefficient}
We de-clutter and de-clash notations here. Let $E^*$ be the base ring the $E_n^*$-theory. Let $L(E_n^*)$ be the construction of \cite{HKR} via localization of $E^*(B\Lambda_n)$. 
We avoid using $L$ for Bousfield localization, instead, let $E_{K_{n-t}}$ be the localization of $E_n^*$ with respect to the Morava K-theory $K_{n-t}$. We then have ring homomorphisms
$E_n^*\to (E_n^*)_{ K_{n-t}} \to C_{n-t}^*$ by definition. Recall that \cite[Proposition~2.17]{St} with
the unique maximal ideal of $E^0_n$ contained in $I_t\subset E_{K_{n-t}}$, 
the ring homomorphism $E_{ K_{n-t}} /I_t\to C_{n-t}/I_t$ is faithfully flat.  Hence, the unique maximal ideal of $E^0_n$ can be lifted (although non-uniquely in general) to a maximal ideal of $C_{n-t}^0$. 
Let $C^{\wedge*}_{n-t}$ be the completion of $C_{n-t}^*$ at any one of such maximal ideals. Here the purpose of the completion is so that $C^{\wedge*}_{n-t}$ satisfies Assumption~\ref{ass:HKR} of height $n-t$. Similar to $L(E_n^*)$, we also have  $L(C^{\wedge*}_{n-t})$.

Let $L_{K_{n-t}E}$ be the base-change of $L(E_n^*)$ from $E_n$ to $E_{K_{n-t}}$, so that $L_{K_{n-t}E}$ is a ring over $E_{K_{n-t}}$ and the formal group law of $E_n$ base-changed to $L_{K_{n-t}E}$ is  constant $(\Z/p)^n$. On the other hand, the ring homomorphism $E_{K_{n-t}}\to C_{n-t}$ is universal with respect to the property that the formal group  of $E_n$ base-changed to a sum of $(\Z/p)^{t}$ with a  formal group of height $n-t$. Hence, the map $E_{K_{n-t}}\to L_{K_{n-t}E}$ factors through $C_{n-t}$. The height-$(n-t)$ factor on $C_{n-t}$ when base-changed to $L_{K_{n-t}E}$  becomes constant $(\Z/p)^{n-t}$. 
The universal property of $L(C^{\wedge*}_{n-t})$ then implies that the base-change of $L_{K_{n-t}E}$  to $C^\wedge_{n-t}$ is isomorphic to $L(C^{\wedge*}_{n-t})$. 
The above is summarized in the commutative diagram.
\begin{equation}\label{eq:field L}
\xymatrix@R=1em{
C_{n-t}^{\wedge} \ar[r]& L(C^{\wedge*}_{n-t})\\
C_{n-t} \ar[u] \ar[r] & L_{K_{n-t}E} \ar[u]\\
E_{K_{n-t}} \ar[u]\ar[ur]& \\
E_n\ar[u] \ar[r]&  L(E_n^*)\ar[uu]
}
\end{equation}

Let us consider $E^*_{\langle b\rangle \times \Lambda_t\times H}$ for a finite group $H$ as in \S~\ref{subsec:restr}. 
Recall that we have ring homomorphisms $S_{\Lambda}^{-1}E_{\Lambda_t}^*\to C_{n-t}^*$ as well as $S_b^{-1}E_{\langle b\rangle}^*\to \Phi_{t}^E$.  Recall Lemma~\ref{lem2} implies that after inverting $S_{\Lambda}$ the $\Lambda_t\times\langle b\rangle$-equivariant $c_1$ of non-trivial $\Lambda_t$-representations are invertible; similar for $S_{b}$. Hence, we obtain \[S_{\Lambda}^{-1} S_{b}^{-1}E^*_{\Lambda_t\times\langle b\rangle}\cong S_{\Lambda}^{-1}E^*_{\Lambda_t}\otimes_{E^0}S_{b}^{-1}E^*_{\langle b\rangle}\to C_{n-t}^0\otimes_{E^0}S_{b}^{-1}E^*_{\langle b\rangle}\to L_{K_{n-t}E}\otimes_{E^0}S_{b}^{-1}E^*_{\langle b\rangle} \cong  \Phi_{t}^{L_{K_{n-t}E}}.\]
More generally, for any finite $H$ $CW$-complex $X$ with trivial $\Lambda_t\times\langle b\rangle$-action, we then have \begin{equation}\label{eqn:LocPhi}
E^*_{\Lambda_t\times\langle b\rangle\times H}(X)\to S_\Lambda^{-1}S_{b}^{-1}E^*_{\Lambda_t\times\langle b\rangle\times H}(X)\to \Phi_{t}^{L_{K_{n-t}E}}\otimes_{C^0_{n-t}} 
C^*_{n-t, H}(X). \end{equation}
\subsection{Definition of Frobenii}\label{subsec:def_Frob}
The morphism $\Fr_{n, n-t}$ is defined to be the composition of the following sequence of morphisms.

\begin{enumerate}
\item[(a)]
The change of group morphism
\[
E^*_{\varmathbb{G}}( Z) \to E^*_{\langle b\rangle\wr (\alpha\times H)}(Z).
\]
\item [(b)] 
Let $i: Z^{\im \alpha}|_\Delta\inj Z$ be embedding of the fixed point set of $\im \alpha$-action. 
The pullback 
\[
i^*: E^*_{\langle b\rangle\wr (\alpha\times H)}(Z)
\to E^*_{\langle b\rangle\times\Lambda_t\times H}( Z^{ \im \alpha}|_\Delta)\cong E^*_{\langle b\rangle\times\Lambda_t\times H}( Z(w/p^t)).
\]
\item[(c)]
The natural localization maps
\[
E^*_{\langle b\rangle\times \Lambda_t \times H}( Z^{ \im \alpha}|_\Delta)
\to
S_\Lambda^{-1}
S_b^{-1}E^*_{\langle b\rangle\times \Lambda_t\times H}( Z^{ \im \alpha}|_\Delta)
\to \Phi_t^{L_{K_{n-t}E}}\otimes_{C^0_{n-t}} 
C^*_{n-t, H}(Z(\frac{w}{p^t})).
\]
\item[(d)]
Let $N$ be the normal bundle of the embedding $i:\mathfrak{M}(w)^{\im \alpha}|_\Delta\inj \mathfrak{M}(w)$ with equivariant Euler class denoted by $e$. Multiplication by $1\boxtimes \frac{1}{e}$
\[
1\boxtimes \frac{1}{e}: \Phi_t^{L_{K_{n-t}E}}\otimes_{C^0_{n-t}}
C^*_{n-t, H}(Z(\frac{w}{p^t}))
\to
\Phi_t^{L_{K_{n-t}E}}\otimes_{C^0_{n-t}}
C^*_{n-t, H}(Z(\frac{w}{p^t})).
\]
\end{enumerate}

We assemble all the discussions so far into the following. 
\begin{theorem}\label{thm:algebra hom}
The morphism $\Fr_{n, n-t}$ is an algebra homomorphism.
\end{theorem}
\begin{proof}
First of all, we factor the maps in (b) and (c) above into two steps each. Note that the embedding $i: Z(w)^{\im \alpha}|_\Delta\inj Z(w)$ factors as $i_\alpha: Z^{\im \alpha}\inj Z(w)$ and $i_\Delta: Z^{\im \alpha}|_\Delta\to Z^{\im \alpha}$. The normal bundle of $i$ is an extension that of $i_\alpha$ by that of $i_\Delta$, and hence its Euler class is a product that of the two individual factors $e=e_\Delta\cdot e_\alpha$.

We have $\frac{1}{e_\alpha}i^*_{\alpha}: E^*_{\langle b\rangle\times\Delta(\alpha)\times H}(Z(w))\to E^*_{\langle b\rangle\times\Delta(\alpha)\times H}(Z(w)^{\im \alpha})\otimes_{E_{\Lambda_t}}{L(E^*)}\cong E^*_{\langle b\rangle\times H}(Z(w)^{\im \alpha})\otimes_{E^0}{L(E^*)}$ from
Proposition~\ref{prop:localization_convolution}, which is an algebra homomorphism.

Now we have the diagonal pullback divided by $e_\Delta$, which exists thanks to 
Lemma~\ref{lem:t_diag}
\[\frac{1}{e_\Delta}i^*_{\Delta}:E^*_{\langle b\rangle\times H}(Z(w)^{\im \alpha})\otimes_{E^*_{\langle b\rangle}}\Phi_t^{L(E^*)}\to E^*_{\langle b\rangle\times H}(Z(w)^{\im \alpha}|_\Delta)\otimes_{E^*_{\langle b\rangle}}\Phi_t^{L(E^*)}.\]
The algebra structures on $E^*_{\langle b\rangle\times H}(Z(w)^{\im \alpha})$ and $E^*_{\langle b\rangle\times H}(Z(w)^{\im \alpha}|_\Delta)$ are both given by convolution.  
The same calculation (with $M_i$ replaced by $\fM(w)^{\im \alpha}$ and $M_i^{A_1}$ replaced by $\fM(w)^{\im \alpha}|_{\Delta}$) as in the proof of Proposition~\ref{prop:localization_convolution} now yields that this is also an algebra homomorphism. 

Base-change from $E^0$ to $E_{K_{n-t}}$, the discussion from \S~\ref{subsec:coefficient} then gives an algebra homomorphism \[ E^*_{\langle b\rangle\times H}(Z(w)^{\im \alpha}|_\Delta)\otimes_{E^*_{\langle b\rangle}}\Phi_t^{L(E^*)}\to \Phi_t^{L_{K_{n-t}E}}\otimes_{C^0_{n-t}}
C^*_{n-t, H}(Z(\frac{w}{p^t})).\]
The map $\Fr_{n, n-t}$ then is the composition of (a) with the three algebra homomorphisms above. 
\end{proof}

In what follows, we also use the terminology Frobenius for variants of the above map, when the domain is modified with $\varmathbb{G}$ replaced by a finite subgroup $G$ with $\langle b\rangle\wr (\alpha\times H)\subseteq G\subseteq \varmathbb{G}$, or when the target is further scalar changed along any ring homomorphism $\Phi_t^{L_{K_{n-t}E}}\to R$.
Starting from \S~\ref{sec:Frob phen} we consider representations over a field $\mathbb{F}$ which factors through $\Phi_{t}^{L(C^{\wedge*}_{n-t})}$. 

\begin{definition}
\label{def:U}
Let $G\subset \GL_w\times \C^*$ be a subgroup. Let $\mathbb{F}$ be a field of characteristic zero endowed with a map from $E^*_{G}(\pt)$. 
Define 
\begin{align*}
&
U^n_{w}(G, \mathbb{F})\subset \mathbb{F}\otimes_{E^*_{G}(\pt)}E^*_{n, \GL_w\times \C^*}(Z(w))
\end{align*}
to be the subalgebra generated by the Chern classes of the tautological vector bundles $\calV_{k}^{(n)} \to \fP(v, v+n\alpha_k, w)$ and $\calV_{k}^{(n), \op} \to \fP( v+n\alpha_k, v, w)$ where $n$ varies in $\N$ and $k$ varies in $I$. 

Let 
\[U'^n_{w}(G, \mathbb{F})\subset \mathbb{F}\otimes_{E^*_{G}(\pt)}E^*_{n, \GL_w\times \C^*}(Z(w))\] be the subalgebra generated by the Chern class of the tautological line bundles $\calL_{k} \to \fP(v, v+\alpha_k, w)$ and $\calL_{k}^{\op} \to \fP(v+\alpha_k, v, w)$ where $k$ varies in $I$. 
Recall the quantum groups $U_w(\bbG;E_n)$ and $U_w'(\bbG;E_n)$ from \S~\ref{subsec:def_quant_groups}. The algebras $U^n_{w}(G, \mathbb{F})$ and $U'^n_{w}(G, \mathbb{F})$ are their specializations  along the map $E^*_{\bbG}(\pt)\to E^*_{G}(\pt)\to \bbF$.

Let $\gamma:\Lambda_n\to \varmathbb{G}$ be a group homomorphism. When $G$ is chosen to be $\im \gamma$, we denote $U^n_{w}(\im\gamma, \mathbb{F})$ by $U^n_{w}(\gamma, \mathbb{F})$ for short. Sometimes, we skip $n$ or skip $(\gamma, \mathbb{F})$ in the notation and simply denote $U^n_{w}(\gamma, \mathbb{F})$ by $U_{w}(\gamma, \mathbb{F})$ or $U^n_{w}$ if it is understood from the context. 
\end{definition}

Notice that the normal bundle $N$ in (d) can be expressed in terms of the tautological bundles of the quiver variety, and hence $e$ is a tautological class. As a consequence, suppose $p^t\mid w$,  $\gamma=\alpha \oplus \gamma'$, where $\alpha: \Lambda_t \to \GL_w\times \C^*$ is as in \S\ref{subsec:alpha} and $\gamma':\Lambda_{n-t}\to \GL_{w/p^t}\times \C^* \xrightarrow{\Delta} \GL_{w}\times \C^*$ is a group homomorphism. Assume further that the  $\Phi_t^{L(E^*)}$-field $\mathbb{F}$ factors through $\Phi_t^{L_{K_{n-t}E}}$, then we have the Frobenius homomorphism
\begin{equation}
\Fr_{n, n-t}: U^n_{w}(\gamma, \mathbb{F})
\to U^{n-t}_{w/p^t}(\gamma', \mathbb{F}). 
\end{equation}

\section{Discussions and examples}	
This section consists of some miscellaneous discussions around Theorem~\ref{thm:algebra hom}.
\subsection{Some remarks}
We expect the target of the Frobenius map above to be explained as a version of power operations. This expectation is made precise in the $A_1$-case in \S~\ref{sec:a1case}.
\begin{remark}\label{rmk:localization}
The localization of inverting $S_\Lambda, S_b$ in (c) has two purposes. 
One is so that the multiplication $1\boxtimes \frac{1}{e}$ is well-defined. Another, as we have seen in \S~\ref{sub:Recall Morava}, shows up in the definition of transchromatic character map $E_n^*\to E^*_{n}(B(\Lambda_t))[\frac{1}{S_{\Lambda}}] \to C_{n-t}^*$. 
\end{remark}

More about the coefficients. 
\begin{remark}\label{rmk:coef}
\begin{enumerate}
    \item 
This is used in \S~\ref{subsec:irred_mod}. We have $E_n\to (E_{n})_{K_{n-t}}\to C_{n-t}$, and any maximal ideal $\mathfrak{m}$ of $C_{n-t}$ as in \S~\ref{subsec:coefficient}, the formal group on $C_{n-t}/\mathfrak{m}$ is of height $n-t$ of a perfect field of characteristic $p$.  The theorem of Lazard \cite{Laz} implies that this is isomorphic to the formal group law of $K_{n-t}$.  The universal deformation theorem of Hopkins and Miller \cite{Rez98} then implies the existence of a morphism (although non-unique)  $E_{n-t}\to C_{n-t}^\wedge$ with an isomorphism between the formal group law on the latter and that of the pushforward from $E_{n-t}$.
\item Now let $t=t_1+t_2$. The paragraph above gives the first row of the following diagram. Localization with respect to $K(n-t)$ gives the second row. Here we retain the convention of writing localization. \[\xymatrix{
E_{n-t_1}\ar[d]\ar[r]&C^\wedge_{n,n-t_1}\ar[d]&E_n\ar[l]\ar[d]\\
E_{n-t_1,K(n-t)}\ar[r]\ar[d]&C^\wedge_{n-t_1,K(n-t)}\ar[d]\ar[dr]&E_{n,K(n-t)}\ar[l]\ar[d]\\
C_{n-t_1,n-t}\ar[r]& C_\star \ar[r]&C_{n,n-t}
}\]Taking $C_\star$ as the base-change filling the lower left corner, universal properties then give all the remaining arrows in the diagram. 
\end{enumerate}
\end{remark}

\subsection{Iteration of the Frobenius}
Let $t=t_1+t_2$ and $p^t|w$. Assume $n\geq t$, 
and $\alpha:\Lambda_t \to \GL_w\times\bbC^*$ with $\alpha':\Lambda_{t_1}\to \GL_w\times\bbC^*$, $\alpha'':\Lambda_{t_2}\to \GL_{w/p^{t_1}}\times\bbC^*$ and $H\subseteq\GL_{w/p^{t}}\times\bbC^*$ a finite subgroup. 
Let $\langle b\rangle =\langle b'\rangle \oplus \langle b''\rangle$ be the operators as before. 
Let $H'\subseteq \GL_{w/p^{t_1}}\times\bbC^*$ be the subgroup $\langle b''\rangle\wr (\alpha''\wr (\langle b'\rangle\times H))$. We have the two Frobenii 
\begin{align*}
& E^*_{n,\varmathbb{G}}(Z(w))\to \Phi_{t_1}^{L(C^\wedge_{n,n-t_1})}\otimes_{C_{n,n-t_1}} C^*_{n,n-t_1,H'}(Z(w/p^{t_1}))\\
\text{as well as } \,\ 
& E^*_{n-t_1,H'}(Z(w/p^{t_1}))\to \Phi_{t_2}^{L(C^\wedge_{n-t_1,n-t})}\otimes_{C_{n-t_1,n-t}} C^*_{n-t_1,n-t,H}(Z(w/p^t)). 
\end{align*}
Remark~\ref{rmk:coef}(2) then gives the base-change of the latter $C^{\wedge*}_{n,n-t_1,H'}\to \Phi_{t_2}^{L(C_\star)}\otimes_{C_\star}C_{\star, H}^*(Z(w/p^t))$, which then is endowed with a map to $\Phi_{t_2}^{L(C^\wedge_{n,n-t})}\otimes_{C_{n,n-t}}C^*_{n,n-t,H}(Z(w/p^t))$ that preserves the convolution algebra structures. 

Taking $\alpha=\alpha'\oplus \alpha''$, we also have $E^*_{n,\varmathbb{G}}(Z(w))\to \Phi_{t}^{L(C^\wedge_{n,n-t})}\otimes_{C_{n,n-t}}C^*_{n,n-t,H}(Z(w/p^t))$. Recall that Lemma~\ref{lem2} gives  a map $\Phi_{t_1}^{L(C^\wedge_{n,n-t})}\otimes_{C_{n,n-t}}\Phi_{t_2}^{L(C^\wedge_{n,n-t})}\to \Phi_{t}^{L(C^\wedge_{n,n-t})}$.
The same argument as the last paragraph in the proof of Theorem~\ref{thm:algebra hom} yields that the following diagram commutes. 
\[
\xymatrix{
E^*_{n, \varmathbb{G}} (Z(w)) \ar[r]^(0.3){\Fr_{n, n-t_1}} \ar[dr]_{\Fr_{n, n-t}}& \Phi_{t_1}^{L(C^\wedge_{n,n-t_1})}\otimes C^{\wedge}_{n,n-t_1, H'} ( Z(w/p^{t_1}))
\ar[d]^{\Fr_{n-t_1, n-t}}\\
&\Phi_{t}^{L(C^\wedge_{n,n-t})}\otimes C_{n-t,  H} ( Z(w/p^t)) .
}
\]
Let $\gamma: \Lambda_n\to \GL_w\times \C^*$, $\gamma': \Lambda_{n-t_1}\to \GL_{w/p^{t_1}}\times \C^*$ and $\gamma'': \Lambda_{n-t}\to \GL_{w/p^{t}}\times \C^*$
be group homomorphisms, such that
\[
\gamma=\alpha'\oplus r'=\alpha\oplus r'', \,\   \text{and} \,\
\gamma'=\alpha''\oplus r''.
\]
It induces the following commutative diagram.
\[
\xymatrix{
U^n_{w}(\gamma, \mathbb{F}) \ar[r]^{\Fr_{n, n-t_1}}
\ar[dr]_{\Fr_{n, n-t}}&
U^{n-t_1}_{w/p^{t_1}}(\gamma', \mathbb{F}) \ar[d]^{\Fr_{n-t_1, n-t}}\\
& U^{n-t}_{w/p^{t}}(\gamma'', \mathbb{F})}
\]

In this sense, we say the Frobenius $\Fr_{n, n-t}$ can be obtained as the composition of $\Fr_{n, n-t_1}$ and $\Fr_{n-t_1, n-t}$. This gives the iteration of the Frobenius maps.

\subsection{Examples}
\label{ex:two}
\begin{enumerate}
\item
Assume the weight $w$ is divisible by $p^n$. 
Starting with the Morava $E_n$-theory, the quantum Frobenius $\Fr_{n, 0}$ lands in \[
\mathbb{F}\otimes_{E_0} E_0(Z(w/p^n)),\,\  (\text{precisely $\Phi_{n}^{L(E_0)}\otimes C^{\wedge}_{0} ( Z(w/p^{n}))$.}
\]
The target  is the non-equivariant cohomology with height $0$ formal group law. There is an algebra homomorphism \cite{Nak98}
\[
U(\fg)_{\mathbb{F}} \to U^0_{w/p^n}(1, \mathbb{F})
\]
from the enveloping algebra of the Lie algebra $\fg$ with coefficients in the field $\mathbb{F}$, to a non-equivariant cohomology. It is shown by Nakajima that  furthermore the irreducible representations of $U(\fg)_{\mathbb{F}}$ with highest weight less than $w$ are in one-to-one correspondence with those of $\mathbb{F}\otimes_{E_0} E_0(Z(w/p^n))$.
\item
Although the equivariant $K$-theory of taking Grothendieck group of the equivariant vector bundles is not a special case of the Morava theories, 
the construction of the Frobenius maps applies to this equivariant $K$-theory with the Theorem C of \cite{HKR} 
replaced by the Atiyah-Segal formula. 
We make the following substitutions
\begin{align*}
&\text{$E_1^*$ is replaced by the equivariant K-theory $K$, $E_0^0$ is replaced by the cohomology theory $H$},\\
&\text{$L(E_1^0)=L((E_1^0)_{K(0)})$ is replaced by $L(K)$, which is a cyclotomic field},\\
&\text{and $\Phi_1^{L(E_1^0)}$ is replaced by $\Phi_1^{L(K)}$, which is the cyclotomic ring of $L(K)$}.
\end{align*}
Let $U_{\epsilon}(L\fg)$ be the quantum loop algebra with $\epsilon^p=1$. We have an algebra homomorphism
\[
U_{\epsilon}(L\fg)_{\Phi_1^{L(K)}}\to U^1_{w}(\gamma, \Phi_1^{L(K)}). 
\]
The Frobenius map (for all $w\in X^+$) lifts to an algebra homomorphism 
\[
U_{\epsilon}(L\fg)_{\Phi_1^{L(K)}}
\to U(\fg)_{\Phi_1^{L(K)}}
\]
with coefficient in the cyclotomic ring $\Phi_1^{L(K)}$. 
This recovers Lusztig's quantum Frobinius map \cite[Proposition 7.5]{Lu2}. 
Recall that the use of the cyclotomic ring is crucial in the construction of Lusztig's  quantum Frobenius maps. 
As explained in Example \ref{ex:K-theory}, passing to $\Phi_1^{L(K)}$ is an analogue to Lusztig's construction of taking the Grothendieck group of equivariant category of constructible sheaves \cite[Chapter 11]{Lus93}.
Lusztig's construction has been used by McGerty \cite{Mc} in a categorification of quantum Frobenius, where the cyclotomic ring occurs in the same way as in  Example~\ref{ex:K-theory}. 

However, note that in the coefficients $\Phi_1^{L(K)}$, if $\Phi_1^K$ is cyclotomic ring, then 
 $L(K)$ is an ``extra"  cyclotomic ring.  Passing to $L(K)$ is necessary only because we are working with the quantum loop algebras instead of the finite dimensional ones. Roughly speaking, the loop parameters contributes to these ``extra" cyclotomic extensions.  It is out of scope of the present paper but in principal one could construct the finite dimensional quantum groups associated to the Morava $E_n$ theory without going to the loop algebras, hence, potentially remove $L$ and only uses $\Phi_t^{E_n}$. 
\end{enumerate}

Similar to how $K$-theory is to $E_1^*$, in \cite{Ba} $E_2^*$ is proven to be a direct summand of elliptic cohomology spectrum localized at $p$. Using 
\cite{KM} in place of \cite{HKR} and the Atiyah-Segal formula, we expect to obtain a quantum Frobenius from the elliptic quantum group of
\cite{YZ3} at a torsion point of the elliptic curve, to quantum loop algebra at a root of unity.

\section{Power operations and another construction in the $A_1$-case}
\label{sec:mod}
We give a different construction of the Frobenius map when $Q$ is of type $A_1$. Although only a special case, we include this for its conceptual nature in terms of topology. 
\subsection{The definition}\label{subsec:power_def}
We first recall some basic properties of power operations, which are needed for the present paper. We suggest experts in this area to skip \S~\ref{subsec:power_def}, and refer  interested readers  to \cite{A95,G} for more details. Let $E^*$ be an oriented cohomology theory. Let $\Sigma_r$ be the symmetric group on $r$ letters. For a finite group $G$, set
\[
G \wr \Sigma_r:=G^r\rtimes \Sigma_r. 
\]
\begin{definition}(see \cite[VIII.1.1]{BMMS86} and \cite[Definition 4.3]{G})
An $H_{\infty}$-structure on $E$ is a collection of maps (power operations)
\[
P_r: E_G^*(X)\to E_{G \wr \Sigma_r}^* (X^r)\]
with finite group $G$ and finite $G$ $CW$-complex $X$ satisfying the following conditions
\begin{enumerate}
\item[(a)] $P_{1}=\id$ and $P_{0}(x)=1$.
\item[(b)] the (external) product of two power operations is 
$
P_j(x) \wedge P_{k}(x)=\res|^{\Sigma_{j+k}}_{\Sigma_j\times \Sigma_k} (P_{j+k}(x)). 
$
\item[(c)]\label{cond_c} the composition of two power operations is
$
P_j(P_k(x))=\res|^{\Sigma_{jk}}_{\Sigma_k\wr \Sigma_j} (P_{jk}(x)).
$
\item[(d)] the $P_j$'s preserve (external) products 
$P_j(x\wedge y)=\res|^{\Sigma_{j}\times \Sigma_{j}}_{\Sigma_j}(P_j(x)\wedge P_j(x))$, where the restriction is along the map
\[
[((X \curvearrowleft G)^2)^j \curvearrowleft \Sigma_j]
\to
[(X \curvearrowleft G)^{2j} \curvearrowleft (\Sigma_j\times \Sigma_j)]
\cong 
[( (X \curvearrowleft G)^j\curvearrowleft \Sigma_j)^2 ]. 
\]
\end{enumerate}
\end{definition}
\begin{example}
The power operations in the equivariant cobordism, expressed in terms of generators, are given by
\[
P_r^{\MU}: \MU^{2*}_G(X)\to \MU^{2r*}_{G \wr \Sigma_r} (X^r), \,\
[M\to X]\mapsto [M^r\to X^r] \curvearrowleft \Sigma_r.
\]
\end{example}
 We work on a cohomology theory $E^*$ with an $H_{\infty}$-structure, such that there is an $H_{\infty}$-map $\MU^*_G(X)\to E^*_G(X)$. That is, we have the following commutative square. 
 \begin{equation}\label{H_infty}
 \xymatrix{
 \MU^*_G(X) \ar[d]^{P_r}\ar[r]&  E^*_G(X) \ar[d]^{P_r}\\
 \MU^*_{G \wr \Sigma_r} (X^r)\ar[r] & E^*_{G \wr \Sigma_r} (X^r)
 }
 \end{equation}
 where the vertical maps are the power operations for $\MU^*_G$ and $E^*_G$ respectively. 
 We refer this assumption as the $H_{\infty}$-assumption. 
\Omit{ 
For the rest of this section, we follow \cite{A95}. We skip the finite group $G$ and still denote by
\[
P_r^{\MU}: MU^{2*}(X)\to MU^{2r*}(E\Sigma_r\times_{\Sigma_r} X^r)
\] the total power operation for complex cobordism. 
Let $\pi$ be an abelian $p$-group of order $r$.  By ordering the elements of $\pi$, we obtain a map
$
E\pi\times_{\pi} X^\pi\to 
E\Sigma_r\times_{\Sigma_r} X^r, 
$
where $X^{\pi}:=Map(\pi, X)$ is the left $\pi$-space obtained by using the right action of $\pi$ on itself. 
The operation $P_n^{\MU}$ determines the power operation
\[
P_{\pi}^{\MU}: \MU^{2*}(X)\to \MU^{2r*}(E\pi\times_{\pi} X^\pi)
\]
which factors the $r$-th external power map
$
\MU^{2*}(X)\xrightarrow{z \,\ \mapsto z^{\times r}} 
\MU^{2*}(X^{\pi}).
$ through the inclusion of the fiber 
$i: X^{\pi}\inj E\pi\times_{\pi} X^\pi$:
$
\MU^{2r^*}(E\pi\times_{\pi} X^\pi)\xrightarrow{i^*}
\MU^{2r^*}(X^{\pi}). 
$

 Let 
\[
 \Delta: B\pi\times X\to E\pi\times_{\pi} X^\pi
\]
 denote the diagonal map. 
 Let $L\to X$ be a complex line bundle and $eL:=s^*s_*(1)\in MU^{2}(X)$ its Euler class, where $s: X\to L$ is the zero section. 
 \begin{prop}\cite[Proposition 3.2.10]{A95}
 \[
 \Delta^* P_{\pi}^{\MU}(eL)=\prod_{u\in \pi^*} \Big(e(E\pi\times_{u}\C\to B\pi)+_{F_{\MU}} eL
\Big) \in MU^{2r}(B\pi\times X), 
 \]
where $\pi^*:=\Hom(\pi, \C^*)$ is the character group of $\pi$. 
 \end{prop}
 }

Recall that for two formal group laws $F_1$ and $F_2$ over a commutative ring $R$, a homomorphism is power series $h(u)\in R[\![u]\!]$, such that
\[
h(u_1+_{F_1}u_2)=h(u_1)+_{F_2} h(u_2). 
\]
It is an isomorphism is $h(u)=au \mod u^2$ with $a\in R$ a unit.
Let $[p]_{F}(t):=t+_{F}+t+_{F}\cdots+_{F}t \in R[[t]]$ be the $p$-series of the formal group law $F$.  Then, $[p]_{F}$ is a homomorphism of group laws (isogeny)
\[
F\xrightarrow{[p]_{F}} F. 
\]
Let $D_k$ be the ring extension of $E_n$ obtained by adjoinging all the roots of the $p^k$-series $[p]_{F}(t)$ of $F$. 
Let $_{p^k}F(D_k)$ be the subgroup of points of order $p^k$. Then, 
\[
_{p^k}F(D_k)\cong (\Z/p^k\Z)^n. 
\]

For any  finite subgroup $H\subset _{p^k}F(D_k)$,  let $H^*=\Hom(H, \C^*)$ be the character group of $H$. Taking the dual of the embedding $H\subset (\Z/p^k\Z)^n)^*\cong  _{p^k}F(D_k)$ induces a map $ (\Z/p^k\Z)^n\to H^*$. 
Then, the construction of \cite{HKR} gives $\chi: E_n^*(BH^*)\to D_k \subset S^{-1}E_n^*(B (\Z/p^k\Z)^n)$. 
 It is shown by Ando that there is an transformation of ring-valued functors 
 \[
E^*_n(X)\xrightarrow{\Psi^H} D_k\otimes_{E_n} E_n^*(X)
 \] such that the following diagram commutes \cite[Theorem 1]{A95}
 \[
 \xymatrix{
 \MU^{2*}(X) \ar[r]^(0.4){P_{H^*}^{\MU}} \ar[d]  & \MU^{2r*} (EH^*\times_{H^*} X) \ar[r]^{t} &E_n^* (EH^*\times_{H^*} X)   \ar[d]^{\chi^H}\\
 E_n^*(X) \ar[rr]^{\Psi^H} & & D_k\otimes_{E_n} E_n^*(X)
 },
 \]
 where $\chi^H$ is given by \cite[Definition 3.3.5]{A95} 
 \[
 \chi^H: E_n^*(EH^*\times_H^*X) \xrightarrow{\Delta^*} 
 E_n^*(BH^*)\otimes_{E_n} E_n^*(X)
 \xrightarrow{\chi\otimes 1} D_k\otimes_{E_n} E_n^*(X).
 \]
 %Under certain conditions on $F$  \cite[Theorem 4]{A95}, the effect of $\Psi^H$ is given as \[ \Psi^{H}(eL)=f_H(eL)\in D_k\otimes_{E_n} E_n(X).  \]

\begin{prop}\cite[Theorem 2 and Proposition 3.6.1]{A95}\label{p series}Notations as above, 
 taking $H$ to be  $H=_{p^k}F(D_k)$, then $\Phi^H$  factors through $E_n(X)\subseteq D_k\otimes_{E_n} E_n(X)$.  In particular, we obtain an operation
 \[
 E_n^*(X)\xrightarrow{\Psi^{p^k}} E_n^*(X)
 \]
 which is a natural transformation of rings. It induces the identity on coefficients and its effect on the Euler class $c_1L$ of a line bundle $L\to X$ is given by the $p^k$ series
 \[
 \Psi^{p^k}(c_1L)=[p^k]_{F} (c_1L). 
 \]
 \end{prop}

 \subsection{Power operations and convolution algebras}

Now we retain the setup of \S~\ref{subsec:conv}, where $M$ is smooth with an action of a finite group $G$ and an equivariant  projective map $M\to N$. Let $Z=M\times_N M$ with the induced diagonal $G$-action. 
 \begin{lemma}\label{lem:power_alg_hom}
Assume $E^*$ satisfies the $H_{\infty}$-assumption. 
Then, the power operation $P_n$ is an algebra homomorphism 
\[
P_n: E^*_G(Z)\to E^*_{G\wr\Sigma_n} (Z^n). 
\]
Moreover, \[
P_n: E^*_G(M)\to E^*_{G\wr\Sigma_n} (M^n). 
\]
is a module homomorphism under the algebra homomorphism above. 
\end{lemma}
\begin{proof}
It is clear that the power operation $P_n$ on $\MU^*_G(Z)$ commutes with pullback, pushforward and tensor product. 
Indeed, we only need to check this for cobordism classes of manifolds where for any $f: X\to Y$, we have 
$f^*: [M\to Y] \mapsto [f^*M\to X]$, and $f_*: [N\to X] \mapsto [N\to X\to Y]$. 
For two elements $[M\to Z]$ and $[N\to Z]$, the tensor product is given by $[M\times N\to Z]$. 
Thus, 
\[
\xymatrix{
[M\to Y] \ar@{|->}[r]^{f^*} \ar@{|->}[d]^{P_n}& [f^*M\to X]\ar@{|->}[d]^{P_n}\\
 [M^n\to Y^n]\ar@{|->}[r]^(0.4){(f^n)^*}& [f^*(M^n)=f^*(M)^n\to X^n]
}
\,\ \text{and}\,\ 
\xymatrix{
[N\to X] \ar@{|->}[r]^{f_*} \ar@{|->}[d]^{P_n}& [N\to X\to Y]\ar@{|->}[d]^{P_n}\\
 [N^n\to X^n]\ar@{|->}[r]^{(f^n)_*}& [N^n \to X^n\to Y^n]
}
\]
and 
\[
\xymatrix{
([M\to Z], [N\to Z]) \ar@{|->}[r]^{\otimes} \ar@{|->}[d]^{P_n}& [M\times N\to Z]\ar@{|->}[d]^{P_n}\\
 ([M^n\to Z^n], [N^n\to Z^n])\ar@{|->}[r]^{\otimes}& [M^n\times N^n \to Z^n]
}
\]
By the commutativity of the square \eqref{H_infty} and the fact that the top and the bottom maps commute with pullback and pushforward, we obtain the power operation $P_n$ on $E^*(Z)$ commutes with pullback, pushforward and tensor product. This implies the assertion, as the algebra structures of $E^*_G(Z)$ and $E^*_{G\wr \Sigma_n} (Z^n)$ are given by convolution.
Similar for the modules.  
\end{proof}

\begin{lemma}
Let $\Z/p\subseteq \Sigma_p$ be the subgroup of $p$-cycles, and let $\Delta:M\to M^p$ be the diagonal embedding. Then, the action of $\Z/p$ is free on the complement of $\Delta(M)$. In particular, pulling back  $\Delta^*:E^*_{\Z/p\times G}(M^p)\to E^*_{\Z/p\times G}(M)$ becomes an isomorphism after applying $-\otimes_{E^*_{\Z/p}}\Phi_1$.
\end{lemma}
\begin{proof}
This follows the same argument as that of Lemma~\ref{lem:diag_power}.
\end{proof}

In Lemma~\ref{lem:power_alg_hom}, take $n=p$ and then composing with the restriction from $G\wr\Sigma_p$ to the subgroup $\Z/p\times G$, we obtain an algebra homomorphism the target of which is the domain of the $\Delta^*$ from Lemma~\ref{lem:diag_power}.

Now we discuss the iteration of this process.
By condition~Definition~\ref{cond_c}(c), the composition $P_{p}\circ P_{p}:E_G^*(X)\to E_{G\wr \Sigma_p}^*(X^p)\to E^*_{(G\wr \Sigma_p)\wr \Sigma_p}X^{p^2}$ can be obtained by $E_G^*(X)\to E_{G\wr S_{p^2}}(X^{p^2})\to E^*_{(G\wr \Sigma_p)\wr \Sigma_p}X^{p^2}$. We are interested in restriction to  further subgroups $\Z/p\wr \Z/p$ and $\Delta(\Z/p)\times\Z/p$. 
 Lemma~\ref{lem:power_alg_hom} now yields that $\res_{(G\wr \Sigma_p)\wr \Sigma_p}^{G\times\Z/p\times\Z/p}\circ P_{p^2}:E^*_G(Z)\to E^*_{G\times\Z/p\times\Z/p}(Z^{p^2})$ is an algebra homomorphism of convolution algebra, and similar for the module coming from $M$.

\begin{lemma}
\label{lem:diag pullback}
Let $\Delta:X\inj X^{p^2}$ be the diagonal embedding, which is equivariant under $\Delta(\Z/p)\times\Z/p$-action. Then, the pullback via the diagonal 
$
 E^*_{\Delta(\Z/p)\times\Z/p}(X^{p^2})\to E^*_{\Delta(\Z/p)\times\Z/p}(X)$ is an isomorphism after applying $-\otimes_{E^*_{(\Z/p)^2}}\Phi_2$. 
\end{lemma}
\begin{proof}
This follows the same argument as that of Lemma~\ref{lem:t_diag}.
\end{proof}

To summarize, we have the following.
\begin{prop}\label{prop:power_alg}
There is an algebra structure on $\Phi_t\otimes_{E^*_{\Lambda_t}}E^*_{G\times\Lambda_t}(Z)$ and a module structure on $\Phi_t\otimes_{E^*_{\Lambda_t}}E^*_{G\times\Lambda_t}(M),$ such that $\Delta^*\circ P_{p^t}:E^*_{G}(Z)\to \Phi_t\otimes_{E^*_{\Lambda_t}}E^*_{G\times\Lambda_t}(Z)$ is an algebra homomorphism, and  $\Delta^*\circ P_{p^t}:E^*_{G}(M)\to \Phi_t\otimes_{E^*_{\Lambda_t}}E^*_{G\times\Lambda_t}(M)$ is a module homomorphism.
\end{prop}

We note that the algebra structure on $\Phi_t\otimes_{E^*_{\Lambda_t}}E^*_{G\times\Lambda_t}(Z)$ above comes from that of $E^*_{G^{p^t}}(Z^{p^t})$ via $\Delta^*$. In other words, on $\Phi_t\otimes_{E^*_{\Lambda_t}}E^*_{G\times\Lambda_t}(Z)$, there is a natural convolution product. However, the pullback $\Delta^*$ does not respect to this convolution multiplication (see the proof of Theorem \ref{thm:algebra hom}).

\subsection{Cyclic quivers of $\mathfrak{sl}_2$}
\label{ex:sl_2}
Let $\mathfrak{g}$ be $\mathfrak{sl}_2$. 
Let $w \in \N$ be a positive integer. 
Let 
\[
X=T^*\Gr(w)=(\sqcup_{v=0}^w T^*\Gr(v, w))
\]be the cotangent bundle of Grassmannian. 
That is, 
\[
T^*\Gr(w)=\{(V\subset W=\C^w, x)\mid V \,\ \text{is a subvector space of $W$}, x\in 
\End(\C^w),\text{such that}\,\  x(W)\subset V, x(V)=0
\}
\]
Let $ \calN\subset \End(\C^w)$ be the nilpotent cone. 
We have the natural map 
\[
\pi: T^*(\Gr(w))\to \calN, \,\ (V\subset W=\C^w, x)\mapsto x.
\] 
Let $\varmathbb{G}=\C^* \times \GL_{w}$. 
The action of $\GL_{w}$ is induced by its action on $\C^w$. 
The action of $\C^*$ is by $x\mapsto t x$. 
%This action is compatible with the action $(i, j)\mapsto (ti, tj)$ as $x:=j \circ i$. 
Under the action of $g_0\in \End(W)$, we have the eigenspace decomposition 
$W=\oplus_{s=0}^{p-1} W_{\zeta^{s}}$. 
Thus, 
\begin{align*}
\calN^{(g_0, \zeta)}
=&\{x\in \calN\mid g_0xg_0^{-1}=\zeta^{-1}x\}
=\{x\in \calN\mid x(W_{\zeta^{s}})\subset W_{\zeta^{s-1}}, s=0, \cdots, p-1\}, \\
\text{and} \,\ \,\
T^*(\Gr(w))^{(g_0, \zeta)}
=&\{(V=\oplus_{s=0}^{p-1} V_{\zeta^{s}}\subset W=\oplus_{s=0}^{p-1} W_{\zeta^{s}}, x)\mid \\&
V_{\zeta^s} \,\ \text{is a subvector space of $W_{\zeta^{s}}$}, x\in \End(W)\,\ 
\text{such that}\,\  
x(W_{\zeta^{s}})\subset V_{\zeta^{s-1}},
x(V)=0
\}. 
\end{align*}
Under the action of $b$, we have
\[
(\oplus_{s=0}^{p-1} V_{\zeta^s}, x)\mapsto 
(\oplus_{s=1}^{p} V_{\zeta^s}, bxb^{-1})
\]
% Under the action of $bg_0b^{-1}$, $W$ has the  eigenspace decomposition 
% $W=\oplus_{s=0}^{p-1} W'_{\zeta^{s}}$, where $W'_{\zeta^{s}}=W_{\zeta^{s-1}}$. 
% The following two fixed point sets are the same
% \begin{align*}
% &T^*(\Gr(w))^{(bg_0b^{-1}, \zeta)}=T^*(\Gr(w))^{(g_0, \zeta)}\\
% =&\{(\oplus_{s=0}^{p-1} V'_{\zeta^{s}}\subset \oplus_{s=0}^{p-1} W'_{\zeta^{s}}, x)\mid 
% V'_{\zeta^s} \,\ \text{is a subvector space of $W'_{\zeta^{s}}$}, x\in \End(W)\,\ 
% \text{such that}\,\  
% x(W'_{\zeta^{s}})\subset V'_{\zeta^{s-1}},
% x(V)=0
% \}. 
% \end{align*}

The components of $T^*(\Gr(w))^{(g_0, \zeta)}$ are labeled by the dimension vectors of $V$. 
Note that by the choice of $g_0$, all the eigenspaces $W_{\zeta^{s}}$ have the same dimension. The component of $T^*(\Gr(w))^{(g_0, \zeta)}$ is an 
\textit{even} component, if $dim(V_{\zeta^0})=dim(V_{\zeta^1})=\cdots =dim(V_{\zeta^{p-1}})$. 
Otherwise, it is  an \textit{uneven} component. 

In the case of $A_1$, we have the following result which is stronger than Proposition \ref{lem:iso to products}. 
\begin{prop}\label{prop:A_1}
We have the isomorphism 
\[
T^*\Gr(w)^{\im \alpha}|_{\even}\cong
\prod_{p} (T^*\Gr(w/p) )|_{\even}. 
\]
\end{prop}

\begin{proof}
If $p=2$, it is clear from their quiver descriptions that 
$\mathfrak{M}^{\cyc}_{(\overline{Q^{\heartsuit}})_2}(w)$ and $ \mathfrak{M}^{\cyc}_{\overline{Q^{\heartsuit}}\times \mu_2}(w)$ are tautologically the same as $T^*\Gr(w/p)\times T^*\Gr(w/p)$.
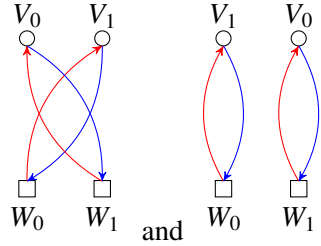
\[\begin{tikzpicture}
[square/.style={regular polygon,regular polygon sides=4}]
 \node at (0, 1.2 ) {$V_0$};
 \node at (1, 1.2 ) {$V_1$};
  \node at (0, -1.5 ) {$W_0$};
 \node at (1, -1.5 ) {$W_1$};
\foreach \x in {0, 1}{
	  \draw ($(\x,-0.1+1)$) circle (.1);
	}
\foreach \x in {0, 1}{
	  \node at (\x,-0.1-1) [square,inner sep=0.2em, draw]{};
	}
\draw[->, red] (0, -1) to[bend left] (1, 1-0.2);
\draw[<-, blue] (0, -1) to[bend right] (1, 1-0.2);
\draw[->, red] (1, -1) to[bend left] (0, 1-0.2);
\draw[<-, blue] (1, -1) to[bend right] (0, 1-0.2);
\end{tikzpicture}
\,\ \text{and}\,\ 
\begin{tikzpicture}
[square/.style={regular polygon,regular polygon sides=4}]
 \node at (0, 1.2 ) {$V_1$};
 \node at (1, 1.2 ) {$V_0$};
  \node at (0, -1.5 ) {$W_0$};
 \node at (1, -1.5 ) {$W_1$};
\foreach \x in {0, 1}{
	  \draw ($(\x,-0.1+1)$) circle (.1);
	}
\foreach \x in {0, 1}{
	  \node at (\x,-0.1-1) [square,inner sep=0.2em, draw]{};
	}
\draw[->, red] (0, -1) to[bend left] (0, 1-0.2);
\draw[<-, blue] (0, -1) to[bend right] (0, 1-0.2);
\draw[->, red] (1, -1) to[bend left] (1, 1-0.2);
\draw[<-, blue] (1, -1) to[bend right] (1, 1-0.2);
\end{tikzpicture}
\]
For general $p$, the $b$ identifies $W_1, W_{\zeta}, \cdots, W_{\zeta^{p-1}}$. It induces an isomorphism between the two Grassmannians
\[
b: \Gr(a, W_{\zeta^i}) \cong
\Gr(a, W_{\zeta^j}).
\]
where $a=\dim(V_{0})$. 
Hence, on the tautological sequences
\[
0\to \calV_{\zeta^{i+1}} \to W_{\zeta^i} \to \calQ_i\to 0, \,\ \,\ 
0\to \calV_{\zeta^{j+1}} \to W_{\zeta^j} \to \calQ_j\to 0,
\]
we have an isomorphism of vector bundles 
\[
b^*\calV_{\zeta^{j+1}}
\cong \calV_{\zeta^{i+1}}.
\]
The blue arrows in the following picture on the left become the embeddings of the tautological subbundles on the Grassmannians $\Gr(V_{\zeta^{i+1}}, W_{\zeta^i})$. 
They can be identified with the blue arrows in the picture on the right. 

The red arrows on the left, a priori, are bundles maps from $W_{\zeta^i}$ to $V_{\zeta^{i-1}}$. Using the identification of $b^*\calV_{\zeta^{i-1}}
\cong \calV_{\zeta^{i+1}}$, the red arrows 
define elements in $\sHom_{\Gr(a, W_{\zeta^i})}(W_{\zeta^i}, V_{\zeta^{i+1}})$ and can be identified with the red arrows in the picture on the right. This gives an isomorphism between $\Gr(w)^{\im \alpha}|_{\even}$ and $\prod_{p} \Gr(w/p)|_{\even}$. 
\[\begin{tikzpicture}
[square/.style={regular polygon,regular polygon sides=4}]
 \node at (0, 1.3 ) {$V_1$};
 \node at (1, 1.2 ) {$V_{\zeta}$};
  \node at (2, 1.2 ) {$V_{\zeta^2}$};
   \node at (4, 1.2 ) {$V_{\zeta^{p-2}}$};
    \node at (5, 1.2 ) {$V_{\zeta^{p-1}}$};
     \node at (6, 1.3 ) {$V_{1}$};
  \node at (0, -1.5 ) {$W_1$};
 \node at (1, -1.5 ) {$W_{\zeta}$};
  \node at (2, -1.5 ) {$W_{\zeta^2}$};
   \node at (4, -1.5 ) {$W_{\zeta^{p-2}}$};
    \node at (5, -1.5 ) {$W_{\zeta^{p-1}}$};
    \node at (6, -1.5 ) {$W_{1}$};
     \node at (2.5, 0 ) {$\cdots$};
\foreach \x in {0, 1, 2, 4, 5, 6}{
	  \draw ($(\x,-0.1+1)$) circle (.1);
	}
\foreach \x in {0, 1, 2, 4, 5, 6}{
	  \node at (\x,-0.1-1) [square,inner sep=0.2em, draw]{};
	}
	%=====
\foreach \x in {1, 2, 4, 5, 6}{	
\draw[red,->] (\x, -1) to[bend left] (\x-1, 1-0.2);
}
\foreach \x in {1, 2, 4, 5, 6}{	
\draw[blue, <-] (\x-1, -1) to[bend left] (\x, 1-0.2);
\draw[dashed] (0, -1.2) to[bend right] (6, -1.2);
\draw[dashed] (0, 1) to[bend left] (6, 1);
}
\end{tikzpicture}
\,\ \,\ \,\ \text{and}\,\ \,\  \,\
\begin{tikzpicture}
[square/.style={regular polygon,regular polygon sides=4}]
 \node at (0, 1.2 ) {$V_{\zeta}$};
 \node at (1, 1.2 ) {$V_{\zeta^2}$};
  \node at (3, 1.2 ) {$V_{1}$};
  \node at (0, -1.5 ) {$W_1$};
 \node at (1, -1.5 ) {$W_{\zeta}$};
  \node at (3, -1.5 ) {$W_{\zeta^{p-1}}$};
  \node at (2, 0 ) {$\cdots$};
\foreach \x in {0, 1, 3 }{
	  \draw ($(\x,-0.1+1)$) circle (.1);
	}
\foreach \x in {0, 1, 3}{
	  \node at (\x,-0.1-1) [square,inner sep=0.2em, draw]{};
	}
	\foreach \x in {0, 1, 3}{
\draw[blue, <-] (\x, -1) to[bend left] (\x, 1-0.2);
\draw[red,->] (\x, -1) to[bend right] (\x, 1-0.2);}
\end{tikzpicture}
\]
In the above picture, the two vertices
connected by the dashed arrows are identified. 
\end{proof}

\subsection{$A_1$-case via power operations}\label{sec:a1case}
In the case of $A_1$, by Proposition \ref{prop:A_1}, we have the isomorphisms
\[
Z(w)^{\im \alpha}|_{\even} \cong \prod_{p^t} Z(w/p^t)|_{\even}, \,\ 
\fM(w)^{\im \alpha}|_{\even} \cong \prod_{p^t} \fM(w/p^t)|_{\even}. 
\]
In this case, we give a ``correct" construction of the Frobenius morphism where the power operation naturally show up.

Let $i_\alpha :Z(w)^{\im\alpha}|_{\even}\inj Z(w)$ be the embedding. We have the pullback map
\[
i_\alpha^*: E^*_{n,\varmathbb{G}}(Z(w))
\to  E^*_{n,\langle b\rangle\wr (\alpha\times H)}(Z(w)^{\im \alpha}). 
\]
Let $e_\alpha$ be the Euler class of the normal bundle of the second factor $\fM(w)^{\im\alpha}\times \fM(w)^{\im\alpha}\subseteq \fM(w)\times\fM(w)$ restricted to $Z(w)^{\im\alpha}|_{\even}\subseteq \fM(w)^{\im\alpha}|_{\even}\times \fM(w)^{\im\alpha}|_{\even}$. 
Observe  that $e_\alpha$ is invertible in $S_{\Lambda}^{-1}E^*_{n,\langle b\rangle\wr (\alpha\times H)}(Z(w)^{\im\alpha})$ thanks to Lemma~\ref{lem:lambda_invertible}. Proposition~\ref{prop:localization_convolution} then implies that 
\[
\frac{i_\alpha^*}{e_\alpha}: E^*_{n,\varmathbb{G}}(Z(w))\to S_{\Lambda}^{-1}E^*_{n,\langle b\rangle\wr (\alpha\times H)}(Z(w)^{\im\alpha}|_{\even})
\cong 
S_{\Lambda}^{-1}E^*_{n,\langle b\rangle\wr (\alpha\times H)}(\prod_{p} Z(w/p)|_{\even})
\]
is an algebra homomorphism, where the latter isomorphism follows from Proposition \ref{prop:A_1}. Observe also that $S_b^{-1}E^*_{n,\langle b\rangle\wr (\alpha\times H)}(\prod_{p^t} Z(w/p^t)|_{\even})\cong S_b^{-1}E^*_{n,\langle b\rangle\wr (\alpha\times H)}(\prod_{p^t} Z(w/p^t))$ thanks to Proposition~\ref{prop:folding}.

Composing these and using \eqref{eqn:LocPhi}, we obtain
\begin{equation}\label{eq:pullback}
\frac{i_\alpha^*}{e_\alpha}:
E^*_{n,\varmathbb{G}}(Z(w))\to S_\Lambda^{-1}S_b^{-1}E^*_{n,\langle b\rangle\wr (\alpha\times H)}(\prod_{p^t}Z(w/p^t))
\to \Phi_t^{L_{K_{n-t}E}}\otimes_{C_{n-t,\langle b\rangle}}C^*_{n-t,\langle b\rangle\wr  H}(\prod_{p^t}Z(w/p^t)).
\end{equation}
Pullback along the diagonal embedding $
\Delta: Z(w/p^t)\subset \prod_{p^t}Z(w/p^t)$, 
by Lemma \ref{lem:diag pullback}, we get an isomorphism 
\[
\Delta^*: 
\Phi_t^{L_{K_{n-t}E}}\otimes_{C_{n-t,\langle b\rangle}}C^*_{n-t,\langle b\rangle\wr  H}(\prod_{p^t}Z(w/p^t))
\cong
\Phi_t^{L_{K_{n-t}E}}\otimes_{C_{n-t,\langle b\rangle}}C^*_{n-t,\langle b\rangle\wr  H}(Z(w/p^t)). 
\]

After inverting $p$, $p^t$ is a unit, then, the $p^t$-series of $F$
\[
[p^t]_F(z)=p^t z+\cdots
\]
becomes an isomorphism of formal group laws. 
Therefore, by Proposition \ref{p series}, the following composition is an isomorphism of $\Phi_t^{L_{K_{n-t}E}}$-modules
\[
\Phi_t^{L_{K_{n-t}E}}\otimes_{C_{n-t}}C^*_{n-t,H}(Z(w/p^t))\xrightarrow{P_{p^t}^{C_{n-t,\langle b\rangle}}}
\Phi_t^{L_{K_{n-t}E}}\otimes_{C_{n-t,\langle b\rangle}}C^*_{n-t, \langle b\rangle\wr  H}(\prod_{p^t}Z(w/p^t))
\xrightarrow[\cong]{\Delta^*}
\Phi_t^{L_{K_{n-t}E}}\otimes_{C_{n-t}}C^*_{n-t,  H}(Z(w/p^t)).
\]
In particular, $P_{p^t}^{C_{n-t}}$ is an isomorphism and the inverse map exists. 

Composing \eqref{eq:pullback} and the inverse of the power operation, we obtain an algebra homomorphism by Lemma~\ref{lem:power_alg_hom}.  
\[
\tilde{\Fr}_{n, n-t}: E^*_{n,\varmathbb{G}}(Z(w))
\xrightarrow{ \frac{i_\alpha^*}{e_\alpha}:
} 
\Phi_t^{L_{K_{n-t}E}}\otimes_{C_{n-t,\langle b\rangle}}C^*_{n-t,\langle b\rangle\wr  H}(\prod_{p^t}Z(w/p^t))
\xrightarrow[\cong]{ (P_{p^t}^{C_{n-t}})^{-1}
} 
\Phi_t^{L_{K_{n-t}E}}\otimes_{C_{n-t}}C^*_{n-t,  H}(Z(w/p^t)). 
\]

\section{The Frobenius phenomenon on representations}
In this section we gather some basic properties of the effect of Frobenii on representation theory of the convolution algebras. 
\label{sec:Frob phen}

\subsection{The standard modules}
Assume $p^t\mid w$. We define a morphism on the modules
\begin{align}
\Fr _{n, n-t}^{\mathfrak{M}}: E^*_{n, \varmathbb{G}}( \mathfrak{M}(w))
\to & 
\Phi_t(C^*_{n-t})\otimes_{C^*_{n-t}(\pt)} 
C^*_{n-t,H}(\mathfrak{M}(\frac{w}{p^t})), 
\end{align} 
similar to \S~\ref{subsec:def_Frob}, 
with  the space $Z(w)$ substituted by $\mathfrak{M}(w)$, and do not  multiply by $\frac{1}{e}$ in (d).
More precisely, $\Fr_{n, n-t}^{\mathfrak{M}}$ is defined as the composition of the following morphisms. 

\begin{enumerate}
\item[(a)]
The change of group morphism induced from \eqref{eq:group}
\[
E^*_{\varmathbb{G}}( \mathfrak{M}(w)) \to E^*_{\langle b\rangle\wr (\alpha\times H)}(\mathfrak{M}(w)).
\]
\item [(b)] 
Let $i: \mathfrak{M}(w)^{\im \alpha}\mid_{\Delta}\inj \mathfrak{M}(w)$ be embedding of the fixed point set of $im(\alpha)$-action. 
The pullback $i^*$ along $i$. 
\[
i^*: E^*_{\langle b\rangle\wr (\alpha\times H)}(\mathfrak{M}(w))
\to E^*_{\langle b\rangle \times \Lambda_t\times H}( \mathfrak{M}(w)^{ \im \alpha}\mid_{\Delta}).
\]
\item[(c)]
The localization
\[
E^*_{\langle b\rangle \times \Lambda_t\times H}( \mathfrak{M}(w)^{ \im \alpha}\mid_{\Delta})
\to S_{\Lambda}^{-1} S_b^{-1}E^*_{\langle b\rangle \times \Lambda_t\times H}( \mathfrak{M}(w)^{ \im \alpha}\mid_{\Delta}).
\]

\item[(d)]
The natural map 
\[
S_{\Lambda}^{-1} S_b^{-1}E^*_{\langle b\rangle \times \Lambda_t\times H}( \mathfrak{M}(w)^{ \im \alpha}\mid_{\Delta})
\to \Phi_t^{L_{K_{n-t}E}}\otimes_{C^0_{n-t}} 
C^*_{n-t, H}(\mathfrak{M}(\frac{w}{p^t})).
\]
\end{enumerate}

\begin{theorem}\label{mainthm}
The morphism $\Fr_{n, n-t}^{\mathfrak{M}}$ is compatible with the algebra homomorphism $\Fr_{n, n-t}$. 
That is, we have the following diagram
\[
\begin{xymatrix}@R=0.3em{
% P_n(w) \ar@{->>}[r]^(0.4){\chi} \ar@{^{(}->}[d]&\widetilde{P}_{n-1}(w/p)\ar@{^{(}->}[d]\\
 E^*_{\varmathbb{G}}( Z) \ar[r]^(0.3){\Fr_{n, n-t}} &\Phi_t^{L_{K_{n-t}E}}\otimes_{C^0_{n-t}} 
C^*_{n-t, H}(Z(\frac{w}{p^t}))\\
&\\
\ar@(ul,ur)[]&\ar@(ul,ur)[]\\
 E^*_{\varmathbb{G}}( \mathfrak{M}(w)) \ar[r]^(0.3){\Fr_{n, n-t}^{\mathfrak{M}}} &
\Phi_t^{L_{K_{n-t}E}}\otimes_{C^0_{n-t}} 
C^*_{n-t, H}(\mathfrak{M}(\frac{w}{p^t}))
}
\end{xymatrix}
\]
\end{theorem}
\begin{proof}
It follows from a similar proof as Theorem \ref{thm:algebra hom}, which relies on Proposition~\ref{prop:localization_convolution}, with the following changes. 
Let $M_1=M_2:=\mathfrak{M}(w)$ and $M_3:=\{\pt\}$. 
Let $Z_{12}:=\mathfrak{M}(w)\times_{\mathfrak{M}_0}\mathfrak{M}(w) \subset M_1\times M_2$, $Z_{23}:=M_2\times M_3=\mathfrak{M}(w)$ and $Z_{13}:=M_1\times M_3=\mathfrak{M}(w)$. Let $p_{ij}: M_1\times M_2\times M_3\to M_i\times M_j$ be the projection. Thus, we have the convolution diagram which gives the action of $E^*_{n, \varmathbb{G}}(Z(w)$ on $E^*_{n, \varmathbb{G}}(\mathfrak{M}(w)$. 
\[
\xymatrix{
&M_1\times M_2\times M_3 \ar[ld]_{p_{12}} \ar[rd]^{p_{23}} \ar[d]^{p_{13}}&\\
Z_{12}\subseteq M_1\times M_2 & M_1\times M_3 \supseteq Z_{13}& M_2\times M_3 \supseteq Z_{23}
}
\]
Choose $e_i\in E_{A }^*(M_i^{\im \alpha})$, $i=1, 2$ and $e_3$ to be $1$. We can repeat the argument of the proof of Proposition~\ref{prop:localization_convolution} under this setting. 

\end{proof}

\subsection{The irreducible modules}
\label{subsec:irreducible}
For  the Morava $E_n^*$-theory,  $w\in \N^I$,  
and  a group homomorphism $\gamma: \Lambda_n \to \varmathbb{G}:=\GL_w\times \C^*$, let us consider representations of the convolution algebra $E_{n,\varmathbb{G}}^*(Z(w))$ valued in a field $\mathbb{F}$, which we assume to be endowed with  a ring homomorphism  $E^*_{n,\Lambda_n}(\pt)\to \mathbb{F}$.  

We start with the algebra $U_w^n(\gamma, \mathbb{F})\subset \mathbb{F} \otimes_{E_{n,\Lambda_n}} E^*_{n, \varmathbb{G}} (Z(w))
$, which acts on the module 
\begin{equation}\label{module:V}
V_w^n(\gamma, \mathbb{F}):=
\oplus_{v\in \N^I} (V_w^n(\gamma, \mathbb{F})_v)
=\oplus_{v} (\mathbb{F}\otimes_{E^*_{n,\Lambda_n}} E^*_{n, \varmathbb{G}} (\mathfrak{M}(v, w))).
\end{equation}
This module $V_w^n(\gamma, \mathbb{F})$ is an analogue of the local Weyl module \cite{Kash} of the quantum loop algebra at a root of unity. 
Define the highest weight vector $|\emptyset \rangle$ of $V_w^n(\gamma, \mathbb{F})$ such that $|\emptyset \rangle$ is the unity of $\mathbb{F}$ under the isomorphism
\[
\mathbb{F}\otimes_{E^*_{n,\Lambda_n}} E^*_{n, \varmathbb{G}} (\mathfrak{M}(0, w))=\mathbb{F}\otimes_{E_{n,\Lambda_n}} E_{n, \varmathbb{G}} (\pt). 
\]
Let $L_w^n(\gamma, \mathbb{F})$ be the irreducible quotient module of $V_w^n(\gamma, \mathbb{F})$ with the highest weight vector $|\emptyset \rangle$. 

Let $X$ be the weight lattice. We identify $X$ with $\Z^I$ by choosing the fundamental weights $\{\overline{\omega}_1, \overline{\omega}_2, \cdots, \overline{\omega}_n\}$. 
The decomposition $V_w^n(\gamma, \mathbb{F})=\oplus_{v\in \N^I} V_w^n(\gamma, \mathbb{F})_v$ gives the weight space decomposition, where each weight space $V_w^n(\gamma, \mathbb{F})_v$ has weight 
\[
\mu(v, w)=\sum_k w_k\overline{\omega}_k-\sum_k v_k\alpha_k, \,\ 
\text{where $\{\alpha_1, \cdots, \alpha_n\}$ are the roots.}
\]
It induces the weight space decomposition of the irreducible quotient $L_w^n(\gamma, \mathbb{F})=\oplus_{v\in \N^I} L_w^n(\gamma, \mathbb{F})_v$, 
where $L_w^n(\gamma, \mathbb{F})_v$ is the quotient of $V_w^n(\gamma, \mathbb{F})_v$ and has weight $\mu(v, w)$.

%=============
\subsection{Group homomorphisms}\label{subsec:group_hom}
In this section, we drop the assumption that $p^t\mid w$ and consider the weight $w$ in general. 
Fix a positive integer $n$ and a prime number $p$. Let $w\in \N^I$ be any dimension vector. 
Consider the set of group homomorphisms 
\[
\alpha: \Lambda_n \to \GL_{w}\times \C^*
\] satisfying the following property $\mathrm{P}_{w}^n$. 

(1): In the case when $w\in X^+_{\red}$, for any $n\in \N$, write $\alpha=(\alpha_1, \alpha_2, \cdots , \alpha_n)$ for $\alpha_i: \Z/p\to \GL_{w}\times \C^*$. We have $\alpha_1(1)$ is a diagonal matrix of the form
\[
\alpha_1=
\left(\begin{bmatrix}
1 &0& 0&\cdots &0\\
0& \zeta & 0&\cdots &0\\
0& 0& \zeta^2  &\cdots &0\\
&\vdots && \ddots&\vdots\\
0& 0&  0 & \cdots& \zeta^{w-1} \end{bmatrix},\zeta\right)\]
and 
\[
\alpha_i=t_i \alpha_1, i=2, 3, \cdots, n,  \,\
\text{where $t_i$ is any $p$-th root of $1$.}
\]

(2): If $w=w_0+pw_1$ for $w_0\in X^+_{\red}$, 
then $\alpha$ is of the form $\alpha=\alpha^{(0)}\times \alpha^{(1)}$, where
\begin{align*}
& \alpha^{(0)}: \Lambda_n \to T_{w_0}\times \C^*\subset \GL_{w_0}\times \C^*, \\
&\alpha^{(1)}: \Lambda_n \to T_{pw_1}\times \C^*\subset\GL_{pw_1}\times \C^*. 
\end{align*}
Here $T_w\subset \GL_w$ is the maximal torus consisting of diagonal matrices. 
Furthermore, $ \alpha^{(0)}$ satisfies property $\mathrm{P}^n_{w_0}$, and 
\[
\alpha^{(1)}=\alpha'\oplus (b \alpha' b^{-1})\oplus (b^2\alpha'b^{-2})\oplus \cdots \oplus( b^{p-1} \alpha'b^{1-p}), 
\]
for some $\alpha': \Lambda_{n-1} \to \GL_{w_1}\times \C^*$ satisfying $\mathrm{P}^{n-1}_{w_1}$. Here we decompose the vector space $W^{(1)}$ (with $\dim(W^{(1)})=pw_1$) as
$W^{(1)}=W^{(1)}_{1}\oplus W^{(1)}_{\zeta}\oplus\cdots \oplus W^{(1)}_{\zeta^{p-1}}$ (with $\dim(W^{(1)}_{\zeta^i})=w_1$), and $b$ is the automorphism that identifies the eigenspaces \[
b: W^{(1)}_{1} \xrightarrow{\cong} W^{(1)}_{\zeta}, 
b: W^{(1)}_{\zeta}\xrightarrow{\cong} W^{(1)}_{\zeta^2}, \cdots, 
b: W^{(1)}_{\zeta^{p-2}}\xrightarrow{\cong} W^{(1)}_{\zeta^{p-1}}, 
b: W^{(1)}_{\zeta^{p-1}}\xrightarrow{\cong} W^{(1)}_{1} 
\]

\begin{example}
We give one example that satisfies the property $\mathrm{P}_{w}^n$, 
for illustration purpose. We take 
$w=p^t w'$, where $w'\in X^+_{\red}$. 
We assume $n\geq t$. 
We decompose the vector space $W$ as 
\[
W=\oplus_{0\leq a_i \leq p-1} W_{\zeta^{a_1}, \cdots, \zeta^{a_t}}, \,\ \text{where}\,\  \dim(W_{\zeta^{a_1}, \cdots, \zeta^{a_t}})=w'. 
\]
Choose $(\alpha_1, \alpha_2, \cdots , \alpha_t)$ such that
$W_{\zeta^{a_1}, \cdots, \zeta^{a_t}}$ is the common eigenspace for $(\alpha_1, \alpha_2, \cdots , \alpha_t)$ with eigenvalue $\zeta^{a_i}$ for $\alpha_i$, $i=1, \cdots, t$. Choose $\alpha_{t+1}=\cdots=\alpha_{n}$ such that on each $W_{\zeta^{a_1}, \cdots, \zeta^{a_t}}$,  $\alpha_{t+1}$ is the following  diagonal matrix
\[\left(\begin{bmatrix}
1&0& 0&\cdots &0\\
0& \zeta  & 0&\cdots &0\\
0& 0& \zeta^{2} &\cdots &0\\
&\vdots && \ddots&\vdots\\
0& 0&  0 & \cdots& \zeta^{w'-1} \end{bmatrix},\zeta\right).\]
This homomorphism $\alpha$ satisfies $\mathrm{P}_{w}^n$. 
\end{example}

\begin{corollary}
\label{cor:8.3}
\begin{enumerate}
\item If $w\in X^+_{\red}$, we have
\[\mathfrak{M}(w)^{\im \alpha_1}=\mathfrak{M}(w)^{\im \alpha_2}=\cdots =\mathfrak{M}(w)^{\im \alpha_n}.\] 
\item If $p^t|w$ for some integer $t$ such that $1< t\leq n$, then the common eigenspaces for $\alpha_1, \alpha_2, \cdots, \alpha_t$ all have dimension $w/(p^t)$. That is, let $W=\oplus W_{a_1, \cdots, a_t}$ where $W_{a_1, \cdots , a_t}:=
\{w\in W\mid 
\alpha_i(1)w=a_iw, i=1, \cdots, t\}$, then  
\[
\dim( W_{a_1, \cdots, a_t})=w/(p^t).
\]
\end{enumerate}
\end{corollary}

\subsection{Frobenius pullback of irreducible representations}\label{subsec:irred_mod}

Now we assume that the field $\mathbb{F}$ factors through $E_{n}\to \Phi^{L(C^\wedge_{n-t})}_t$, and that $\gamma$ is obtained from $\gamma'$ via \S~\ref{subsec:group_hom}(2) that is, there exists a group homomorphism $\alpha: \Lambda_t\to \varmathbb{G}$ be as \S\ref{sec:quantum Frob} so that $\gamma=\alpha\times\gamma'$. 
Remark~\ref{rmk:coef}(1) provides us with a map $E_{n-t}\to \mathbb{F}$, and hence $U^{C_{n-t}}_{w/p^t}(\gamma',\mathbb{F})\cong U^{n-t}_{w/p^t}(\gamma',\mathbb{F})$.
In the definition of Frobenius homomorphism  \S\ref{sec:quantum Frob}, take $H$ to be $\gamma'$ we obtain
\[
\Fr_{n, n-t}: U^n_w(\gamma,\bbF)
\surj 
U_{w/p^t}^{n-t}(\gamma', \mathbb{F})
\]
which is clearly surjective. 
Let $L$ be any irreducible representation of the target $U_{w/p^t}^{n-t}(\gamma', \mathbb{F})$.
Let $\Fr_{n, n-t}^*L$ be the representation of $U_w^n(\gamma, \mathbb{F})$ that is isomorphic to $L$ as vector spaces and the $U^n_w(\gamma,\bbF)$-module structure is obtained by $\Fr_{n, n-t}$. The surjectivity of $\Fr_{n, n-t}$ implies that  $\Fr_{n, n-t}^*L$ is irreducible.

Taking $L$ to be $L_{w/p^t}^{n-t}(\gamma', \mathbb{F}) $, now we observe that the action of $U^n_w(\gamma,\bbF)$ on $\Fr_{n, n-t}^*L_{w/p^t}^{n-t}(\gamma', \mathbb{F})$ factors through that on $V_w^n(\alpha\times \gamma', \mathbb{F})$ and the map $V_w^n(\alpha\times \gamma', \mathbb{F})\to \Fr_{n, n-t}^*L_{w/p^t}^{n-t}(\gamma', \mathbb{F})$ preserves the highest weight vector and hence implies the corollary below. 

\begin{corollary}\label{cor:FrobIrr}
Under the above assumption of $\mathbb{F}$ and assume further that $w$ is divisible by $p^t$.
Let $\alpha: \Lambda_t\to \varmathbb{G}$ be as \S\ref{sec:quantum Frob}. 
Then, we have an isomorphism of irreducible representations
\[
L_w^n(\alpha\times \gamma', \mathbb{F})
\cong 
\Fr_{n, n-t}^*(L_{w/p^t}^{n-t}(\gamma', \mathbb{F}) ). 
\]
\end{corollary}
\begin{proof}
Indeed, by construction the map $\Fr_{n,n-t}$ factors through $E^*_{n,\langle b\rangle\times \Delta(\gamma)}(Z(w))$ where $b$ is as in \S~\ref{sec:quantum Frob}, which via $-\otimes_{E^*_{n,\langle b\rangle\times \Delta(\alpha)}}\Phi_t^{L(C^\wedge_{n-t})}$ acts on $E^*_{n,\langle b\rangle\times \Delta(\alpha)}(\fM(w))\otimes_{E^*_{n,\langle b\rangle\times \Delta(\alpha)}}\Phi_t^{L(C^\wedge_{n-t})}\to V_w^n(\alpha\times \gamma', \mathbb{F})$.
Theorem~\ref{mainthm} then concludes that 
\[V_w^n(\alpha\times \gamma', \mathbb{F})\to \Fr_{n, n-t}^*V_{w/p^t}^{n-t}(\gamma', \mathbb{F})\to \Fr_{n, n-t}^*L_{w/p^t}^{n-t}(\gamma', \mathbb{F}).\] The highest weight vector  $|\emptyset\rangle$ up to scalar is the only non-zero vector in $E^*_{n,\varmathbb{G}}(\fM(0,w))$ which is mapped isomorphically under the Frobenius. 
\end{proof}

\section{The monoidal structure}

We gather some structures of the convolution algebras considered in the present paper, including a localized coproduct. The convolution algebra is not a Hopf algebra, but these structures are sufficient  to carry out constructions needed for later. In particular, we prove  an analogue of Serre relation for  integrable representations, which  is used in the proof of the Steinberg tensor product theorem in the next section.

\subsection{Coproduct}\label{subsec:loc_coprod}
We construct a localized coproduct. 
This depends on a choice of $\gamma=\gamma_1\times\gamma_2$ with $\gamma_i:\Lambda_n\to \GL_{w_i}\times\bbC^*$ and $w=w_1+w_2$, together with a decomposition $W\cong W_1\oplus W_2$ as $I$-graded vector spaces. 
Such a choice determines the embedding of groups $\GL_{w_1}\times\GL_{w_2}\subseteq \GL_w$, and  an embedding $i:\fM(w_1)\times\fM(w_2)\inj \fM(w)$ which is compatible with the actions.
Let $G$ be a subgroup of $\GL_{w_1}\times\GL_{w_2}\times\bbC^*$ containing the image of $\gamma$ such that $G=G_1\times G_2$, where $G_i$ is a subgroup $\GL_{w_i}\times\bbC^*$ containing $\im \gamma_i$. 
It induces the following embedding of the Steinberg varieties. 
\[
\xymatrix{
Z(w_1)\times Z(w_2)\ar@{^{(}->}[r]^{i} \ar@{^{(}->}[d]& Z(w)\ar@{^{(}->}[d]\\
(\fM(w_1)\times \fM(w_1))\times 
(\fM(w_2)\times \fM(w_2)) \ar@{^{(}->}[r] &
\fM(w)\times \fM(w)
}
\]
Let $N$ be the normal bundle of $\fM(w_1)\times\fM(w_2)\inj \fM(w)$. 
We denote by $N_2$ the pullback of $N$ via the second projection $(\fM(w_1)\times\fM(w_2))^2\xrightarrow{pr_2} \fM(w_1)\times\fM(w_2)$. Its restriction to the subvariety $Z(w_1)\times Z(w_2)$ is still denoted by $N_2$.
We define 
$\Delta_{w_1,w_2}^{\loc}:E^*_{n,G}(Z(w))\to E^*_{n,G}(Z(w_1)\times Z(w_2))_{\loc}$ as $\Delta_{w_1,w_2}^{\loc}:=\frac{i^*}{e(N_2)}$ which induces a map
\[ 
\Delta_{w_1, w_2}^{\loc}: U_w(G,\mathbb{F}) \to U_{w_1}(G_1,\mathbb{F})\otimes U_{w_2}(G_2,\mathbb{F})_{\loc}.\]  We simply write $\Delta^{\loc}$ if $w_1$ and $w_2$ are understood from context. 
\begin{prop}\label{prop:coproduct}
The map $\Delta^{\loc}$ is co-associative in the obvious sense, and is an algebra homomorphism.
\end{prop}

\begin{proof}
Let $\gamma_i:\Lambda_n\to \GL_{w_i}\times\bbC^*$ for $i=1,2,3$ be group homomorphisms with $\gamma=\gamma_1\times\gamma_2\times\gamma_3$. 
We have the embeddings
\begin{align*}
&i_{w_a,w_b}:\fM(w_a)\times\fM(w_b)\inj \fM(w_a+w_b), 1\leq a< b \leq 3\\
&i_{w_a+w_b,w_c}:\fM(w_a+w_b)\times\fM(w_c)\inj \fM(w_a+w_b+w_c),  1\leq a\neq  b \neq c\leq 3\\
& i_{w_1,w_2,w_3}:\fM(w_1)\times\fM(w_2)\times\fM(w_3) \inj \fM(w_1+w_2+w_3). 
\end{align*}
The normal bundle of $i_{w_1,w_2,w_3}$ is an extension of that of $i_{w_1+w_2,w_3}$ by that of $i_{w_1,w_2}\times\id_{\fM(w_3)}$; similarly it is an extension of the normal bundle of $i_{w_1, w_2+w_3}$ by that of $\id_{\fM(w_1)}\times i_{w_2,w_3}$. This implies the co-associativity. 

The fact that $\Delta^{\loc}$ is an algebra homomorphism follows from a similar proof as that of Theorem \ref{thm:algebra hom}. 
\end{proof}

Let $U_w:=U_w(G, \mathbb{F})$ denote the convolution algebra for short. The flipping of the two copies of the quiver varieties 
$
\mathfrak{M}(w)
\times_{\mathfrak{M}_0(w)} \mathfrak{M}(w)
\xrightarrow{(1,2)} \mathfrak{M}(w)
\times_{\mathfrak{M}_0(w)} \mathfrak{M}(w)
$ induces an anti-homomorphism, or equivalently a homomorphism
\[
\sigma: U_w\to U_w^{\op}
\]
on the convolution algebras. By construction, $M_{w}:=E^*_{n,G}( \mathfrak{M}(w))$ is a right module of $U_w$. 
The contragredient dual $M_{w}^\vee$ is the right module of $U_w$ obtained from the obvious left module structure of the dual vector space, composed with $\sigma$. 

Although we will not consider this in the present paper, it is worth to remark that the duality can also be realized via 
the opposite stability condition. More precisely, we have the isomorphism of quiver varieties with opposite stability conditions induced by the longest element of the Weyl group
\[
\mathfrak{M}^{\theta^+}(w) \xrightarrow[\cong]{w_0} 
\mathfrak{M}^{\theta^-}(w). 
\]
Then $U_w$ acts on $M_w \cong E^*_{n,G}(\mathfrak{M}^{\theta^-}(w))$ from the left using the above isomorphism. 
Moreover, the argument \cite[(3.6.10)]{CG} implies that the standard modules of the convolution algebras are self-dual.

%===============Section about adjoint representation======
\subsection{The adjoint representations}
\label{act on hom}
To begin with, we show there exists an antipode $S$ for the coproduct $\Delta^{\loc}$. We have localized Drinfeld coproduct $\Delta^{\Dr}$ on the COHA $\calH(E^*, Q)$ constructed in \cite[Section 3]{YZ2}. 
When $E^*$ is the cohomology theory, the antipode $S^{\Dr}$ of $\Delta^{\Dr}$ is constructed in \cite[Section 2.1]{RSYZ}. In general, the same formula works for an arbitrary formal group law, hence gives the antipode  $S^{\Dr}$ of $\Delta^{\Dr}$ for $\calH(E^*, Q)$ when $E^*$ is a generalized oriented cohomology theory.  
Let $\Phi_w:  \calH(E^*, Q) \to U_w(\varmathbb{G}, E^*)$ be the algebra homomorphism from \eqref{eq:algebra hom}. 
The pair $(\Delta^{\Dr}, S^{\Dr})$ descents to the pair of Drinfeld coproduct and the corresponding antipode of the covolution algebra $U_w(\varmathbb{G}, E^*)$ via $\Phi_w$, still denoted by $(\Delta^{\Dr}, S^{\Dr})$. 

The two localized coproducts $\Delta^{\loc}$ and $\Delta^{\Dr}$ of $U_w:=U_w(\varmathbb{G}, E^*)$ are related by a conjugation. That is, there exists an invertible element $F$ in $E^*_{\GL_{w_1}\times \GL_{w_2}\times \C^*}(\mathfrak{M}(w_1)\times  \mathfrak{M}(w_2))_{\loc}$, for $w=w_1+w_2$ such that, 
\[
\Delta^{\Dr}=F^{-1}* \Delta^{\loc} * F: U_w\to (U_{w_1}\otimes U_{w_2})_{\loc}. 
\]
See \cite[Section 7.2]{VV02} for the construction of $F$ when $E^*$ is the $K$-theory, where $\Delta^{\loc}$ is denoted by $\Delta'_{U}$, $\Delta^{\Dr}$ is denoted by $\Delta^{\diamond}_{U}$, and $F$ is denoted by $\Omega$. 
In other words, $\Delta^{\loc}$ is obtained from $\Delta^{\Dr}$ by twisting via $F$. By \cite[Proposition 5.1]{KT}, $\Delta^{\loc}$ has an antipode, denoted by $S$, constructed from $F$ and  $S^{\Dr}$.

There is an action of $U_w$ on the space $\Hom(L_{w_1}, M_{w_2})$ as follows. Identify $\Hom(L_{w_1}, M_{w_2})$ with $L_{w_1}^{\vee}\otimes M_{w_2}$ as vector spaces. Let $U_w$ act on $L_{w_1}^{\vee}$ by $(X\cdot f)(a):=f(S(X)\cdot a)$, where $X\in U_w$, $f\in L_{w_1}^{\vee}$ and $a\in L_{w_1}$. This gives an action of $U_w$ on 
$\Hom(L_{w_1}, M_{w_2})$ via the algebra homomorphism 
\[
U_{w}\xrightarrow{\Delta_{w_1, w_2}^{\loc}} (U_{w_1}\otimes U_{w_2})_{\loc}
\xrightarrow{S\otimes 1} (U_{w_1}^{\op}\otimes U_{w_2})_{\loc}. 
\]

Let $V_{\bar{w}}$ be a module of $U_{\bar{w}}$. In the case when $w_1=w_2=\bar{w}$.
We call the action 
\[
\ad: U_{2\bar{w}}\otimes 
\Hom(V(\bar{w}), V(\bar{w})) \to \Hom(V(\bar{w}), V(\bar{w}))
\] the adjoint representation of $U_{2\bar{w}}$.

Pick an element $X$ in $E_{n, G}^*(\mathfrak{P}(v, v+k\alpha_i, 2\bar{w}))$, for some $k\geq 2$, and $i\in I$. 
The coproduct gives us
\begin{align}
E_{n, G}^*(\mathfrak{P}(v, v+k\alpha_i, 2\bar{w}))
&\xrightarrow{\Delta^{\loc}}
\bigoplus_{v_1+v_2=v, \,\ k_1+k_2=k}
(E_{n, G}^*(\mathfrak{P}(v_1, v_1+k_1\alpha_i, \bar{w}))
\otimes 
E_{n, G}^*(\mathfrak{P}(v_2, v_2+k_2\alpha_i, \bar{w})))_{\loc},  \label{eq:coprod}\\
X &\mapsto X_1\otimes X_2 \notag
\end{align}
Let $Y$ be an element in 
$E_{n, G}^*(\mathfrak{P}(v', v'+\alpha_j, \bar{w}))$. 
We define an operation 
\[
E_{n, G}^*(\mathfrak{P}(v, v+k\alpha_i, 2\bar{w}))
\otimes E_{n, G}^*(\mathfrak{P}(v', v'+\alpha_j, \bar{w}))\to 
E_{n, G}^*(\mathfrak{P}(\bar{w}))_{\loc}, \,\
(X, Y)\mapsto \ad(X)(Y)
\]
as follows. 
Here $\sigma$ means considering $SX_1$ as a class on $\mathfrak{P}_{\bar{w}}^{\op}$.
We view 
\[
\ad(X):=(S\sigma\otimes 1)\Delta^{\loc}(X)=S(\sigma X_1)\otimes X_2
\] as an element in $E_{n,G}^*(\mathfrak{P}_{\bar{w}}^{\op}\times \mathfrak{P}_{\bar{w}})_{\loc}$. Recall we have the embedding
\[
\mathfrak{P}_{\bar{w}}^{\op}\times \mathfrak{P}_{\bar{w}} \subset M_2\times M_1\times M_3\times M_4,
\]
where $M_i=\mathfrak{M}(\bar{w})$. 
The element $Y$ in $E_{n, G}^*(\mathfrak{P}(v', v'+\alpha_j, \bar{w}))$. We have the 
 embedding $\mathfrak{P}(v', v'+\alpha_j, \bar{w})\subset M_2\times M_3$. 
Consider the following correspondence
\[
\xymatrix{
M_1\times \Delta(M_2)\times \Delta(M_3)\times M_4 \ar@{^{(}->}[d]^{\Delta} \ar[r]^(0.7){\pi_{14}}& M_1\times M_4
\\
(M_2\times M_1\times M_3\times M_4)\times (M_2\times M_3) 
}
\]
We define
\[
\ad(X)(Y):=(\pi_{14})_*\Delta^*(\ad(X)\boxtimes Y).
\]
\begin{lemma}
\label{lem:expansion}
We have the expansion
\[
\ad(X)Y
=S(X')* Y+Y * X''+l.o.t., 
\]
where $*$ means the convolution product, and $l.o.t.$ consists of the lower order, i.e.,  elements in 
$E_{n, G}^*(\mathfrak{P}(*, *+k'\alpha_i, \bar{w}))_{\loc}$ for $k'<k$.
\end{lemma}
\begin{proof}
To begin with, by \eqref{eq:coprod}, 
we have the expansion of $\ad(X)$ as
\[
\ad(X):=(S\otimes 1)\circ\Delta^{\loc}(X)=(S\otimes 1)(X'\otimes 1+1\otimes X''+ l.o.t.\otimes l.o.t)=
S X'\otimes 1+1\otimes X''+ l.o.t.\otimes l.o.t., 
\]
where $X'$ is an element in 
$E_{n, G}^*(\mathfrak{P}(v_1, v_1+k\alpha_i, \bar{w}))$ and $X''$ is in 
$E_{n, G}^*(\mathfrak{P}(v_2, v_2+k\alpha_i, \bar{w}))$. Now the lemma is reduced to the following equalities, which we claim to be true
\begin{align*}
&(\pi_{14})_*\Delta^*((X'\otimes 1)\boxtimes Y)=X'* Y, \\
&(\pi_{14})_*\Delta^*((1\otimes X'')\boxtimes Y)= Y* X''. 
\end{align*}
Recall that $X'* Y$ is defined as follows. For 
$X'\in E_{n,G}^*(\mathfrak{P}(v_1, v_1+k\alpha_i, \bar{w}))$, and the embedding $\mathfrak{P}(v_1, v_1+k\alpha_i, \bar{w})\subset M_2\times M_1$. 
\[
\xymatrix{
& M_1\times M_2\times M_3  \ar[ld]_{\pi_{12}} \ar[rd]^{\pi_{23}} \ar[d]^{\pi_{13}}& \\
M_1\times M_2 & M_1\times M_3& M_2\times M_3
}
\]
By definition,  $X'* Y:=(\pi_{13})_*(\pi_{12}^*(X') \cdot \pi_{23}^*(Y))$. 
Note that $X'\otimes 1$ is supported on
$M_2\times M_1\times \Delta(M_3) \subset M_2\times M_1\times M_3\times M_4$. 
This implies the claim about $X'$. The claim about $X''$ is similar.  

%\[ \xymatrix{ & M_2\times M_3 \times M_4  \ar[ld]^{\pi_{23}} \ar[rd]^{\pi_{34}} \ar[d]^{\pi_{24}}& \\ M_2\times M_3 & M_2\times M_4& M_3\times M_4 } \] So $Y * X'':=(\pi_{24})_*(\pi_{23}^*(Y) \cdot \pi_{34}^*(X''))$. Note that $1\otimes X''$ is supported on $\Delta(M_2)\times M_3\times M_4 \subset M_2\times M_1\times M_3\times M_4$. This shows the claim. 
\end{proof}

\subsection{The Serre relation}
The following is a version of the Serre relation.

\begin{prop}\label{lem:Serre type relation}
For any elements $X\in U_{2\bar w}$ and $Y\in U_{\bar w}$ with 
$X\in \bigoplus_{v}E_{n, G}^*(\mathfrak{P}(v, v+k\alpha_i, 2\bar{w}))$ being a Chern class of the tautological vector bundle $\calV_{i}^{(k)} \to \fP(v, v+k\alpha_i, 2\bar w)$ 
or $\calV_{i}^{(k), \op} \to \fP(v+k\alpha_i, v, 2\bar w)$ 
and $
Y\in \bigoplus_{v'}E_{n, G}^*(\mathfrak{P}(v', v'+\alpha_j, \bar{w}))$ a Chern class of  $\calL_{j} \to \fP(v', v'+\alpha_j, \bar w)$ or $\calL_{j}^{\op} \to \fP( v'+\alpha_j, v', \bar w)$,  
we have
\[
\ad(X)Y=0,\,\  \text{unless \,\ $k\leq1$.}
\]
\end{prop}
By definition of the adjoint representation, we start with $X$ which is a cohomology class on 
$\mathfrak{P}(v, v+k\alpha_i, 2\bar{w})
\subset \mathfrak{M}(v, 2\bar{w})\times \mathfrak{M}(v+k\alpha_i, 2\bar{w})$, and 
\begin{align*}
\ad(X) &\in \oplus_{v_1+v_2=v,k_1+k_2=k}E_{n, G}^*(
\mathfrak{M}(v_1, \bar{w})
\times \mathfrak{M}(v_2, \bar{w})
\times \mathfrak{M}(v_1+k_1\alpha_i, \bar{w})
\times \mathfrak{M}(v_2+k_2\alpha_i, \bar{w})
)_{\loc}
\\&\cong
\oplus_{v_1+v_2=v,k_1+k_2=k}E_{n, G}^*(
\mathfrak{M}(v_1, \bar{w})
\times \mathfrak{M}(v_1+k_1\alpha_i, \bar{w})
\times \mathfrak{M}(v_2, \bar{w})
\times \mathfrak{M}(v_2+k_2\alpha_i, \bar{w})
)_{\loc}.
\end{align*}
Note that the only non-zero component of $ad(X)(Y)$ has
$v_2=v_1+k_1\alpha_i+\alpha_j$, $v':=v_1+k_1\alpha_i$. Let $v'':=v_2+k_2\alpha_i=v_1+k\alpha_i+\alpha_j$. 
Then, 
\begin{align*}
 Y &\in E_{n, G}^*(
\mathfrak{M}(v', \bar{w})
\times \mathfrak{M}(v'+\alpha_j, \bar{w}))=E_{n,G}^*(
\mathfrak{M}(v_1+k_1\alpha_i, \bar{w})
\times \mathfrak{M}(v_2, \bar{w}))\\
\text{and} \,\  \ad(X)(Y) &\in E_{n, G}^*(\mathfrak{M}(v_1, \bar{w})
 \times \mathfrak{M}(v'', \bar{w})
 )_{\loc}
\end{align*}

\begin{proof}[Proof of Proposition \ref{lem:Serre type relation}]
We rely on an argument similar to Nakajima \cite[Section 11]{Nak}, which has a modification of quiver varieties $$\fM(v,w)\leftarrow \widehat{\fM}(v,w)\hookrightarrow \tilde{\fM}^0(v,w)\hookrightarrow\tilde{\fM}(v,w)$$ where the first map is a principal $\prod_{l\neq i}\GL_{v_l}$-bundle, the second map is a closed embedding determined by the moment map equation, and the third map is an open embedding determined by the stability condition.
We refer the readers to {\it loc. cit.} for the detailed definitions. For our purpose it suffices to mention that $\tilde{\fM}(v,w)$ parametrizes embeddings $\{V^{(i)} \rightarrow W^{(i)}\oplus(\oplus_{
h: i\to j}V^{(j)})\}$.
This modification induces the modification of the Hecke correspondences. 
As we have $v_2=v'+\alpha_j$, and $v''=v_1+k\alpha_i+\alpha_j$, thus, on the vertex $j$, we have $v''^{(j)}=v_1^{(j)}+\alpha_j$. 
In particular, the element $Y$ comes from a cohomology class on the components
\begin{align*}
&
\mathfrak{P}(v', v_2, w)
\leftarrow
\hat{\mathfrak{P}}(v', v_2, w)
\subset
\tilde{\mathfrak{P}}(v', v_2, w)
= \left\{
 \vcenter{\xymatrix@R=0.5em{
V'^{(i)} \ar[r] \ar@{=}[d]& W^{(i)}\oplus(\oplus_{
h: i\to j}V'^{(j)}) \ar@{^{(}->}[d]\\
V_2^{(i)} \ar[r] & W^{(i)}\oplus(\oplus_{h: i\to j}V_2^{(j)})\ar@{->>}[d]\\
& L^{(j)}
}}
\right\}\\
\text{and} \,\ &
\mathfrak{P}(v_1, v', w)
\leftarrow
\hat{\mathfrak{P}}(v_1, v'', w)
\subset
\tilde{\mathfrak{P}}(v_1, v'', w)=
\left\{
 \vcenter{
\xymatrix@R=0.5em{
V_1^{(i)} \ar[r] \ar@{^{(}->}[d]& W^{(i)}\oplus(\oplus_{
h: i\to j}V_1^{(j)}) \ar@{^{(}->}[d]\\
V''^{(i)} \ar[r] \ar@{->>}[d]& W^{(i)}\oplus(\oplus_{h: i\to j}V''^{(j)})\ar@{->>}[d]\\
Q^{(i)}\ar[r]& L^{(j)}
}}
\right\}
\end{align*}
where $Q^{(i)}$ is of rank $k$ at the vertex $i$, and $L^{(j)}$ is of rank $1$ at the vertex $j$. 

The correspondence on which $\ad(X)$ lives is
$\mathfrak{P}(v, v+k\alpha_i, 2\bar{w})\cap \Big(\mathfrak{M}(v_1, \bar{w})
\times \mathfrak{M}(v', \bar{w})
\times \mathfrak{M}(v_2, \bar{w})
\times \mathfrak{M}(v'', \bar{w})\Big)$. The modification $\tilde{\mathfrak{P}}(v, v+k\alpha_i, 2\bar{w})\cap\Big(\tilde{\mathfrak{M}}(v_1, \bar{w})
\times \tilde{\mathfrak{M}}(v', \bar{w})
\times \tilde{\mathfrak{M}}(v_2, \bar{w})
\times \tilde{\mathfrak{M}}(v'', \bar{w})\Big)$ parametrizes the following set of diagrams 
\begin{equation}\label{eqn:correspondence}
\left\{ 
 \vcenter{\xymatrix@R=0.5em{
& V'^{(i)} \ar@{=}[dd] \ar@{^{(}->}[rd]& \\
%%=====
V_1^{(i)} \ar[rr] \ar@{^{(}->}[dd] \ar@{^{(}->}[ru]&& W^{(i)}\oplus(\oplus_{
h: i\to j}V_1^{(j)}) \ar@{^{(}->}[dd]\\
%===
& V_2^{(i)} \ar@{_{(}->}[ld] \ar@{^{(}->}[rd]& \\
%=============
V''^{(i)} \ar[rr] \ar@{->>}[d]&& W^{(i)}\oplus(\oplus_{h: i\to j}V''^{(j)})\ar@{->>}[d]\\
%====
Q^{(i)} \ar[rr]^{B_{ij}}&& L^{(j)}
}}\right\}.
 \end{equation}
Now let $\calI$ be the image of the map $B_{ij}$ and let $K$ be the kernel of $B_{ij}$. We take $K''\subset V''^{(i)}$ to be the preimage of $K$ under the projection $V''^{(i)}\surj Q^{i}$. Then, there is a short exact sequence
\[
0\to K''\to V''^{(i)}\to \calI \to 0. 
\]
Furthermore,  since $V_2^{(i)}=V'^{(i)}$ and it maps to zero in $L^{(j)}$, we have $V_2^{(i)}\subset K''$. The space parametrizing diagrams \eqref{eqn:correspondence} is isomorphic to the space parametrizing  diagrams
 \begin{equation}
 \label{eq:diag2}
 \left\{ 
 \vcenter{
 \xymatrix@R=0.5em{
&& V'^{(i)} \ar@{=}[dd] \ar@{^{(}->}[rd] & \\
%%=====
&V_1^{(i)} \ar[rr] \ar@{^{(}->}[dd] \ar@{^{(}->}[ru]&& W^{(i)}\oplus(\oplus_{
h: i\to j}V_1^{(j)}) \ar@{^{(}->}[dd]\\
%===
V_2^{(i)}\ar[d] \ar@{=}[rr]&& V_2^{(i)} \ar@{_{(}->}[ld] \ar@{^{(}->}[rd]& \\
%=============
K''\ar@{->>}[d]\ar@{^{(}->}[r]&V''^{(i)} \ar[rr] \ar@{->>}[d]&& W^{(i)}\oplus(\oplus_{h: i\to j}V''^{(j)})\ar@{->>}[d]\\
%====
K \ar@{^{(}->}[r]&Q^{(i)} \ar[rr]^{B_{ij}}&& L^{(j)}
}}\right\}
\end{equation}
In particular, the inclusion $V_2^{(i)}\inj V''^{(i)}$ becomes the composition of two inclusions
\[
V_2^{(i)}\inj K'' \inj V''^{(i)}. 
\]
That is, \eqref{eq:diag2} is the fiber product of 
\begin{equation}
\label{eq:diag3}
\left\{
\vcenter{\xymatrix@R=0.5em{
%===line 1====
& V'^{(1)} \ar@{^{(}->}[rd] \ar@{=}[dd]&\\
%===line 2====
V_1^{(i)} \ar@{^{(}->}[ru] \ar@{^{(}->}[dd]\ar@{^{(}->}[rr]&& W^{(i)}\oplus(\oplus_{
h: i\to j}V_1^{(j)}) \ar@{^{(}->}[dd]\\
%===line 3====
& V_2^{(i)} \ar@{^{(}->}[ld]\ar@{^{(}->}[rd]&\\
%===line 4====
K'' \ar@{^{(}->}[rr] \ar[rrd]^{0}& & W^{(i)}\oplus(\oplus_{h: i\to j}V''^{(j)}) \ar@{->>}[d]\\
%====line 5===
&& L^{(j)}
} }\right\}
\end{equation}
and
\begin{equation}\label{eq:diag4}
\left\{
\vcenter{\xymatrix@R=0.5em{
%===line 1====
& V_1^{(i)} \ar@{^{(}->}[rd] \ar@{=}[dd]&\\
%===line 2====
V''^{(i)} \ar@{=}[ru] \ar@{^{(}->}[dd]\ar@{^{(}->}[rr]&& W^{(i)}\oplus(\oplus_{
h: i\to j}V_1^{(j)}) \ar@{^{(}->}[dd]\\
%===line 3====
& K'' \ar@{^{(}->}[ld]\ar@{^{(}->}[rd]&\\
%===line 4====
V''^{(i)}\ar[d] \ar@{^{(}->}[rr] & & W^{(i)}\oplus(\oplus_{h: i\to j}V''^{(j)}) \ar@{->>}[d]\\
%====line 5===
\calI\ar@{^{(}->}[rr]&& L^{(j)}
} }\right\}
\end{equation}
fibered over 
\[
\left\{
\vcenter{
\xymatrix
{
V_1^{(i)} \ar@{^{(}->}[r] \ar@{^{(}->}[d]& W^{(i)}\oplus(\oplus_{
h: i\to j}V_1^{(j)}) \ar@{^{(}->}[d]\\
K''\ar@{^{(}->}[r] &
W^{(i)}\oplus(\oplus_{h: i\to j}V''^{(j)})
}}\right\}
\]
As $X$ is a tautological class, hence can be expressed in terms of Chern classes of the tautological quotient on $\mathfrak{P}(v, v+k\alpha_i, 2\bar{w})$ and its restriction to  $\mathfrak{P}(v, v+k\alpha_i, 2\bar{w})\cap \Big(\mathfrak{M}(v_1, \bar{w})
\times \mathfrak{M}(v', \bar{w})
\times \mathfrak{M}(v_2, \bar{w})
\times \mathfrak{M}(v'', \bar{w})\Big)$ can be expressed in terms of Chern classes of $\coker(V_1^{(i)}\to V'^{(i)})$ and $\coker(V_2^{(i)}\to V''^{(i)})$, the latter in turn can be expressed in terms of those of $\calI$ and those of $\coker(V_2^{(i)}\to K'')$.
In other words, one computes $\ad(X)(Y)$ as follows. Let $\Delta^{\loc}(X)=\sum X_1\otimes X_2$, and $\Delta^{\loc}(X_2)=\sum X_2'\otimes X_2''$, with $X_1$, $X_2'$ and $X_2''$ coming from $\coker(V_1^{(i)}\to V'^{(i)})$, $\calI$ and  $\coker(V_2^{(i)}\to K'')$ respectively. 
Then, there exists an tautological class $H$ depending on $X_2$ only such that $X_2=(\sum X_2' * X_2'')*H$. (An explicit formula of $H$ can be found but not needed for the purpose of this paper.) Therefore, 
\[
\ad(X)(Y)= \sum S(X_{1})* Y * X_2=\sum S(X_{1})* Y * X_{2}'* X_{2}''*H.
\]

We claim that $X_2'$ commutes with $Y$.  Roughly speaking,  $K''$ maps to $0$ in $L^{(j)}$ and \eqref{eq:diag4} is so that as if vertices $i$ and $j$ are not adjacent. If $i$ and $j$ are not adjacent, the operators commute and the adjoint action is zero thanks to the antipode relation. Indeed, \eqref{eq:diag4} is isomorphic to the space of diagrams
\[
\left\{
\vcenter{\xymatrix@R=0.5em{
%===line 1====
& K''\ar@{^{(}->}[rd] \ar@{=}[dd]&\\
%===line 2====
V_1^{(i)} \subseteq V'^{(i)} \ar@{^{(}->}[ru] \ar@{^{(}->}[dd]\ar@{^{(}->}[rr]&& W^{(i)}\oplus(\oplus_{
h: i\to j}V_1^{(j)}) \ar@{^{(}->}[dd]\\
%===line 3====
& K'' \ar@{=}[ld]\ar@{^{(}->}[rd]&\\
%===line 4====
K'' \ar@{^{(}->}[rr] \ar[rrd]^{0}& & W^{(i)}\oplus(\oplus_{h: i\to j}V''^{(j)}) \ar@{->>}[d]\\
%====line 5===
&& L^{(j)}
} }\right\}
\]
This is the convolution of \[
\left\{
\vcenter{
\xymatrix@R=1em
{
K'' \ar@{^{(}->}[r] \ar@{=}[d]& W^{(i)}\oplus(\oplus_{
h: i\to j}V_1^{(j)}) \ar@{^{(}->}[d]\\
K''\ar@{^{(}->}[r] &
W^{(i)}\oplus(\oplus_{h: i\to j}V''^{(j)})
}}\right\}
\]
 with space of flags $\{V_1^{i}\inj V'^{i}\inj K''\}$.

Let
$\Delta^{\loc}(X_1)=\sum X_1'\otimes X_1''$.
By co-associativity, we have
\[
\sum X_1'\otimes X_1''\otimes X_2=
\sum X_1\otimes X_2'\otimes X_2''. 
\]
As $S$ is the antipode, we have $\sum S(X_1')X_1''=\epsilon(X_1)=\sum X_1' S(X_1'')$. 
Therefore, we have
\begin{align*}
\ad(X)(Y)
=& \sum S(X_{1})* Y * X_{2}'* X_{2}''*H \\
=& \sum S(X_{1}) * X_{2}' * Y * X_{2}''*H\\
=& \sum S(X_{1}') * X_{1}'' * Y * X_{2}*H\\
=& \sum \epsilon(X_1) * Y* X_2*H. 
\end{align*}
As $\dim(\calI)\leq \dim L^{(j)}=1$,  we have $Y* X_2=0$ unless $X_2$ increases dimension vector by no more than $e_i$. If $k\geq2$, then $X_1$ must increase dimension vector by $(k-1)e_i\geq e_i$. 
In this case, $\epsilon(X_1)=0$. Therefore, 
$\ad(X)(Y)=0$ for $k\geq 2$. 
\end{proof}

Pick an arbitrary element $X''$ in $E_{n, G}^*(Z(v_2, v_2+k\alpha_i, \bar{w}))$ with $k\geq 2$. There exists $X\in E_{n, G}^*(Z(v, v+k\alpha_i, 2\bar{w}))$, such that 
\[
\Delta^{\loc}(X)=X'\otimes 1+1\otimes X''+l.o.t. 
\]
Let $Y$ be a tautological class in 
$E_{n, G}^*(Z(v', v'+\alpha_j, \bar{w}))$, i.e., a generator of $U'^n_{\bar{w}}$. 
By Lemma \ref{lem:expansion} and Proposition \ref{lem:Serre type relation}, we have
\begin{equation}\label{eq:ad(X)(Y)}
0=\ad(X)Y 
=S(X')* Y +Y* X'' +l.o.t.
\end{equation}

\section{The Steinberg tensor product theorem}
\label{sec:10}

In this section, we prove two results about irreducible representations of the new quantum groups. The first result (Theorem \ref{thm:reduced region}) is a comparison of the irreducible modules of $U^n_w$ and the quantum loop algebra when the highest weights lie in the reduced region. The second result (Theorem \ref{thm:Steinberg tensor}) is the Steinberg tensor product formula. 

%====== Section about 
\subsection{Reduced highest weights}
In this section  we assume $\mathbb{F}$ is an algebraically closed field endowed with a ring homomorphism  $E_{n,\Lambda_n}\to L(E_n)\to \mathbb{F}$. 
By an abuse of notation, in this section, the algebra $U^1_w$ and its representations are constructed from $K$-theory with coefficients in the field $\mathbb{F}$. The Morava theory $E_1$ is related to $K$-theory by a completion. That is,  we write $E$ for $E_n$, then $E_1\to E_{K_1}$ induces the ring homomorphism $L(E_1)\to L(E_{K_1})$
as in Remark~\ref{rmk:coef}(1). 
Therefore, passing to an extension if necessary, we may assume there is a ring homomorphism $E_{1,\bbZ/p}\to L(E_1)\to \mathbb{F}$. 
Let
\[
X_{\red}^+:=\{\lambda\in X\mid 0\leq \langle \lambda, \alpha_i^\vee \rangle <p, \text{for any $i\in I$}\}. 
\]\begin{theorem}
\label{thm:reduced region}
Assume $w\in X^+_{\red}$. Let $\alpha$ be the group homomorphism satisfying the property $\mathrm{P}_{w}^n$
as in \S\ref{subsec:group_hom}. 
Let $L_{w}^{(n)}(\alpha, \mathbb{F})$
be the simple representation of $U^n_w$. 
Then, we have an isomorphism of vector spaces
\[
L_{w}^{(n)}(\alpha, \mathbb{F})
\cong 
L_{w}^{(1)}(\alpha_1, \mathbb{F}). 
\]
\end{theorem}Recall we have the group homomorphism 
\[
\alpha:
\Lambda_n\to \GL_w\times \C^*.
\] 
Write $\alpha$ as $\alpha=(\alpha_1, \cdots, \alpha_n)$, where $\alpha_i: \Z/p\to \GL_w\times \C^*, 1\mapsto (g_i, \zeta)$.  
\begin{proof}
The assumption of $w\in X^+_{\red}$ implies each eigenspace of $g_1$ is one-dimensional. By the choices of $\alpha$ in \S\ref{subsec:group_hom}, we have  $Z^{\im \alpha }=Z^{\im \alpha_1 }$ by Corollary~\ref{cor:8.3}(1). 

For the above choice of $\alpha$, we have the action of $\im\alpha$ on $Z$. 
By \cite[Theorem C]{HKR}, we have the isomorphism
\begin{align*}
L(E_n)\otimes_{E_n} E^*_{n, \im(\alpha)}(Z)
\cong &
L(E_n)\otimes_{E_n} E^*_{n}(E\im \alpha\times_{\im \alpha} (\sqcup_{\gamma\in \Hom(\Lambda_n, \im \alpha)}Z^{\im \gamma}))
 \\
 \cong &
\oplus_{\gamma\in \Hom(\Lambda_n, \im(\alpha))}L(E_n)\otimes_{E_n}
E^*_{n}(E\im\alpha\times_{\im\alpha}  Z^{\im  \gamma })
\end{align*}
It decomposes the ring $L(E_n)\otimes_{E_n} E^*_{n, \im(\alpha)}(Z)$ into direct sums. In particular, each summand is given by an idempotent of $L(E_n)\otimes_{E_n}
E^*_{n}(\pt)$. Let $1_{\alpha}$ be the 
idempotent coming from the identity map $1\in \Hom(\Lambda_n, \im \alpha)$. 
We pick the direct summand $L(E_n)\otimes_{E_n}
E^*_{n}(B\im\alpha\times  Z^{\im  \alpha})$ by localizing at $1_{\alpha}$: 
\begin{align*}
&
(L(E_n)\otimes_{E_n} E^*_{n, \im(\alpha)}(Z))_{1_\alpha}\cong 
L(E_n)\otimes_{E_n}
E^*_{n}(B\im\alpha\times  Z^{\im  \alpha}).
\end{align*}

Let $i_{\alpha}: Z^{\im \alpha} \inj Z$ be the embedding of the fixed points. 
The above isomorphism is induced by pullback $i^*_{\alpha}$. 
Note that both sides have algebra structures by convolution. Similar to the proof of Theorem \ref{thm:algebra hom}, we modify the above isomorphism by  $\frac{1}{e^n(N_2)}$, where $N_2$ is the normal bundle of the embedding 
$\mathfrak{M}(w)^{\im \alpha} \inj \mathfrak{M}(w)$ of the second factor, and $e^n$ is the Euler class in $E_n$-theory.
Therefore, for the field $\mathbb{F}$ containing $L(E_n)$ and $L(E_1)$, for $w\in X^+_{\red}$,  we have the following diagram
\begin{equation}
\label{eq:iso for red w}
\xymatrix{
(\mathbb{F}\otimes_{E_n} E^*_{n,\im \alpha}(Z))_{1_{\alpha}}
\ar[d]^{ \frac{i_{\alpha}}{e^n(N_2)}}_{\cong}
\ar[r]
&
(\mathbb{F}\otimes_{E_1} E^*_{1,\im \alpha_1}(Z))_{1_{\alpha_1}} \ar[d]^{ \frac{i_{\alpha_1}}{e^1(N_2)}}_{\cong}\\
\mathbb{F}\otimes_{E_n}
E^*_{n}(B\im\alpha\times  Z^{\im  \alpha})\ar@{=}[r]&
\mathbb{F}\otimes_{E_1}
E^*_{1}(B\im\alpha_1\times  Z^{\im  \alpha_1})
}
\end{equation}
Here $e^1$ is the $K$-theory Euler class. 
The top horizontal arrow is defined to be the composition of the three other isomorphisms in the diagram. It is an algebra homomorphism since $e^n(N_2)$ and $e^1(N_2)$ agree in $L^*_{n}(B\im\alpha\times  Z^{\im  \alpha})$. 

This implies the isomorphism of the  representations
$L_{w}^{(n)}(\alpha, \mathbb{F})
\cong 
L_{w}^{(1)}(\alpha_1, \mathbb{F})$.
This finishes the proof. 
\end{proof}
Recall that the subalgebra $U'^n_{w}$ of $U^n_{w}$ is generated by the tautological classes in 
\[
E_{n, \varmathbb{G}}^*(\mathfrak{P}(v, v+1, w)) \,\  \text{and} \,\ E_{n, \varmathbb{G}}^*(\mathfrak{M}(v+1, v,  w)). 
\]
\begin{corollary}\label{cor:spherical}
Assume $w\in X^+_{\red}$. 
The isomorphism $(U^n_w)_{1_{\alpha}}\cong (U^1_w)_{1_{\alpha_1}}$ from \eqref{eq:iso for red w}  induces an isomorphism of the subalgebra $(U'^n_w)_{1_{\alpha}}\cong (U'^1_w)_{1_{\alpha_1}}$. 
\end{corollary}
\begin{proof}
Note that the normal bundle of the embeddings above can be expressed in terms of tautological bundles.  It is straightforward to check the isomorphism maps the subalgebra $(U'^n_{w})_{1_{\alpha}}$ to $(U'^n_{w})_{1_{\alpha_1}}$. 
\end{proof}

In the quantum loop case, $U'^1_w(G,\bbF)$ is the image of the De Concini-Kac form in $K_{G}(Z(w))_{\bbF}$. 
By \cite[Theorem 14.3.2]{Nak}, we have $L_{w}^{(1)}(\alpha, \mathbb{F})$
is an irreducible module over the quantum loop algebra and an irreducible module of the convolution algebra simultaneously.

Denote by $\Res L_{w}^{(n)}(\alpha, \mathbb{F})$
the restriction of the irreducible module $L_{w}^{(n)}(\alpha, \mathbb{F})$ to $U'^n_w$. 
\begin{corollary}\label{cor:irreducible}
Let $w\in X^+_{\red}$. Then, 
$\Res L_{w}^{(n)}(\alpha, \mathbb{F})$
is an irreducible module over  $U'^n_w$. 
 \end{corollary}
\begin{proof}
Let $(P_i)_{i\in I}$ be the Drinfeld polynomial determined by $\alpha$. Hence 
$\deg(P_i)=w_i$ for each $i\in I$. 
As $w\in X^+_{\red}$, each Drinfeld polynomial $P_i$ is not divisible by $1-au^p$ for any non-zero $a\in \C$ for degree reasons.  Hence, by \cite[Theorem 9.2]{CP} $\Res L_{w}^{(1)}(\alpha, \mathbb{F})$ is irreducible over $U'^1_w$. The assertion now follows from Theorem \ref{thm:reduced region} and Corollary~\ref{cor:spherical}. 
\end{proof}

%===============Section about Steinberg tensor======

\subsection{The tensor product of modules with dynamical parameters}
\label{sec:tensor product with dynamical} The coproduct from \S~\ref{subsec:loc_coprod} is a localized coproduct, which is enough for the purpose of establishing the Serre relations. 
In this section, as a preparation for the Steinberg tensor product formula, we consider the tensor product of irreducible modules of the convolution algebra using the coproduct constructed in Appendix \ref{app:coproduct}.

Let $w=w'+pw''$ and choose 
\[
A=
\left[
\begin{array}{c|c}
\id_{w'} & 0 \\
\hline
0 & t \id_{pw''}
\end{array}
\right]
\] to be the torus where $\id_{w}$ is the identity matrix of size $w$. We then have the embedding of $A$-fixed points \[
\mathfrak{M}(w)^{A}=\mathfrak{M}(w') \times \mathfrak{M}(pw'') \subset \mathfrak{M}(w).\] 
% We have maps induced by pullback
% \begin{align*}
% & \Spec(E^*_G(\mathfrak{M}(w'))) \times \Spec(E^*_G(\mathfrak{M}(pw'')))
% \to \Spec(E^*_G(\mathfrak{M}(w))), \\
% & \Pic(\Spec(E^*_G(\mathfrak{M}(w))))\to \Pic(\Spec(E^*_G(\mathfrak{M}(w'))) \times \Spec(E^*_G(\mathfrak{M}(pw'')))). \end{align*}

Fix a virtual bundle $T^{1/2}\mathfrak{M}(w) \in K_{T}(\mathfrak{M}(w))$, such that, 
\[
T\mathfrak{M}(w)=T^{1/2}\mathfrak{M}(w)+\hbar^{-1} \otimes (T^{1/2}\mathfrak{M}(w))^{\vee}, 
\]
referred to as the \textit{polarization} of $\mathfrak{M}(w)$.
Let $\Theta^{E^*}: K_G(\mathfrak{M}(w))\to \Pic( \Spec E^*_G(\mathfrak{w}))$ be the group homomorphism  determined by the Euler class of the $G$-equivariant vector bundle on $\mathfrak{M}(w)$. ( See Appendix \ref{app:coproduct} and \cite[Proposition 3.12]{ZZ} for details). 

Let $\rho$ be the Poincare line bundle on 
\[
\Spec( E_G^*(\mathfrak{M}(w)))\times \mathbb{B},  
\]
as in Definition \ref{def:Poincare}, where $\mathbb{B}=\Spec E^*_{\C^*}(\pt)\otimes_{\bbZ} \bbZ^{I}=\Spec E^*_{\C^*}(\pt)\otimes_{\bbZ}
 \Hom(\GL_v, \C^*)$. 

 Note that, $\Spec E^*_{\C^*}(\pt)=F$ is the formal group of $E^*$. Thus, there is a group structure on $\mathbb{B}$. We denote this group structure by $\otimes$. 
\begin{definition}
\begin{enumerate}
\item
Define the Weyl module with dynamical parameter to be
\[
V_{w, \rho}:=\Theta^{E^*}(T^{1/2}\mathfrak{M}(w))\otimes \rho, 
\] which is a line bundle on $\Spec(E^*_G(\mathfrak{M}(w)))\times \bbB$. 
Let $\calU_w$ be the convolution algebra with dynamical parameters defined in \S\ref{sec:convolution algebra}. 
Then, $V_{w, \rho}$ carries a representation of the convolution algebra $\calU_w$ by \eqref{equ:action Weyl}.  
\item
Let $L_{w, \rho}$ be the irreducible representation of $\calU_w$ that is obtain by taking the quotient of $V_{w, \rho}$. 
\end{enumerate}
\end{definition}
For the purpose of the present paper, we fix a $\bbF$-point in $\bbB$ so that representations are in the category of $\bbF$-vector spaces. In particular, the character of $L_{w, \rho}$ is the same as that of $L^{(n)}_w(G, \bbF)$.

Given two Weyl modules $V_{w', \rho_1}$, $V_{pw'', \rho_2}$, which are line bundles on 
\[
\Spec(E^*_G(\mathfrak{M}(w')))\times E^*(\pt)\otimes_{\bbZ} \bbZ^{I} \,\ 
\text{and}\,\ 
\Spec(E^*_G(\mathfrak{M}(pw'')))\times E^*(\pt)\otimes_{\bbZ} \bbZ^{I}
\]
respectively. Define the tensor product of $V_{w', \rho_1}$ and $V_{pw'', \rho_2}$ to be
\[
\Big(\Theta^{E^*} (T^{1/2} \mathfrak{M}(w'))\boxtimes  \Theta^{E^*} (T^{1/2} \mathfrak{M}(pw''))\Big)\otimes (\rho_1\boxtimes \rho_2)\otimes \rho^{-1} \otimes \delta, 
\]
which is a line bundle on 
$\Spec(E^*_G(\mathfrak{M}(w')))\times
\Spec(E^*_G(\mathfrak{M}(pw'')))\times \bbB
$. Here $\delta$ is the dynamical shift defined in \eqref{eq:shift d}. 
It carries a module structure of $\calU_w$ via the coproduct $\Delta^{E^*}$ \eqref{DeltaE}. 

Similarly, define the tensor product of irreducible modules
 \[
L_{w', \rho_1}\otimes_{\rho}  
L_{pw'', \rho_2}:= (L_{w', \rho_1}\boxtimes 
L_{pw'', \rho_2})
\otimes \rho^{-1} \otimes \delta. 
\] 
Then, $\calU_w$ acts on the tensor $L_{w', \rho_1}\otimes_{\rho}  
L_{pw'', \rho_2}$ via the coproduct $\Delta^{E^*}$,  
as  $L_{w', \rho_1}\otimes_{\rho}  
L_{pw'', \rho_2}$ is a quotient of $V_{w', \rho_1}\otimes_{\rho}  
V_{pw'', \rho_2}$ through the composition
\[
V_{w, \rho}
\to V_{w', \rho_1}\otimes_{\rho}  
V_{pw'', \rho_2}
\surj
L_{w', \rho_1}\otimes_{\rho}  
L_{pw'', \rho_2}.
\]
We obtain an $\calU_w$-module homomorphism
\[
L_{w', \rho_1}\otimes_{\rho}  
L_{pw'', \rho_2}
\to L_{w, \rho}. 
\]
When we specialize the dynamical parameters to be
trivial, the convolution algebra $\calU_w$ with parameter and the Weyl module $V_{w, \rho}$ with dynamical parameter become the usual convolution algebra $U_w$ and the Weyl module $V_w(G):=E^*_{n, G}(\mathfrak{M}(w))$. 
We remark that the dynamical shifting naturally shows up when considering tensor product of modules, the reason for including it into the present paper. 

Moreover, assume $p^t|w_1$ and $p^t|w_2$, and that $b_i\in G$ for $i=1,\dots,t$, then $\Delta_{w_1,w_2}^{E^*}$ commutes with the Frobenii $\Fr_{n,n-t}$. 
The condition that $G$ contains $b_i$ is equivalent to asking $G$ to contain the subgroup $\gamma \wr \langle b\rangle$ which is used in the construction of the Frobenii. 
The statement that $\Delta_{w_1,w_2}^{E^*}$ commutes with the Frobenii follows from the definition of $\Delta^{E^*}$ and the observation that the subvariety $Z(w_1)\times Z(w_2)\subseteq Z(w)$ is realized as the fixed locus of an element $g\in Z(G)$ and hence taking fixed points under $g$ and $\alpha$ commutes. 

\subsection{The Steinberg tensor product theorem}
\label{subsec:Steinberg tensor}
Let $w=w'+pw''$, where $w'\in X^+_{\red}$ and $w''\in X^+$. Let $\gamma:\Lambda_n\to \GL_{w}\times\bbC^*$ be as in \S~\ref{subsec:alpha} so that we have $\gamma=\gamma'\times\gamma''$ with $\gamma':\Lambda_n\to\GL_{w'}\times\bbC^*$ and $\gamma'':\Lambda_n\to\GL_{pw''}\times\bbC^*$.

We have a chain of ring homomorphisms $E_n\to E_{n,(\bbZ/p)}\to C^\wedge_{n-1}\to L(C^\wedge_{n-1})\to \Phi_1^{L(C^\wedge_{n-1})}$ as in \S~\ref{subsec:coefficient}. As in Remark~\ref{rmk:coef}(1), we also have $E_{n-1}\to  C^\wedge_{n-1}$. In this section we assume that the field $\mathbb{F}$ is endowed with a ring homomorphism $\Phi_1^{L(C^\wedge_{n-1})}\to \mathbb{F}$.
\begin{theorem}
\label{thm:Steinberg tensor}
Notations as above, there is an isomorphism of representations of  $U^n_w$
\[
L_w^{(n)}(\gamma, \mathbb{F})
\cong L_{w'}^{(n)}(\gamma', \mathbb{F})\otimes_\rho
 L_{pw''}^{(n)}(\gamma'', \mathbb{F}). 
\]
\end{theorem}

% \begin{prop}\label{Serre rel}
% We have the following relations of $U^n_w$
% \begin{align}
% &E_{i, r} E_{j, k}^{[N]} +E_{j, k}^{[N]} E_{i, r}  =\text{(l. o. t of $E_{j, k}$)} \label{rel:Serre}. 
% \\
% &E_{i, r} F_{j, k}^{[N]} 
% =F_{j, k}^{[N]} E_{i, r}+\text{(l. o. t of $F_{j, k}$)}. 
% \label{rel:E and F}
% \end{align}
% \end{prop}
% \begin{proof}
% The relation \eqref{rel:Serre} is a type of the Serre relation. It uses the adjoint representation of $\mathfrak{sl}_2$, which is an integral representation of $\mathfrak{sl}_2$. 
% \yaping{Need to define the adjoint representations, it is something like
% \[
% (\ad X_i) X_j=X_iX_j-q^{a_{ij}/2} X_j X_i. 
% \].}
% It gives the form
% \[
% \ad(X_{i, r}^{\pm})^m X_{j, s}^{\pm}=0, \,\ m=1-a_{ij}. 
% \]
% By expending the identity, two of the terms are 
% $X_{i, r} X_{j, k}^{[N]} +X_{j, k}^{[N]} X_{i, r} $ and the rest of the terms gives the linear combination of the 
% lower order terms of $X_{j, k}$. 

% For the relation \eqref{rel:E and F}:
% We have the relations that 
% \begin{align*}
% &E_{i, r} F_{j, k} 
% =F_{j, k}E_{i, r}+\delta_{ij}h_{i, r+k},\\
% &H X_{i, r}=X_{i, r} H', 
% \end{align*}
% \yaping{Do I need to explain where those two relations come from?}
% where $h_{i, r+k}, H, H'$ are elements in the Cartan subalgebra. The choice of $H'$ depends on $H$ and $X_{i, r}$. The relation \eqref{rel:E and F} can be easily deduced by induction. 
% \end{proof}

We notice that the trivial representation $\triv_w$ of any $U^n_w$ for $w\in \bbN^I$ can be constructed geometrically. Indeed, following Nakajima \cite[Proposition~13.3]{Nak} we have a standard module $M_{\gamma,x}$ depending on $x\in \fM_0(w)$, where $\fM_0(w)$ has a stratification with strata labelled by $v_0$ which determines the highest weight of $M_{\gamma,x}$ as $w-v_0$. The trivial representation is obtained at $v_0=w$. 

Start with the representation $L_w:=L_{w}^{(n)}(\alpha, \mathbb{F})$ of $U^n_w$. 
It is a module over the subalgebra $U'^n_w$ by restriction. 
By Corollary \ref{cor:irreducible}, $L_{w'}:=L_{w'}^{(n)}(\gamma', \mathbb{F})$ is irreducible over  $U'^n_{w'}$. 
Tensoring with the trivial representation $\triv_{pw''}$, we have $L_{w'}\cong L_{w'}\otimes_\rho \triv_{pw''}$, which is irreducible over $U'^n_w$. This gives an $U'^n_w$-injective homomorphism by irreducibility
\[
L_{w'}\otimes_\bbF V
\inj L_w, \]
where $V:=\Hom_{U'^n_w}(L_{w'}, L_w)$ is the multiplicity space of $L_{w'}$ in $L_w$.
\begin{lemma}
\begin{enumerate}
\item We have the identification 
\[
V\cong \Hom_{\mathbb{F}} (L_{w'}, L_w
)^{U'^n_{w'+w}}. 
\]
\item The space $V$ carries an action of $U^n_{w'+w}$. 
\end{enumerate}
\end{lemma}
\begin{proof}
We first prove (1). 
It is clear that 
\[
\Hom_{U'^n_{w'+w}} (\triv, L_{w'}^{\vee}\otimes L_w)
\cong
\Hom_{U'^n_{w'+w}} (\triv, \Hom_{\mathbb{F}}(L_{w'}, L_w)
) \cong
\Hom_{\mathbb{F}} (L_{w'}, L_w
)^{U'^n_{w'+w}}.
\]
Here the module structure on the dual comes from the contragredient duality \cite[(3.6.10)]{CG}. 
For any $U'^n_{w'+w}$-morphism 
$f: \triv\to L_{w'}^{\vee}\otimes_\rho L_w$, 
the composition of  $U_{2w'+w}$-morphisms
\[
L_{w'}\otimes_\rho \triv \xrightarrow{\id\otimes f} L_{w'}\otimes_\rho L_{w'}^{\vee}\otimes_\rho L_w
\to \triv\otimes_\rho L_w
\]
gives an element in $\Hom_{U'^n_{2w'+w}}(L_{w'}, L_w)$. 
This proves an isomorphism 
\[
\Hom_{U'^n_{2w'+w}}(L_{w'}, L_w)\cong 
\Hom_{\mathbb{F}}(L_{w'}, L_w)^{U'^n_{2w'+w}}. 
\]
Now tensoring with the trivial representation, we have
\[
\Hom_{U'^n_{w}}(L_{w'}, L_w)\cong 
\Hom_{U'^n_{2w'+w}}(L_{w'}\otimes_\rho \triv, L_w \otimes_\rho \triv)\cong 
\Hom_{U'^n_{2w'+w}}(L_{w'}, L_w). 
\]
This completes the proof of (1).

For (2), for simplicity of the notations, we write $\bar{w}=w'+w$. 
Note that $U^n_{\bar{w}}$ is generated by $E_{n, \varmathbb{G}}^*(Z(v, v+k\alpha_i, \bar{w}))$ for various $k\in \N$ and $i\in I$. Pick an arbitrary $X''$ in $E_{n, G}^*(Z(v_2, v_2+k\alpha_i, \bar{w}))$ with $k\geq 2$.  
Let $Y$ be a tautological class in 
$E_{n, G}^*(Z(v', v'+\alpha_j, \bar{w}))$, i.e., a generator of $U'^n_{\bar{w}}$. By \eqref{eq:ad(X)(Y)}, for any $f\in V$, we have
\[
0=\ad(X)Y (f)=S(X')* Y(f) +Y* X''(f)+l.o.t.(f) , 
\]
where the l.o.t. means the linear combination of product of elements in $E_{n, G}^*(Z(v, v+k'\alpha_i, w))$, where $k'<k$. 
As $Yf=0$ for $Y\in U'^n_w$ and $l.o.t(f)=0$ by induction on $k$, we conclude that $Y(X''f)=0$. This shows $X''f\in V$. As can be seen from the proof of Lemma~\ref{lem:expansion}, $X''$ and $X$ determine each other. The choice of $X$ is arbitrary, and hence this induces an action of $U_{\bar{w}}$ on $V$. 
\end{proof}
\begin{proof}[Proof of Theorem \ref{thm:Steinberg tensor}]
Consider the tensor product $L':=L_{w'}\otimes_\rho L_{pw''}:=L_{w'}^{(n)}(\gamma', \mathbb{F})\otimes_\rho L_{pw''}^{(n)}(\gamma'', \mathbb{F})$ and the $U'^n_{w}$-homomorphism
$L_{w'}\otimes_\bbF V'
\to L'$, 
where $V':=
\Hom_{U'^n_{w}} (L_{w'}, L'
)$ is the multiplicity space for $L'$.
We claim that $V'=L_{pw''}$. 
Indeed, $U_w$ acts on the tensor $L'=L_{w'}\otimes_\rho L_{pw''}$ through the algebra homomorphism
\[
U_w\xrightarrow{\Delta^{E^*}_{w', pw''}} U_{w'}\otimes U_{pw''}
\]
The module $L_{pw''}$ is isomorphic to the Frobenius pullback $\Fr_{n, n-1}^* (L_{w''}^{n-1})$, where 
$\Fr_{n, n-1}:  U^n_{pw''}\to U^{n-1}_{w''}$ is the Frobenius morphism. 
Restricting to the subalgebras, $\Fr_{n, n-1}$ kills the generators of $U'_{pw''}$, therefore $U'_{w}$ acts trivially on $L_{pw''}$. This shows $L_{pw''}$ is the multiplicity space of $L_{w'}$ in $L'$. 

Note that the natural map $V'\to V$ is an $U^n_{w'+w}$-module homomorphism. As $V'$ is irreducible over $U^n_{pw''}$ and hence irreducible over $U^n_{w'+w}$, we obtain that $V'\inj V$ is injective. Now we have the following commutative diagram of $U'_{2w'+w}$-modules homomorphisms. 
\[
\xymatrix{
L_{w'}\otimes V'\ar[r]^(0.7){\cong}  \ar@{^{(}->}[d] &L' \ar@{->>}[d]\\
L_{w'}\otimes V \ar@{^{(}->}[r] &L_w
}
\]
It implies that $L'\cong L_w$. 
This completes the proof. 
\end{proof}

\section{Characters of irreducible modules}

We retain the assumptions on  $\mathbb{F}$ and $\gamma$ as in \S~\ref{sec:10}. In this section we study the character of the irreducible representations $L^n_w(\gamma,\mathbb{F})$ associated to $\mathbb{F}$ and $\gamma$.

\subsection{The $\epsilon, t$-character}
Following Nakajima \cite{Nak04}, we define the $\epsilon, t$-character of certain modules of $U^n_w(\gamma,\mathbb{F})$ as follows. 
Let $V, W$ be two $I\times (\Z/p\Z)^n$-graded vector spaces. Denote by $V_{i, a}$ and $W_{i, a}$ the $(i, a)$-component respectively, where $i\in I, a\in (\Z/p\Z)^n$. 
Recall the Weyl module $V^n_{w}(\gamma, \mathbb{F})$  \eqref{module:V} is
\begin{align*}
V_w^n(\gamma, \mathbb{F})
=& \oplus_v\mathbb{F}\otimes_{E_{n,\Lambda_n}} E^*_{n, \varmathbb{G}} ( \mathfrak{M}(v, w))\\
\cong & 
\oplus_vH^{BM}_*(\mathfrak{M}(v, w)^{\im \gamma};\bbF)
\end{align*}
where the second isomorphism follows from \cite[Theorem C]{HKR}. 
The simple module $L^n_w(\gamma, F)$ as a quotient of $V_w^n(\gamma, \mathbb{F})$ also has the decomposition $L^n_w(\gamma, F)=\oplus_{k, V, W} L^n_w(\gamma, F)_{k, V, W}$, where $L^n_w(\gamma, F)_{k, V, W}$ is the quotient of $\mathbb{F}\otimes H_k(
\mathfrak{M}^{\im \gamma}(V, W))$.
Set 
\[
e^V:=\prod_{i, a} V_{i, a}^{\dim V_{i, a}} \text{and} \,\ e^W:=\prod_{i, a} W_{i, a}^{\dim W_{i, a}}\in \Z[V_{i, a}, W_{i, a}]_{i\in I, a\in (\Z/p\Z)^n}. 
\]
Define the $\epsilon, t$-character of $V_w^n(\gamma, \mathbb{F})$ and $L_w^n(\gamma, \mathbb{F})$
to be
\begin{align*}
&\ch_{\epsilon, t}(V^n_{w}(\gamma, \mathbb{F})):=\sum_{[V]} (-t)^k \dim_{\mathbb{F}} H^{BM}_k(
\mathfrak{M}(v, w)^{\im \gamma};\mathbb{F})e^{V} e^W 
\in \Z[t, t^{-1}, V_{i, a}, W_{i, a}]_{i\in I, a\in (\Z/p\Z)^n}. \\
&\ch_{\epsilon, t}(L^n_{w}(\gamma, \mathbb{F})):=\sum_{[V]} (-t)^k \dim_{\mathbb{F}}L^n_w(\gamma, F)_{k, V, W} e^{V} e^W 
\in \Z[t, t^{-1}, V_{i, a}, W_{i, a}]_{i\in I, a\in (\Z/p\Z)^n}.
\end{align*}
Similarly, we
define the \textit{character} of $V_w^n(\gamma, \mathbb{F})$
and $L_w^n(\gamma, \mathbb{F})$ to be
\begin{align*}
&\ch(V_w^n(\gamma, \mathbb{F})):=\sum_{v\in \N^I} \dim(V_w^n(\gamma, \mathbb{F})_v) e^{\mu(v, w)}\in \Z[X]\\
&
\ch(L_w^n(\gamma, \mathbb{F})):=\sum_{v\in \N^I} \dim(L_w^n(\gamma, \mathbb{F})_v) e^{\mu(v, w)}\in \Z[X]. 
\end{align*}

Consider the ring homomorhism \cite[Section 2]{Nak04}
\begin{align*}
\Pi: \Z[t, t^{-1}, V_{i, a}, W_{i, a}]_{i\in I, a\in (\Z/p\Z)^n}
\to \Z[y_{i}, y_i^{-1}], 
t\mapsto 1, \prod_{i, a} V_{i, a}^{v_{i, a}} W_{i, a}^{w_{i, a}}
\mapsto \prod_{i, a} y_i^{u_{i, a}}, 
\end{align*}
where $u^{i, a}:=w_{i, a}-v_{i, a\epsilon^{-1}}-v_{i, a\epsilon}+\sum_{j: C_{ji}=-1} v_{j, a}$. 
We then have
\[
\Pi \circ \ch_{\epsilon, t}(V^n_{w}(\gamma, \mathbb{F}))
=\ch(V^n_{w}(\gamma, \mathbb{F})) \,\ \,\  \text{and}\,\ \,\ 
\Pi \circ \ch_{\epsilon, t}(L^n_{w}(\gamma, \mathbb{F}))
=\ch(L^n_{w}(\gamma, \mathbb{F})). 
\]
We now compute the character $\ch_{\epsilon, t}(L^n_{w}(\gamma, \mathbb{F}))$ using the Steinberg tensor product Theorem~\ref{thm:Steinberg tensor} and Theorem~\ref{thm:reduced region}. 

For an element $\hat{\xi}=\sum_{i, b} c_{i, b} V_{i, b}^{v_{i,b}} W_{i, b}^{w_{i,b}}\in \Z[t, t^{-1}, V_{i,b}, W_{i, b}]_{i\in I, b\in (\Z/p\Z)^{n-t}}$, we define
\begin{align*}
\hat{\xi}^{[t]}:=\sum_{i, b} c_{i, b} (\prod_{\xi\in (\Z/p\Z)^t}V_{i, b\xi} )^{v_{i,b}} (\prod_{\xi\in (\Z/p\Z)^t}W_{i, b\xi})^{w_{i,b}}
\in \Z[t, t^{-1}, V_{i,b}, W_{i, b}]_{i\in I, b\in (\Z/p\Z)^{n}}.
\end{align*}
The following is a direct consequence of Corollary~\ref{cor:FrobIrr}. 
\begin{corollary}\label{cor:FrobChar}
Assume that $w$ is divisible by $p^t$.
Let $\alpha: \Lambda_t \to \varmathbb{G}$ and $\gamma'$ be as \S\ref{sec:quantum Frob}.  We have \begin{equation}\label{eq:L^n}
\ch_{\epsilon, t}(L_w^n(\alpha\times \gamma', \mathbb{F}))=\ch_{\epsilon, t}(L_{w/p^t}^{n-t}(\gamma', \mathbb{F}))^{[t]}
\end{equation}
as elements in $\Z[t, t^{-1}, V_{i,a}, W_{i, a}]_{i\in I, a\in (\Z/p\Z)^{n}}$.

\end{corollary}

Let $w=w^{(0)}+pw^{(1)}+\cdots +p^{n-1}w^{(n-1)}+p^nw'$, where $w^{(i)}\in X^+_{\red}$, for $i=0, 1, \cdots, n-1$ and $w'\in X^+$. 
As in \S~\ref{subsec:group_hom},  we assume  $\gamma= \gamma_0\times \cdots\times \gamma_n$, where each $\gamma_i:\Lambda_n\to \GL_{p^iw^{(i)}}$ can be written as $\gamma_i=\alpha_i\oplus \overline{\gamma_i}$ with $\alpha_i:\Lambda_i \to \GL_{p^iw^{(i)}}$.
For simplicity, we write $\ch_{\epsilon, t}(L_w^1(\gamma,\mathbb{F}))$ as $\hat{E}^1_{w,\gamma}$. 
The formula $\hat{E}^1_{w,\gamma}$ is known for type $ADE$ \cite[Theorem 8.4]{Nak04} in terms of the analogous Kazhdan-Lusztig polynomials (see also \cite[Conjecture 6.6]{H04} for a generalization).  
Iterating \eqref{eq:L^n} we obtain
\begin{equation}\label{eq:ch_et}
\ch_{\epsilon, t}(L_w^n(\gamma, \mathbb{F}))=(E^{(0)}_{w'})^{[n]}\prod_{i=0}^{n-1}(\hat{E}^1_{w^{(i)},\overline{\gamma_i}})^{[i]}.
\end{equation}
as elements in $\Z[t, t^{-1}, V_{i,a}, W_{i, a}]_{i\in I, a\in (\Z/p\Z)^{n}}$.

\subsection{Lusztig character formulas}
In this section, we computer the character $\ch(L_w^n(\gamma, \mathbb{F}))$ and compare it with Lusztig's character formula $E_w^{(n)}\in \Z[X]$ constructed in \cite{Lusz}. 

Recall that for an element $\xi=\sum_{\mu} c_{\mu}e^{\mu}\in \Z[X]$, we write $
\xi^{[t]}:=\sum_{\mu} c_{\mu}e^{p^t\mu}\in \Z[X]$.
Note that for an element $\hat{\xi}\in \Z[t, t^{-1}, V_{i,b}, W_{i, b}]_{i\in I, b\in (\Z/p\Z)^{n-t}}$, we have the equality
\[
\Pi(\hat{\xi}^{[t]})=(\Pi(\hat{\xi}))^{[t]}. 
\]
As a consequence of the formula \eqref{eq:L^n}, we have the following. 
\begin{corollary}
Assume the weight $w$ is of the form $w=p^n w'$ with $w'\in X^+$. Then, 
\[
\ch(L_w^n(\gamma, \mathbb{F}))=E^{(n)}_{w}. 
\]
\end{corollary}
\begin{proof}
Lusztig's character formula satisfies
$
E^{(n)}_{w}=(E^{(0)}_{ w'})^{[n]}. 
$
By \eqref{eq:L^n}, we have
\[
\ch(L_w^n(\gamma, \mathbb{F}))=\ch(L_{w'}^{0}(\mathbb{F}))^{[n]}
=(E^{(0)}_{w'})^{[n]}
=E^{(n)}_{w}. 
\]
The middle equality follows from the discussion of Example \ref{ex:two}. By \cite{Nak98}, $L_{w'}^{0}(\mathbb{F})$ is an irreducible representation of the enveloping algebra $\fg$ over characteristic zero with highest weight $w'$. Therefore, its character is given by the Weyl character formula $E^{(0)}_{w'}$. 
\end{proof}

In general, we conjecture that
\begin{conj}\label{conj:Lusztig}
Let $w\in X^+_{\red}$, there exists
 $\gamma$ satisfying \S~\ref{subsec:group_hom} with 
\[
\ch(L_w^1(\gamma, \mathbb{F}))=E^{(1)}_{w}. 
\]
\end{conj}
This conjecture immediately implies the following conjecture thanks to formula \eqref{eq:ch_et}. 
\begin{conj}\label{con:for n}
For general $w$, there exists a group homomorphism $\gamma$ such that $$\ch(L_w^{(n)}(\gamma, \mathbb{F}))= E^{(n)}_w. $$
\end{conj}
Below we prove Conjecture \ref{conj:Lusztig} when the quiver $Q$ is of type $A$. The proof replies on two facts. One of them is the well-known fact that 
 $E^{(1)}_{w}$ is the character formula of the irreducible representation of Lusztig's quantum group at a $p$-th root of $1$, which holds for any simple Lie algebra $\fg$. Another is a result of Chari and Pressley realizing in type $A$ each irreducible representation of Lusztig's quantum group as a quotient of the standard module of the quantum loop algebra. In principle the latter could be done in all types using the formula for the $\epsilon,t$-character. However, the combinatorics of both formulas are complicated, and hence we leave this out of the present paper.  

\subsection{Proof of Conjecture \ref{conj:Lusztig} in type $A$}
\subsubsection{Evaluation representations of quantum group at the roots of unity}
We recall some facts about the evaluation representations of the quantum loop algebra $U_{\epsilon}(L\mathfrak{sl}_{n+1})$, when $\epsilon$ is a root of unity. 
The references are \cite{CP94} (for a generic parameter) and \cite{AN06} (for a root of unity).

Consider the Lie algebra $\mathfrak{sl}_{n+1}$. Let $L\mathfrak{sl}_{n+1}=\mathfrak{sl}_{n+1}\otimes \C[t, t^{-1}]$ be the loop algebra of $\mathfrak{sl}_{n+1}$. 
We set $I=\{1, 2, \cdots, n\}$. Let $\{\alpha_i\}_{i\in I}$ be the set of simple roots of $\mathfrak{sl}_{n+1}$ and $\{\varpi_i\}_{i\in I}$ be the fundamental weights given by
\begin{equation}\label{eqn:fund_weight}
    \varpi_i:=\frac{1}{n+1}(
(n-i+1)\sum_{k=1}^{i} k\alpha_k +i \sum_{k=i+1}^n (n-k+1)\alpha_k). 
\end{equation}

Let $\epsilon$ be a primitive $l$-th root of unity and assume $gcd(l, n+1)=1$. We consider the quantum group $U_{\epsilon}(\mathfrak{sl}_{n+1})$ and 
the quantum loop algebra $U_{\epsilon}(L\mathfrak{sl}_{n+1})$ of restricted type defined by Lusztig \cite{Lu2, CP, AN06}. Let $U_{\epsilon}^{\fin}(\mathfrak{sl}_{n+1})\subset  U_{\epsilon}(\mathfrak{sl}_{n+1})$ and
$U_{\epsilon}^{\fin}(L\mathfrak{sl}_{n+1})\subset  U_{\epsilon}(L\mathfrak{sl}_{n+1})$ be the $\C$-subalgebras of small quantum groups (see \cite[Section 3.2]{AN06} for the definitions). %By the Steinberg tensor product theorem \cite{Lu2, CP}, it suffices to understand the representation theory of $U_{\epsilon}^{\fin}(\mathfrak{sl}_{n+1})$ and $U_{\epsilon}^{\fin}(L\mathfrak{sl}_{n+1})$. 
Every finite dimensional irreducible representation of $U_{\epsilon}^{\fin}(\mathfrak{sl}_{n+1})$ (resp. $U_{\epsilon}^{\fin}(L\mathfrak{sl}_{n+1})$) is a highest weight representation and is determined by the highest weight. There exists a one-to-one correspondence from the set of the highest weights to $(\Z/l\Z)^n$ (resp. $\C_l[t]^n$, where $\C_l[t]=\{P\in \C[t]\mid$ $P$ is not divisible by $1-ct^l$ for all $c\in \C^* \}$ ). 
We denote by $V(\lambda)$ the irreducible module of $U_{\epsilon}^{\fin}(\mathfrak{sl}_{n+1})$ corresponding to $\lambda\in (\Z/l\Z)^n$ and $V(\bold{P})$ the one of $U_{\epsilon}^{\fin}(L\mathfrak{sl}_{n+1})$ corresponding to $\bold{P}=(P_1, P_2, \cdots, P_n)\in \C_l[t]^n$.

Let $U_{\epsilon}(\mathfrak{gl}_{n+1})$ be the quantum group of $\mathfrak{gl}_{n+1}$. This is denoted by $U'_{\epsilon}(\mathfrak{sl}_{n+1})$ in \cite[Definition 2.1 and Section 4]{AN06}, a notation which we avoid using in order to reduce clashes. 

For any $a\in \C$, there exist evaluation homomorphisms of algebras \cite[Proposition 12.2.10]{CP94}\cite[Section 4 (4.2)]{AN06}
\[
\ev_a^{\pm}: U_{\epsilon}^{\fin}(L(\mathfrak{sl}_{n+1})) \to 
U_{\epsilon}^{\fin}(\mathfrak{gl}_{n+1}). 
\]
For any $\lambda\in (\Z/l\Z)^n$, we regard $V(\lambda)$ as a $U_{\epsilon}^{\fin}(\mathfrak{gl}_{n+1})$ representation (\cite[Section 6]{AN06}). Let $V_{a}(\lambda)^{\pm}$ be the $U_{\epsilon}^{\fin}(L(\mathfrak{sl}_{n+1}))$-module obtained from $V(\lambda)$ by using the evaluation homomorphism $\ev_{a_{\pm}^{\lambda}}^{\pm}$, where \cite[Definition 4.12]{AN06} for $\lambda=(\lambda_i)_{i\in I}\in (\Z/l\Z)^n$, we set $\lambda_{\varpi_i}:=\sum_{j\in I} \lambda_j(\varpi_i, \varpi_i)$ and 
\begin{align*}
& a^{\lambda}_{+}:=a\epsilon^{-\lambda_{\varpi_1}+\lambda_{\varpi_n}+n}, \,\ 
a^{\lambda}_{-}:=a(-1)^{n+1}\epsilon^{\lambda_{\varpi_1}-\lambda_{\varpi_n}+2n+1}. 
\end{align*}
%Here the formula of $\varpi_i$ is given by \eqref{eqn:fund_weight}.
 The evaluation representation $V_{a}(\lambda)^{\pm}$ is a finite dimensional irreducible $U_{\epsilon}^{\fin}(L(\mathfrak{sl}_{n+1}))$-module, as $V(\lambda)$ is irreducible as a $U_{\epsilon}^{\fin}(\mathfrak{sl}_{n+1})$-module. 
 
 \begin{theorem}
 \label{Thm:P}
 (\cite[Theorem 4.8 and Theorem 4.13]{AN06}, see also  \cite[Proposition 12.2.13]{CP94})
There exists a unique Drinfeld polynomial $\bold{P}_a^{\pm}=(P_{i, a})_{i\in I} \in \C_l[t]^n$, such that 
 \[
 V_{a}(\lambda)^{\pm}\cong V(\bold{P}_a^{\pm}) \,\ 
 \text{ as a $U_{\epsilon}^{\fin}(L(\mathfrak{sl}_{n+1}))$ module}. 
 \]
Explicitly, let $i\in I$ such that $P_{i, a}^{\pm}\neq 0$. Then , 
\[
P_{i, a}^{\pm}=\prod_{j=1}^{\lambda_i} (t-
a^{-1} \epsilon^{c^{\pm}_{i, j}} ), \,\ \text{where}\,\  
c^{\pm}_{i, j}:= \pm (\sum_{k=1}^{i-1} \lambda_k-\sum_{k=i+1}^{n}\lambda_k+i)+\lambda_i-2j+1. 
\]
\end{theorem}

\subsubsection{Characters of evaluation representations}

Now let $w=(w^i)_{i\in I}\in X^+$ be the highest weight. 
We have the isomorphisms
\[
K_{\Z/p\Z}(\pt)\cong \C[t^{\pm}]/(t^p=1) \,\ \text{and}  \,\ 
K_{T_{w}\times \C^*}(\pt)\cong \C[q^{\pm}]\otimes (\otimes_{i\in I} 
\C[z^{(i)}_1, \cdots, z^{(i)}_{w^i}]). 
\]
Let $\gamma=(\gamma^1, \cdots, \gamma^i): \Z/p\Z\to T_w\times \C^*\subset \GL_w\times \C^*$ be a group homomorphism such that
\[
\gamma^i: \Z/p\Z\to T_{w^i}\times \C^*, 1\mapsto \left(\begin{bmatrix}
\zeta^{n_1^{(i)}}&&\\
&\ddots&\\
&& \zeta^{n^{(i)}_{w^{i}}} \end{bmatrix}, \zeta\right). 
\]
where each $n_j^{(i)} \in \{0, 1, \cdots, p-1\}$. 
%with multiplicities $n_0^{(i)}, n_1^{(i)}, \cdots, n_{p-1}^{(i)}$ such that 
%$n_0^{(i)}+n_1^{(i)}+\cdots+n_{(p-1)}^{(i)}=w^{(i)}$.
Then, $\gamma$ induces the following ring homomorphism
\[
K_{T_{w}\times \C^*}(\pt)\to K_{\Z/p\Z}(\pt),  \,\ 
q\mapsto t, z^{(i)}_j\mapsto t^{n_j^{(i)}}. 
\]
\begin{corollary}\label{cor:for type A proof}
When the quiver is of type $A$, Conjecture \ref{conj:Lusztig} is true, with the choice of group homomorphism $\gamma^{\pm}$ as above where $n_i^{(i)}=c_{i, j}^{\pm}$. 
\end{corollary}
The formula of $c_{i,j}$ is given in Theorem \ref{Thm:P}.
\begin{proof}
The irreducible module $V(\bold{P}^{\pm})$ in Theorem \ref{Thm:P} factors through the quantum group $U_w^1(\gamma^{\pm}, \C)$ (\cite[Section 13]{Nak})
\[
U_{\epsilon}(L(\mathfrak{sl}_{n+1}))
\to U_w^1(\gamma^{\pm}, \C)
\to \End(V(\bold{P}^{\pm})). 
\]
In particular, we have the following equalities of characters
\[
\ch(V(\bold{P}^{\pm}))
=\ch(V(w))
=E^1_{w}. 
\]
\end{proof}
Consequently, Conjecture \ref{con:for n} is also true in type-$A$.

\appendix
\section{The coproduct and Stable envelope}
\label{app:coproduct}
In this section, we construct a coproduct of the convolution algebra associated to $E^*$. The construction relies on an $E^*$-class which is obtained from the Maulik-Okounkov stable envelope and the character map of Hopkins, Kuhn and Ravenal. We construct this $E^*$ class in \S\ref{E^* class} and the coproduct in \S\ref{sec:coproduct}. The main theorem of this section is Theorem \ref{thm:coproduct} the proof of which makes use of Proposition~\ref{prop:localization_convolution}.

\subsection{The stable envelopes of Maulik-Okounkov}
\label{sec:MO stab}
Let $X$ be a non-singular algebraic variety with a holomorphic symplectic form $\omega$. Assume there is a proper map $\pi: X \to X_0$ where $X_0$ is affine. Let $\varmathbb{G}$ be a reductive group  which acts on $X$, so that $\omega$ transforms via a character of $\varmathbb{G}$,  and $\pi$ is $\varmathbb{G}$-equivariant.
Let $A \subset T\subset \varmathbb{G}$ be a subtorus  such that $\omega$ is fixed by $A$. Let $\mathfrak{a}\subset \mathfrak{t}\subset\mathfrak{g}$ be the corresponding Lie algebras.  

Let  $\text{Cochar}(A):=\{\sigma: \C^*\to A\}$ be the lattice of cocharacters of $A$. We have the real vector space
\[
\mathfrak{a}_{\mathbb{R}}=\text{Cochar}(A)\otimes_{\bbZ}\bbR \subset \mathfrak{a}. 
\]
Define the torus roots to be the $A$-weights $\{\alpha_i\}$ occurring in the normal bundle to $X^A$ in $X$ \cite[Definition 3.2.1]{MO}. Note that each weight $\alpha$ of $A$ defines a rational hyperplane, denoted by $\alpha^{\perp}$, in $\mathfrak{a}_{\mathbb{R}}$. The root hyperplanes partition $\mathfrak{a}_{\mathbb{R}}$ into finitely many (open) chambers $\mathfrak{a}_{\mathbb{R}}\setminus \cup\alpha_i^{\perp}
=\sqcup \mathfrak{C}_i$. 

Let $\mathfrak{C}=\mathfrak{C}_i$ be a chamber for some $i$. 
For $Y\subset X^A$, we denote by $\Att_{\mathfrak{C}}(Y)$ the set of points attracted to $Y$. By definition, 
\[
\Att_{\mathfrak{C}}(Y):=\{
(x, y)\in X\times Y\mid \lim_{t\to 0} \sigma(t) x=y, \text{for all $\sigma: \C^*\to A$ in $\mathfrak{C}$}. 
\}
\]
In particular, we have a limit map $\Att_{\mathfrak{C}}(Y)\to Y$ and a closed embedding $\Att_{\mathfrak{C}}(Y) \to X$. 

Define a partial order on the set of connected components of $X^A$ as follows. 
\begin{align*}
&Z\geq Z' 
\Leftrightarrow  \overline{\Att_{\mathfrak{C}}(Z)}\cap Z'\neq \emptyset 
\end{align*}
For any connected component $Z$, define the \textit{full attacting set} to be
\[
\Att^f_{\mathfrak{C}}(Z)
:=\coprod_{Z'\leq Z} \Att_{\mathfrak{C}}(Z'). 
\]
As shown in \cite[Lemma 3.2.7]{MO}, $\Att^f_{\mathfrak{C}}(Z)$ is closed in $X$.

Let $Z\subset X^A$ be a connected component. Denote by $\pi_Z: \Att_{\mathfrak{C}}(Z)\times Z\to Z\times Z$ the natural map. 
Let 
\[
L_{\mathfrak{C}}
:=\cup_{Z} \overline{ \pi_Z^{-1}(\Delta)}
\subset X\times X^A
\]
be the closure of the preimage of $\Delta$, where the union is over all connected components of $X^A$. Then, $L_{\mathfrak{C}}$ is an $A$-invariant Lagrangian subvariety and $L_{\mathfrak{C}}|_{X\times Z}$ is supported on $\Att^f_{\mathfrak{C}}(Z)$. 
\[
\xymatrix{
&L_{\mathfrak{C}}  \ar@{}[r]|-*[@]{\subset} \ar[ld]_{p_1} \ar[rd]^{p_2}& X^A\times X    \\
X^A & & X
}
\]
where $p_1: L_{\mathfrak{C}} \to X^A$ and $p_2: L_{\mathfrak{C}} \to X$ are the compositions of the inclusion $L_{\mathfrak{C}}\subset X^A\times X$ with the two projections from $X^A\times X$ respectively. Note that $L_{\mathfrak{C}}$ is a Lagrangian subvariety. 
%and it satisfies the conditions (1) (2) in , but not condition (3) there. 

We fix a virtual bundle $T^{1/2}X \in K_{T}(X)$, such that, 
\[
TX=T^{1/2}X+\hbar^{-1} \otimes (T^{1/2}X)^{\vee}, 
\]
referred to as the polarization of $X$.
The product of the non-zero $A$-weights in the restriction of $T^{1/2}X$ to $X^A$ gives a polarization $\epsilon$ in the sense of \cite[Definition 3.3.1]{MO}.

Choose a chamber $\mathfrak{C}$ and a polarization $T^{1/2}X$. We have the following characterization of the stable envelope of Maulik and Okounkov.
Let $Z$ be a component of $X^A$ and let $N_Z$ be the normal bundle to $Z$ in $X$. We have the $T$-invariant decomposition
\[
N_Z=N_{-}\oplus N_{+}
\] into $A$-weights that are positive and negative on $\mathfrak{C}$ respectively. 
\begin{theorem}
\cite[Theorem 3.3.4 and Proposition 3.5.1]{MO}
\label{thm:stable envelope MO}
There exists a unique cohomology class $\Stab_{\mathfrak{C}}^{MO} \in H_T(X\times X^A)$
such that
\begin{enumerate}
\item The restriction of $\Stab_{\mathfrak{C}}^{MO}$ to $X\times Z$ is supported on $\Att^f_{\mathfrak{C}}(Z)\times Z$. 
\item The restriction to $Z\times Z$ equals $e(N_{-})\cap \Delta$,  according to the polarization. Here $e(N_{-})$ is the $A$-equivariant Euler class of $N_{-}$. 
\item For $Z'< Z$, the restriction of $\Stab_{\mathfrak{C}}^{MO}$ to $Z'\times Z$ has $A$-degree less than $\frac{1}{2}\codim (Z')$. 
\end{enumerate}
\end{theorem}

\subsection{A Morava $E^*$-class}
\label{E^* class}
Following the same notations as before, let $X$ be an algebraic variety with a symplectic form $\omega$ and a proper map $\pi: X\to X_0$. Let $\varmathbb{G}$ be a reductive group acting on $X$. Consider the inclusions $A\subset T \subset \varmathbb{G}$ of subgroups, where $T$ is the maximal torus and $A$ is a subtorus. They satisfy the assumptions in \S\ref{sec:MO stab}. 

Let $G\subset \varmathbb{G}$ be a finite subgroup such that $G$ commutes with $A$ and that $G$ fixes $\omega$. 
For any group homomorphism \[
\alpha: \Lambda_t \to G \subset \varmathbb{G}, 
\] as $\pi$ is $\varmathbb{G}$-equivariant, we have an induced proper map $\pi_{\alpha}: X^{\im \alpha}\to X_0^{\im \alpha}$ on the fixed point locus. 
The centralizer $C_{\im \alpha}(\varmathbb{G})$ of $\im \alpha$ in $\varmathbb{G}$ acts on $X^{\im \alpha}$. 
We have inclusions of subgroups $A\subset T \subset C_{\im \alpha}(\varmathbb{G})$, where $T$ is the maximal torus  of $C_{\im \alpha}(\varmathbb{G})$. By assumption, $\im \alpha$ fixes the symplectic form $\omega$, thus $X^{\im \alpha}$ is symplectic. 

We now apply the Maulik-Okounkov construction to the variety $X^{\im \alpha}$ with actions of $C_{\im \alpha}(\varmathbb{G})$ and $A$. Choose the chamber $\mathfrak{C}$ in $\mathfrak{a}=\Lie(A)$ of $A$. As the actions of $\im \alpha$ and $A$ commutes, the decomposition $X^A=\sqcup_{Z\in \pi_0(X^A)} Z$ gives a decomposition
\[
X^{\im \alpha, A}=\sqcup_{Z\in \pi_0(X^A)} Z^{\im \alpha}. 
\]

By definition, we have
\begin{align*}
\Att_{\mathfrak{C}}(Z^{\im \alpha})
=
\{(x, y)\in X^{\im \alpha}\times Z^{\im \alpha}\mid \lim_{t\to 0} \sigma(t)x=y, \forall \sigma: \C^*\to A  \,\ \text{in $\mathfrak{C}$}\}
=\Att_{\mathfrak{C}}(Z)^{\im \alpha}
\end{align*}
Let $\pi_\alpha: \Att_{\mathfrak{C}}(Z^{\im \alpha})\times Z^{\im \alpha}\to Z^{\im \alpha}\times Z^{\im \alpha}$ be the natural map. Let 
\[
L_{\mathfrak{C},\alpha}
=\cup_{Z \in \in \pi_0(X^{A})} \overline{\pi_{\alpha}^{-1}(\Delta_{Z^{\im \alpha}})}
\subset X^{\im \alpha}\times X^{\im \alpha, A}. 
\] 
be the Lagrangian subvariety. We then have
\[L_{\mathfrak{C}_\alpha,\alpha}=L_{\mathfrak{C}}^{\im \alpha}.
\]
For each $\alpha \in \hom(\Lambda_n, G)$,  the stable envelope constructed in Theorem \ref{thm:stable envelope MO} for $X^{\im \alpha}$
induces $\Stab_\alpha^{MO} \in 
L(E^*)\otimes_{E^*}E^*_{T}( 
L_{\mathfrak{C}}^{\im \alpha})$, since $L(E^*)$ has additive formal group law. 
\begin{lemma}
We have
$\sum_\alpha \Stab_\alpha^{MO}\in L(E^*)\otimes_{E^*} E^*_{T}( 
\sqcup_{\alpha\in \hom(\Lambda_n, G)} 
L_{\mathfrak{C}}^{\im \alpha})$ is $G$-invariant. 
\end{lemma}
\begin{proof}
Let $g$ be an element in $G$. As $G$ and $A$ commute, the action of $G$ on $X$ gives a map
\[
g: X^{\im g^{-1}\alpha g}\times (X^{\im g^{-1}\alpha g})^{A}
\to X^{\im \alpha}\times (X^{\im \alpha })^A. 
\]
Furthermore, 
\[
g(\Att_{\mathfrak{C}} (Z^{\im g^{-1}\alpha g}))=\Att_{\mathfrak{C}}(Z^{\im \alpha}), \,\ \,\  \text{and}  \,\ \,\ 
g(L_{\mathfrak{C}}^{\im g^{-1}\alpha g})= L_{\mathfrak{C}}^{\im \alpha}. 
\]

% The decomposition $X^A=\sqcup_{Z\in \pi_{0}(X^A)} Z$ gives the decompositions
% \[
% (X^{\im g^{-1}\alpha g})^{A}=\sqcup_{Z\in \pi_{0}(X^A)} Z^{\im g^{-1}\alpha g}, \,\ 
% (X^{\im \alpha})^{A}=\sqcup_{Z\in \pi_{0}(X^A)} Z^{\im \alpha}. 
% \]
% As $G$ commutes with $A$, we have
% \[

% \]
% As a consequence, 
% \begin{align*}
% &Z^{\im \alpha}\geq Z'^{\im \alpha} 
% \Leftrightarrow  \overline{\Att_{\mathfrak{C}}(Z^{\im \alpha})}\cap Z'^{\im \alpha}\neq \emptyset 
% \Leftrightarrow  \overline{\Att_{\mathfrak{C}}(Z^{\im g^{-1}\alpha g})}\cap Z'^{\im g^{-1}\alpha g}\neq \emptyset 
% \Leftrightarrow  Z^{\im g^{-1}\alpha g}\geq Z'^{\im g^{-1}\alpha g}. 
% \end{align*}
% This shows that 
% \[

% \]
Let
\[
g^*: L(E^*)\otimes_{E^*} E^*_T(L_{\mathfrak{C}}^{\im \alpha })\to L(E^*)\otimes_{E^*} E^*_{T}(L_{\mathfrak{C}}^{\im g^{-1}\alpha g})
\]
be the pullback along $g$. Denote by $g^*(\Stab_\alpha^{MO})$ the image of $\Stab_\alpha^{MO}$ in $L(E^*)\otimes_{E^*} E^*_{T}(L_{\mathfrak{C}}^{\im g^{-1}\alpha g})$. 
We now verify that 
\[
g^*(\Stab_\alpha^{MO})=\Stab_{g^{-1}\alpha g}^{MO}. 
\] This follows from the uniqueness in  Theorem \ref{thm:stable envelope MO}. We spell out the argument below.

(1) The support condition. 
We know that $L_{\mathfrak{C}}^{\im g^{-1} \alpha g} |_{ X^{\im g^{-1} \alpha g}\times Z^{\im g^{-1}\alpha g}}= \Att^{f}_{\mathfrak{C}}(Z^{g^{-1}\alpha g})$. Under the $g$-action, it is mapped to $g(\Att^{f}_{\mathfrak{C}}(Z^{\im g^{-1}\alpha g}))= \Att^{f}_{\mathfrak{C}}(Z^{\im \alpha })$. Therefore, the restriction of $g^*(\Stab_\alpha^{MO})$ to $X^{\im g^{-1}\alpha g}\times Z^{\im g^{-1}\alpha g}$ is supported on $\Att^{f}_{\mathfrak{C}}(Z^{g^{-1}\alpha g}) \times Z^{\im g^{-1}\alpha g}$. 

(2) The leading term condition. 
By definition
$\Stab_{\alpha}^{MO}|_{Z^{\im \alpha}\times Z^{\im \alpha}}=e(N_{-}^{\alpha}) \cap \Delta_{Z^{\im \alpha}}$, where $N^{\alpha}$ is the normal bundle to $Z^{\im \alpha}$ in $X^{\im \alpha}$. 
Consider the following commutative diagram
\[
\xymatrix{
N_{-}^{g^{-1}\alpha g} \ar[d] \ar[r]^{g}& N_{-}^{\alpha } \ar[d]\\
Z^{\im g^{-1}\alpha g} \ar[r]^{g} \ar@/^1.0pc/[u]^{s_2}& Z^{\im \alpha } \ar@/_1.0pc/[u]_{s_1}
}
\]
where $s_1, s_2$ are the zero sections of $N_-^{\alpha}$ and $N_-^{g^{-1}\alpha g}$. We have the equalities  
\[g^*(e(N_-^{\alpha}))=g^*s_1^*s_{1*}(1)=s_2^*g^*s_{1*}(1)=s_2^*s_{2*}(1)=e(N_-^{g^{-1}\alpha g}).
\]
Therefore, 
\[g^*(\Stab_{\alpha}^{MO})|_{Z^{\im g^{-1}\alpha g}\times Z^{\im g^{-1}\alpha g}}=e(N_-^{\im g^{-1}\alpha g})\cap \Delta_{Z^{\im g^{-1}\alpha g}}.
\]

(3) The degree condition. For $Z'<Z$, 
we have
\[ \deg_A( g^*\Stab_{\alpha }^{MO}|_{Z'^{\im g^{-1}\alpha g} \times Z^{\im g^{-1}\alpha g}})=\deg_A( \Stab_{\alpha }|_{Z'^{\im \alpha} \times Z^{\im \alpha }})< \frac{1}{2} \codim (Z'^{\im \alpha} )=\frac{1}{2} \codim (Z'^{\im g^{-1}\alpha g} )
.\] 
We conclude that $\sum_\alpha \Stab_\alpha^{MO}$ is $G$-invariant. 
\end{proof}
By an abuse of notation, we still denote by $\Stab_\alpha^{MO}$ the image of $\Stab_\alpha^{MO}$ under the map
\[
L(E^*)\otimes_{E^*}E^*_{T}( 
L_{\mathfrak{C}}^{\im \alpha})
\to L(E^*)\otimes_{E^*}E^*_{\C^*}( 
L_{\mathfrak{C}}^{\im \alpha}). 
\]

We now apply the \cite[Theorem C]{HKR} to the Lagrangian subvariety $L_{\mathfrak{C}}$ with the action of $G\times \C^*\subseteq \bbG$, where $G$ is a finite group fixing $\omega$ of $X$, and $\C^*$ scales $\omega$ via a character. 
Applying the construction of \cite{HKR}, we obtain a morphism 
\begin{equation}
\label{eq:HKR for L}
HKR:  L(E^*)\otimes_{E^*} E^*_{G}(E\C^*\times _{\C^*}L_{\mathfrak{C}}) \to 
L(E^*)\otimes_{E^*} E^*( 
\sqcup_{\alpha\in \hom(\Lambda_n, G)} E\C^*\times _{\C^*}
L_{\mathfrak{C}}^{\im \alpha})^G. 
\end{equation}
Note that \cite[Theorem C]{HKR} applies to finite $G$-CW complexes only. Nevertheless, 
we claim that the morphism \eqref{eq:HKR for L} is an isomorphism. Indeed,  $B\mathbb{C}^*=\mathbb{P}^{\infty}$, and hence the above morphism can be realized as the limit of 
\[
HKR_k:  L(E^*)\otimes_{E^*} E^*_{G}((\C^{k+1}\setminus\{0\})\times _{\C^*}L_{\mathfrak{C}}) \xrightarrow{\cong} 
L(E^*)\otimes_{E^*} E^*( 
\sqcup_{\alpha\in \hom(\Lambda_n, G)} (\C^{k+1}\setminus\{0\})\times _{\C^*}
L_{\mathfrak{C}}^{\im \alpha})^G. 
\]
By \cite[Theorem C]{HKR}, the above morphism $HKR_k$ is an isomorphism. Thus, \eqref{eq:HKR for L} is an isomorphism, which can be written as
$
L(E^*)\otimes_{E^*} E^*_{G\times \C^*}(L_{\mathfrak{C}}) \to 
L(E^*)\otimes_{E^*} E^*_{\C^*}( 
\sqcup_{\alpha\in \hom(\Lambda_n, G)} 
L_{\mathfrak{C}}^{\im \alpha})^G. $

\begin{definition}\label{def:StabE}
Define the $E^*$-class $\Stab^{E^*}_{\mathfrak{C}}\in L(E^*)\otimes_{E^*} E^*_{G\times \C^*}(L_{\mathfrak{C}})$ by
\[
\Stab^{E^*}:={HKR}^{-1}( \sum_\alpha \Stab_\alpha^{MO} 
e(N(X^{\im \alpha }\subset X)), 
\]
where $N(e(N(X^{\im \alpha }\subset X)))$ is the normal bundle of the embedding $X^{\im \alpha }\subset X$. 
\end{definition}
Let $HKR_\alpha$ be the restriction of $HKR$ \eqref{eq:HKR for L} to the component $\alpha$.  
Modify $HKR_\alpha$ by $\frac{HKR_\alpha}{e(N(X^{\im \alpha }\subset X))}$ and denote the resulting map $\sum_\alpha \frac{HKR_\alpha}{e(N(X^{\im \alpha }\subset X))}$ by $\frac{HKR}{e(N(X^{\im \alpha }\subset X))}$ for simplicity. 
Then, under the map $\frac{HKR}{e(N(X^{\im \alpha }\subset X))} $, we have
\[
\Stab^{E^*}\mapsto \sum_\alpha \Stab_\alpha^{MO}. 
\]

\subsection{The dynamical twist}
In this section, we define a version of convolution algebra and standard module with dynamical parameters. The usual standard modules are recovered when the dynamical twist is trivial. The dynamical twist is not necessary in the present paper, although it naturally appears in all the cohomology theories. We include it here for completeness and future references. 

We begin with the parameter space of the dynamical parameters. This is motivated and reformulated from the dynamical shifts in \cite{AO, O}. Let $E^*$ be any oriented cohomology theory. Let $\tilde{G}:=G\times \C^*$. 

Suppose $X=\bar{X}//H$ is obtained by a GIT-quotient, and it carries a representation of $G$. For example, we may take $X$ to be the quiver varieties. 
Let \[
\mathbb{B}:=\Spec E^*_{\C^*}(\pt)\otimes \Hom(H, \C^*). 
\]
We have the isomorphism of $\Spec E^*_{\C^*}(\pt)$ with the formal group law $F$ of $E^*$. There is a group structure on $\mathbb{B}$ for which we denote by $\otimes$.  

The projection
$X=\bar{X}//H \to \pt/H$ 
induces a map
\begin{equation}
\label{eq:B1}
\Spec(E^*_{\tilde G}(X))\times \mathbb{B}
=
\Spec(E^*_{H\times \tilde G}(\bar{X}))\times \mathbb{B}
\to 
\Spec(E^*_{H}(\pt))\times \mathbb{B}
\end{equation}
Each character $\sigma \in \Hom(H, \C^*)$ of $H$ induces a morphism
\[
\sigma: \Spec E^*_{H}(\pt)\to \Spec E^*_{\C^*}(\pt).
\]
These morphisms give rise to the following natural map
\begin{equation}
\label{eq:B2}
 \Spec(E^*_{H}(\pt))\times \mathbb{B}=\Spec(E^*_{H}(\pt))\times\Spec E^*_{\C^*}(\pt)\otimes \Hom(H, \C^*)
\to \Spec E^*_{\C^*}(\pt)\times \Spec E^*_{\C^*}(\pt). 
\end{equation}
Note that $\Spec E^*_{\C^*}(\pt)$ is the formal group law $F$ of $E^*$. On $F\times F$, we have the universal line bundle denoted by $\rho_{F}$. Indeed, we have a map $F\times F\to F\times \Pic(F), (a, b)\mapsto (a, \mathcal{O}(0-b)$, the fiber of $\rho_F$ at the point $(a, b)$ is $\calO(0-b)_{a}$. 
\begin{definition}
\label{def:Poincare}
Let $\rho$ be the line bundle on 
\[
\Spec(E^*_{\tilde{G}}(X))\times \mathbb{B}
\]
obtained by pulling back $\rho_F$ along \eqref{eq:B2} and \eqref{eq:B1}, called the Poincare line bundle.  
\end{definition}

Let $V$ be a $\tilde{G}$-equivariant vector bundle on $X$, let $e^{E^*}(V):=s^*s_*(1)\in E^*_{\tilde{G}}(X)$ be the equivariant Euler class of $V$, where $s$ is the zero section of $V$. Then, the ideal sheaf $(e^{E^*}(V))$ is a line bundle on $\Spec E^*_{\tilde{G}}(X)$, denoted by $\Theta^{E^*}(V)$. 
The assignment $V \mapsto \Theta^{E^*}(V)$ extends to a group homomorphism \cite[Proposition 3.12]{ZZ}
\[
\Theta^{E^*}: K_{\tilde{G}}(X)\to \Pic( \Spec E^*_{\tilde{G}}(X)). 
\] 
We are interested in the line bundle
\[
\Theta^{E^*}(T^{1/2} X)\otimes \rho \,\ \,\  \text{on} \,\ \,\ 
\Spec( E_{\tilde{G}}^*(X))\times \mathbb{B}. 
\]

% When $E^*$ is the cohomology, we can identify $\Spec H^*(\pt)\otimes_{\bbZ} \bbZ^{|I|}$ with the Cartan subalgebra $\mathfrak{h}=\C^{|I|}$ of the Lie algebra associated to the quiver. 
% When $E^*$ is the K-theory, we can identify $\Spec E^*(\pt)\otimes_{\bbZ} \bbZ^{|I|}$ with the torus $=\C^{*|I|}$ of rank $|I|$. 
% When $E^*$ is the elliptic theory, we can identify $\Spec E^*(\pt)\otimes_{\bbZ} \bbZ^{|I|}$ with the $|I|$-power of elliptic curves  $E^{*|I|}$. 

%======================
%======================
%I changed the direction of the following map=====
%======================
%======================
% The pullback along $p_1$ gives a map $E^*_G(X^A)\to E^*_G(L)$ and pushforward along $p_2: \Theta^{E^*}(p_2)\to E^*_G( X)$. 
% We consider the following composition induced by convolution with the stable class $\Stab^{E^*}$. 
% \begin{align*}
% p_{2*}(p_1^*(\bullet)\cdot \Stab^{E^*}): 
% \Theta^{E^*}(p_2)\otimes \Theta^{E^*}(T^{1/2}X|_{X^A})
% \xrightarrow{p_1^*} \Theta^{E^*}(p_2)\otimes
% \Theta^{E^*}(T^{1/2}X|_{L} )
% \xrightarrow{\cdot \Stab^{E^*}} \Theta^{E^*}(p_2)\otimes
% \Theta^{E^*}(T^{1/2}X|_{L} )
% \xrightarrow{p_{2*}} \Theta^{E^*}(T^{1/2}X). 
% \end{align*}
% Equivalently, we have
% \[
% p_{2*}(p_1^*(\bullet)\cdot \Stab^{E^*}): \Theta^{E^*}(T^{1/2}X^A)\otimes \Theta^{E^*}(\delta\nu)\otimes \rho\to
% \Theta^{E^*}(T^{1/2}X) \otimes \rho, 
% \]
% where $\rho \in Pic^0( E^*(X(w)^A))$ is any degree zero line bundle and
% \[
% \Theta^{E^*}(\delta\nu):=
% (T^{1/2}X^A)^{-1}\otimes 
% \Theta^{E^*}(p_2)\otimes \Theta^{E^*}(T^{1/2}X|_{X^A})
% \]
% is the dynamical shift. 
Recall we have the following two maps
\[
\xymatrix{
&L_{\mathfrak{C}}  \ar@{}[r]|-*[@]{\subset} \ar[ld]_{p_1} \ar[rd]^{p_2}& X^A\times X    \\
X^A & & X
}
\]

We have the pullback $p_2^*: E^*_{\tilde{G}}(X)\to E^*_{\tilde{G}}(L_{\mathfrak{C}})$ (ring homomorphism) along $p_2$. Let $\Theta^{E^*}(p_1))$ denote the image under  $\Theta^{E^*}$ of the relative tangent bundle of $p_1$. Note that due to the singularities of $L_{\mathfrak{C}}$, the definition of $\Theta(Tp_1)$ is via the smooth embedding \S~\ref{subsec:conv} similar to \cite[2.5.2]{GKV}.
Then, $\Theta^{E^*}(p_1)$ is a module over the ring $E_{\tilde{G}}^*(L_{\mathfrak{C}})$. We view $\Theta^{E^*}(p_1)$ as a module over $E_{\tilde{G}}^*(X^A)$ via the ring homomorphism $p_1^*: E^*_{\tilde{G}}(X^A)\to E^*_{\tilde{G}}(L_{\mathfrak{C}})$. We have the pushforward along $p_1$, which is a module morphism over the ring $E_{\tilde{G}}^*(X^A)$. 
\[
p_{1*}: \Theta^{E^*}(p_1)\to E^*_{\tilde{G}}( X^A). \]
Convolution with the class $\Stab^{E^*}$ from Definition \ref{def:StabE} induces the following morphism
\[
p_{1*}(p_2^*(\bullet)\cdot \Stab^{E^*}): \Theta^{E^*}(T^{1/2}X)\otimes \rho\to
\Theta^{E^*}(T^{1/2}X^A) \otimes \rho_A \otimes \delta,
\]
where 
\begin{itemize}
\item
$\Theta^{E^*}(T^{1/2}X)\otimes \rho$ is  a line bundle on the space $\Spec( E_{\tilde{G}}^*(X))\times \mathbb{B}$,
\item
$\Theta^{E^*}(T^{1/2}X^A)\otimes \rho_A$ is on $\Spec( E_{\tilde{G}}^*(X^A))\times \mathbb{B}$ with  $\rho_A$ the Poincare line bundle,
\item The following line bundle on $\Spec( E_{\tilde{G}}^*(X^A))\times \mathbb{B}$
\begin{equation}\label{eq:shift d}
\delta:=
\Theta^{E^*}(T^{1/2}X^A)^{-1}\otimes 
\Theta^{E^*}(p_1)^{-1}\otimes \Theta^{E^*}(T^{1/2}X) \otimes \rho\otimes (\rho_A)^{-1}
\end{equation}
For the embedding $X^{A} \subset X$, we have the natural map $
\Spec(E^*_{\tilde{G}}(X^A))\to \Spec(E^*_{\tilde{G}}(X)))$. 
Here we restrict $\Theta^{E^*}(T^{1/2}X)$ to $X^A$. We call $\delta$ the dynamical shift. 
\end{itemize}
%The above morphism is obtained from the following composition \begin{align*} p_{1*}(p_2^*(\bullet)\cdot \Stab^{E^*}):  E_{\tilde{G}}^*(X)\otimes \Theta^{E^*}(p_1)& \xrightarrow{p_2^*} E_{\tilde{G}}^*(L_{\mathfrak{C}})\otimes \Theta^{E^*}(p_1)\\& \xrightarrow{\cdot \Stab^{E^*}} \Theta^{E^*}(p_1) \xrightarrow{p_{1*}} E_{\tilde{G}}^*(X^A).  \end{align*}

\begin{lemma}
The degree of $\delta$ is zero. 
\end{lemma}
\begin{proof}
We know that $\rho$ and $\rho^A$ are degree zero line bundles. It suffices to show 
$\Theta^{E^*}(T^{1/2}X^A)^{-1}\otimes 
\Theta^{E^*}(p_1)^{-1}\otimes \Theta^{E^*}(T^{1/2}X)$ has degree zero. 

We have
\begin{align*}
T^{1/2}X-Tp_1-T^{1/2}X^A
=&T^{1/2}X-T^{1/2}X^A-TL_{\mathfrak{C}}+TX^A\\
=&T^{1/2}X+\hbar^{-1}\otimes (T^{1/2} X^A)^{\vee}-TL_{\mathfrak{C}}.  
\end{align*}
As $L_{\mathfrak{C}}$ is Lagrangian, then $TL_{\mathfrak{C}}$ is a polarization of $X\times X^A$. We have $T^{1/2}X+\hbar^{-1}\otimes (T^{1/2} X^A)^{\vee}$ is another polarization of $X\times X^A$. 
Thus, $T^{1/2}X-Tp_1-T^{1/2}X^A$ is a difference of two polarizations. See \cite[Section 3.2.2]{AO} for the formula in terms of Chern roots. Therefore, $\Theta^{E^*}(T^{1/2}X-Tp_1-T^{1/2}X^A)$ has degree zero. 
\end{proof}

\subsection{The convolution algebra and its module}
\label{sec:convolution algebra}
In this section, we introduce the convolution algebra when there is a dynamical parameter. 
Let $Z_{12}\subset X_1\times X_2$. 
Consider the following diagram
\[
\xymatrix{
&Z_{12}  \ar@{}[r]|-*[@]{\subset} \ar[ld]_{\pi_1} \ar[rd]^{\pi_2}&  X_1\times X_2    \\
X_1 & & X_2
}
\]
where $\pi_i: Z_{12} \to X_i$ are the composition of the inclusion $Z_{12}\subset X_1\times X_2$ with the two projections from $X_1\times X_2$ respectively. Let us consider
\begin{equation}\label{eq:conv algebra}
\calU_{12}=
(\Theta^{E^*}(T^{1/2}(X_2))\otimes \rho_2)^{-1}
\otimes \Theta^{E^*}(\pi_1)
\otimes (\Theta^{E^*}(T^{1/2}(X_1)\otimes \rho_1), 
\end{equation}
where we view $\Theta^{E^*}(T^{1/2}(X_i))\otimes \rho_i$ as a module over $E^*_{{\tilde{G}}}(Z_{12})$ via $\pi_i^*$. The tensor product in \eqref{eq:conv algebra} is 
over modules of $E^*_{{\tilde{G}}}(Z_{12})$. 

We have, with $\Theta^{E^*}(T^{1/2}X)\otimes \rho$, the action map
\begin{equation}\label{equ:action Weyl}
\pi_{1*}\pi_2^*: 
\calU_{12}\otimes 
\Big(\Theta^{E^*}(T^{1/2}(X_2)\otimes \rho_2\Big)
\to \Theta^{E^*}(T^{1/2}(X_1))\otimes \rho_1. 
\end{equation}

\subsection{The coproduct}
\label{sec:coproduct}
In this section, we define the coproduct $\Delta^{E^*}$ of the convolution algebra  in a similar way as \cite{Nak12}. 
Taking the opposite chamber $\mathfrak{C}^{\op}$ in $\mathfrak{a}_{\bbR}$, we have the Lagrangian subvariety $L_{\mathfrak{C}^{\op}}\subset X\times X^A$. Let $(\Stab^{E^*}_{\mathfrak{C}^{\op}})^{-1} \in L(E^*)\otimes_{E^*} E^*(E{\tilde{G}}\times_{\tilde{G}} L_{\mathfrak{C}^{\op}})$ be the $E^*$ class such that
\[
\frac{HKR}{e(N(X^{\im \alpha A}\subset X^A))} ((\Stab^{E^*}_{\mathfrak{C}^{\op}})^{-1})
=(\Stab_{\alpha, \mathfrak{C}^{\op}}^{MO})^{-1},  
\]
where $(\Stab_{\alpha, \mathfrak{C}^{\op}}^{MO})^{-1}$
is the class $c^{-1}$ defined in \cite[Proposition 3.7]{Nak12} by taking $c$ to be $\Stab_{\alpha, \mathfrak{C}^{\op}}^{MO}$. Explicitly, let $(\Stab_{\alpha, \mathfrak{C}^{\op}}^{MO})^{-}$ to be the class defined similarly to $\Stab_{\alpha, \mathfrak{C}}^{MO}$ by choosing the opposite chamber $\mathfrak{C}^{\op}$ and swapping the first the second factor of $X^A\times X$. Define
\[
(\Stab_{\alpha, \mathfrak{C}^{\op}}^{MO})^{-1}
:= ((\Stab_{\alpha, \mathfrak{C}^{\op}}^{MO})^{-}
* \Stab_{\alpha, \mathfrak{C}}^{MO})^{-1} * (\Stab_{\alpha, \mathfrak{C}^{\op}}^{MO})^{-}, 
\]
where $*$ is the convolution.

Consider the following diagram
\[
\xymatrix{ 
&L_{\mathfrak{C}^{\op}}\times L_{\mathfrak{C}} \ar[ld]_{p_{12}\times p_{34}}
\ar[d]^{p_{23}} \ar[rd]^{p_{14}} \ar@{}[r]|-*[@]{\subset}& (X_1\times X^A_1)\times (X^A_2\times  X_2)\\
(X_1\times X^A_1)\times (X^A_2\times  X_2)  & 
X^A_1\times X^A_2
&  X_1\times X_2
}
\]
where $p_{ij}$ are the composition of the inclusion $L_{\mathfrak{C}^{\op}}\times L_{\mathfrak{C}} \subset X_1\times X^A_1\times  X^A_2\times  X_2 $ with the $(i,j)$th projection. 
  
Motivated by the construction of coproduct in \cite{Nak12},  we consider the convolution 
\[
p_{23*}\Big( p_{12}^* (\Stab^E_\mathfrak{C^{\op}})^{-1}\cdot p_{14}^* (\bullet) \cdot p_{34}^*(\Stab^E_{\mathfrak{C}})\Big).
\]It induces the following morphism, which will be the coproduct of the convolution algebra. 
\begin{align}\label{DeltaE}
\Delta^{E^*}: 
\calU_{12}\to
(\Theta^{E^*}(T^{1/2}(X_2^A))\otimes \rho_{2, A}\otimes \delta_2)^{-1}
\otimes \Theta^{E^*}(\pi_1^A)
\otimes (\Theta^{E^*}(T^{1/2}(X_1^A)\otimes \rho_{1, A}\otimes \delta_1)
=: \calU_{12}^A, 
\end{align}
where $\pi_1^A: Z_{12}^A\to X_1^A$ and
where $\rho_{i, A}$ is the Poincare line bundle on $\Spec(E^*_{{\tilde{G}}}(X_i^A))\times \mathbb{B}$ and $\delta_i$ is the dynamical shifts in \eqref{eq:shift d}. The morphism is induced from the composition of the following map
\begin{align*}
E^*_{{\tilde{G}}}(X_1\times X_2)\otimes \Theta^{E^*}(p_{23})\xrightarrow{p_{14}^* }E^*_{{\tilde{G}}}(L_{\mathfrak{C}^{\op}}\times L_{\mathfrak{C}})\otimes \Theta^{E^*}(p_{23})
\xrightarrow{(\Stab^{E}_{\mathfrak{C}^{\op}})^{-1}\boxtimes \Stab^{E}_{\mathfrak{C}}}\Theta^{E^*}(p_{23})
\xrightarrow{(p_{23})_*} E_{{\tilde{G}}}^*(X_1^A\times X^A_2). 
\end{align*}

The coproduct \eqref{DeltaE} is compactible with their respect actions on the modules. That is, the following diagram commutes. 
\[
\xymatrix@C=5em{
\calU_{12}\otimes
\Theta^{E^*}(T^{1/2} X_2\otimes \rho_2) \ar[d]^{\text{action}}\ar[r]^{\Delta^{E^*}\otimes \mathcal{S}tab_{X_2}}& 
\calU_{12}^A\otimes \Theta^{E^*}(T^{1/2} X_2^A\otimes \rho_{2, A}\otimes \delta_2) \ar[d]^{\text{action}}
\\
\Theta^{E^*}(T^{1/2} X_1\otimes \rho_1)\ar[r]^{\mathcal{S}tab_{X_1}} &
\Theta^{E^*}(T^{1/2} X_1^A\otimes \rho_{1, A}\otimes \delta_1)
}
\]
% Choose the torus $A$ to be
% \[
% \C^*=
% \left[
% \begin{array}{c|c}
% \id_{w'} & 0 \\
% \hline
% 0 & t \id_{w''}
% \end{array}
% \right]
% \]
% The coproduct gives a map\[
% \Delta^{E^*}: \calU\to 
% \calU'\otimes \calU'':=\calU^{\C^*}. \]

Recall that we denote by $\frac{HKR}{e(N(X_1^{\im \alpha A} \subset  X_1^{A}))}$ the modified map $\sum_\alpha \frac{HKR_\alpha}{e(N(X_1^{\im \alpha A} \subset  X_1^{A}))}$ of $HKR$. Let $\Delta^{MO}$ be the coproduct defined using the Maulik-Okounkov stable classes, that is, 
$\Delta^{MO}:=p_{23*}\Big( p_{12}^* (\Stab^{MO}_\mathfrak{C^{\op}})^{-1}\cdot p_{14}^* (\bullet) \cdot p_{34}^*(\Stab^{MO}_{\mathfrak{C}})\Big)$. 
\begin{lemma}
\label{lem:commu coproduct}
We have $(\frac{HKR}{e(N(X_1^{\im \alpha A} \subset  X_1^{A}))}) \Delta^{E^*}(x)=\Delta^{MO}(\frac{HKR}{e(N(X_2^{\im \alpha } \subset  X_2))}(x))$, for any $x\in \calU_{12}$. 
\end{lemma}
\begin{proof}
We apply Proposition \ref{prop:localization_convolution} with $M_1=X_1^A\times X_2^A$, $M_2:=X_1\times X_2$ and $M_3=\pt$. 
Denote  $L(E^*)\otimes_{E^*} E^*_{\tilde{G}}(-)$ by $L(E^*)_{{\tilde{G}}}(-)$ for short, where $\tilde{G}=G\times \C^*$. When applying the \cite[Theorem C]{HKR} in the following, we treat the extra $\C^*$ the same way as in \S\ref{E^* class}. 

We have the following commutative diagram
% \[
% \xymatrix{
% L(E)^*_{\tilde{G}}(M_1\times
% M_2
% )\otimes E^*_{\tilde{G}}(M_2) \ar[r]^{ \frac{HKR}{e(N(M_2^{\im \alpha}\subset M_2))}\boxtimes HKR}  \ar[d]^{a^{E^*}}&
% L(E)^*_{\tilde{G}}(M_1^{\im \alpha }\times
% M_2^{\im \alpha }
% )\otimes L(E)^*_{\tilde{G}}(M_2^{\im \alpha})
%  \ar[d]^{a^{MO}}\\
% %============
% L(E)^*_{\tilde{G}}(M_1)\ar[r]^{HKR}
%  & 
%  L(E)^*_{\tilde{G}}(M_1^{\im \alpha})
% }
% \]
% Equivalently, we have
\[
\xymatrix@C=5cm{
L(E)^*_{\tilde{G}}(M_1\times
M_2
)\otimes LE^*_{\tilde{G}}(M_2) \ar[r]^{ \frac{HKR}{e(N(X_2^{\im \alpha}\subset X)) e(N(X_1^{\im \alpha A}\subset X_1^A))}\boxtimes \frac{HKR}{e(N(X_2^{\im \alpha}\subset X_2))}}  \ar[d]^{a^{E^*}}&
L(E)^*_{\tilde{G}}(M_1^{\im \alpha }\times
M_2^{\im \alpha }
)\otimes L(E)^*_{\tilde{G}}(M_2^{\im \alpha})
 \ar[d]^{a^{MO}}\\
%============
L(E)^*_{\tilde{G}}(M_1)\ar[r]^{\frac{HKR}{e(N(X_1^{\im \alpha A}\subset X^{A}_1 )) }}
 & 
 L(E)^*_{\tilde{G}}(M_1^{\im \alpha})
}
\]
Under the top horizontal map, we have
\[
\frac{HKR}{e(N(X_2^{\im \alpha}\subset X_2)) e(N(X_1^{\im \alpha A}\subset X_1^A))}
: 
(\Stab^{E^*}_{\mathfrak{C}^{\op}})^{-1}
\boxtimes \Stab^{E^*}_{\mathfrak{C}} \mapsto 
(\Stab^{MO}_{\alpha, \mathfrak{C}^{\op}})^{-1}
\boxtimes \Stab^{MO}_{\alpha, \mathfrak{C}}. 
\]
For any element $x\in E^*_{\tilde{G}}(M_2)$, note that
$\Delta^{E^*} (x)=a^{E^*} ((\Stab^{E^*}_{\mathfrak{C}^{\op}})^{-1}
\boxtimes \Stab^{E^*}_{\mathfrak{C}}\boxtimes x )$ and 
$
\Delta^{MO} (x):=a^{MO} ((\Stab^{MO}_{\mathfrak{C}^{\op}})^{-1}
\boxtimes \Stab^{MO}_{\mathfrak{C}}\boxtimes x )$. 
The commutativity of the above diagram implies that
\begin{align*}
\Delta^{MO}(\frac{HKR}{e(N(X_2^{\im \alpha}\subset X_2))}(x))=&a^{MO}(
(\Stab^{MO}_{\alpha, \mathfrak{C}^{\op}})^{-1}
\boxtimes \Stab^{MO}_{\alpha, \mathfrak{C}}\boxtimes \frac{HKR}{e(N(X_2^{\im \alpha}\subset X_2))}(x))
\\
=&\frac{HKR}{e(N(X_1^{\im \alpha A}\subset X_1^A))}( a^{E^*} ((\Stab^{E^*}_{\mathfrak{C}^{\op}})^{-1}
\boxtimes \Stab^{E^*}_{\mathfrak{C}}\boxtimes x)\\
=&\frac{HKR}{e(N(X_1^{\im \alpha A}\subset X_1^A))}\Delta^{E^*} (x). 
\end{align*}
This completes the proof. 
\end{proof}

\begin{theorem}
\label{thm:coproduct} Notations as above,
\begin{enumerate}
\item
$\Delta^{E^*}$ is coassociative; 
%(This is a analogue of \cite[Proposition 3.22]{Nak12})
\item
$\Delta^{E^*}$ is an algebra homomorphism. 
\end{enumerate}
\end{theorem}
\begin{proof}

The statement (1) follows from the following
torus reduction of the stable classes. Let $A'\subset A$ be a subtorus of the torus $A$, 
we have a triangle of embeddings
\[
\xymatrix{
X^A \ar[rr]  \ar[rd]&& X\\
&X^{A'} \ar[ur]&
}
\]
There are three classes 
$(\Stab^{E^*}_{\mathfrak{C}^{\op}})^{-1}
\boxtimes \Stab^{E^*}_{\mathfrak{C}}\in E^*_{\tilde{G}}(X_1\times X^A_1\times X^A_2\times X_2)$, 
$(\Stab^{E^*}_{\mathfrak{C'}^{\op}})^{-1}
\boxtimes \Stab^{E^*}_{\mathfrak{C'}}\in E^*_{\tilde{G}}(X_1\times X_1^{A'}\times X_2^{A'}\times X_2)$, and
$(\Stab^{E^*}_{\mathfrak{C''}^{\op}})^{-1}
\boxtimes \Stab^{E^*}_{\mathfrak{C''}}\in 
E^*_{\tilde{G}}(X_1^{A'}\times X_1^A\times X_2^A\times X_2^{A'})$. 
Set $N_1:=X_1\times X_2$. We have the following maps.
\[
\xymatrix{
& N_1 \times N_1^{A'} \times N_1^{A} \ar[dl]_{p_{12}}\ar[d]^{p_{13}}\ar[rd]^{p_{23}}&\\
X_1\times X_1^{A'}\times X_2^{A'}\times X_2 & X_1\times X_1^A\times X_2^A\times X_2 &  X_1^{A'}\times X_1^A\times X_2^A\times X_2^{A'}
}
\]
To show the coassociativity of $\Delta^{E^*}$, it suffices to show the following equality
\begin{equation}\label{eq:stabE}
(\Stab^{E^*}_{\mathfrak{C}^{\op}})^{-1}
\boxtimes \Stab^{E^*}_{\mathfrak{C}}
=((\Stab^{E^*}_{\mathfrak{C'}^{\op}})^{-1}
\boxtimes \Stab^{E^*}_{\mathfrak{C'}})*
((\Stab^{E^*}_{\mathfrak{C''}^{\op}})^{-1}
\boxtimes \Stab^{E^*}_{\mathfrak{C''}})
, 
\end{equation}
where $\mathfrak{C'}\subset \mathfrak{C}$ is a face such that $\mathfrak{a}'=\Lie A'=\text{Span}(\mathfrak{C'})$.  The cone $\mathfrak{C}$ projects to a cone in $\mathfrak{a}/\mathfrak{a}'$ which we denote by $\mathfrak{C}''$.

We now verify the equation \eqref{eq:stabE}. 
Consider the following embeddings of fixed point subvarieties
\[
\xymatrix@C=0.5em {
& N_1 \times N_1^{A'} \times N_1^A \ar[dl]_{p_{12}}\ar[d]^{p_{13}}\ar[rd]^{p_{23}}& 
&\ar@{_{(}->}[l]_{\iota} & N_1^{\im \alpha} \times N_1^{\im \alpha A'} \times N_1^{im \alpha A} 
\ar[dl]_{p_{12}^{\alpha}}\ar[d]^{p_{13}^{\alpha}}\ar[rd]^{p_{23}^{\alpha}} \\
N_1\times N_1^{A'} & N_1\times N_1^A &  N_1^{A'} \times  N_1^{A}&
N_1^{\im \alpha}\times N_1^{\im \alpha, A'} \ar@/^1.5pc/[lll]^{\iota_{12}}& N_1^{\im \alpha}\times N_1^{\im \alpha A} \ar@/^2pc/[lll]^{\iota_{13}}&  N_1^{im\alpha, A'} \times  N_1^{im \alpha, A}\ar@/^1.5pc/[lll]^{\iota_{23}}
}
\]
By Proposition \ref{prop:localization_convolution}, we know $\frac{HKR}{e(N_2)}$ commutes with convolution. That is, 
\begin{align*}
&\frac{HKR}{e(N(X_2^{\im \alpha}\subset X_2)) e(N(X_1^{\im \alpha A}\subset X_1^A))}( ((\Stab^{E^*}_{\mathfrak{C'}^{\op}})^{-1}
\boxtimes \Stab^{E^*}_{\mathfrak{C'}})*
((\Stab^{E^*}_{\mathfrak{C''}^{\op}})^{-1}
\boxtimes \Stab^{E^*}_{\mathfrak{C''}}) )\\
=&\frac{HKR}{e(N(X_2^{\im \alpha}\subset X_2)) e(N(X_1^{\im \alpha A'}\subset X_1^{A'}))}
(\Stab^{E^*}_{\mathfrak{C'}^{\op}})^{-1}
\boxtimes \Stab^{E^*}_{\mathfrak{C'}}) \\&
*\frac{HKR}{e(N(X_2^{\im \alpha A'}\subset X_2^{A'})) e(N(X_1^{\im \alpha A}\subset X_1^A))}(\Stab^{E^*}_{\mathfrak{C''}^{\op}})^{-1}
\boxtimes \Stab^{E^*}_{\mathfrak{C''}})\\
=& (\Stab^{MO}_{\alpha, \mathfrak{C'}^{\op}})^{-1}
\boxtimes \Stab^{MO}_{\alpha, \mathfrak{C'}})*
((\Stab^{MO}_{\alpha, \mathfrak{C''}^{\op}})^{-1}
\boxtimes \Stab^{MO}_{\alpha, \mathfrak{C''}}). 
\end{align*}
It is known that \cite[Lemma 3.6.1]{MO}
\[
(\Stab^{MO}_{\alpha, \mathfrak{C}^{\op}})^{-1}
\boxtimes \Stab^{MO}_{\alpha, \mathfrak{C}}
=((\Stab^{MO}_{\alpha, \mathfrak{C'}^{\op}})^{-1}
\boxtimes \Stab^{MO}_{\alpha, \mathfrak{C'}})*
((\Stab^{MO}_{\alpha, \mathfrak{C''}^{\op}})^{-1}
\boxtimes \Stab^{MO}_{\alpha, \mathfrak{C''}}), \,\ \text{for each $\alpha$.}
\]
and also 
\[
\frac{HKR}{e(N(X_2^{\im \alpha}\subset X_2)) e(N(X_1^{\im \alpha A}\subset X_1^A))}: (\Stab^{E^*}_{\mathfrak{C}^{\op}})^{-1}
\boxtimes \Stab^{E^*}_{\mathfrak{C}}
\mapsto 
(\Stab^{MO}_{\alpha, \mathfrak{C}^{\op}})^{-1}
\boxtimes \Stab^{MO}_{\alpha, \mathfrak{C}}
\]
This implies \eqref{eq:stabE} and hence the coassocitivity.

The statement (2) follows from (1) and Lemma \ref{lem:commu coproduct}. 
By  \cite[Theorem 3.20]{Nak12}, $\Delta^{MO}$ is an algebra homomorphism, that is, $\Delta^{MO}(a*b)=\Delta^{MO}(a)*\Delta^{MO}(b)$. 
For simplicity, we write $\chi:=\frac{HKR}{e(N_2)}$ for short. 
We have the following isomorphisms
\begin{align*}
&\chi\Delta^{E^*}(a*b)=
\Delta^{MO}(\chi(a*b))=
\Delta^{MO}(\chi(a)*\chi(b))=
\Delta^{MO}(\chi(a))*\Delta(\chi(b))\\
=&
\chi \Delta^{E^*} (a)* \chi \Delta^{E^*} (b)
=\chi(\Delta^{E^*} (a)* \Delta^{E^*} (b) ). 
\end{align*}
The injectivity of $\chi$ implies that
\[
\Delta^{E^*}(a*b)=\Delta^{E^*}(a) * \Delta^{E^*}(b). 
\]
\end{proof}
We remark that any convolution algebra has an anti-involution coming from $X\times X\to X\times X$ by swamping the two copies. The 
anti-involution induces a ``contragradient duality", which is compatible with the tensor structure in the usual sense.

\newcommand{\arxiv}[1]
{\texttt{\href{http://arxiv.org/abs/#1}{arXiv:#1}}}
\newcommand{\doi}[1]
{\texttt{\href{http://dx.doi.org/#1}{doi:#1}}}
\renewcommand{\MR}[1]
{\href{http://www.ams.org/mathscinet-getitem?mr=#1}{MR#1}}

\end{document}